\documentclass[11pt]{amsart}
\usepackage{mathpazo}
\usepackage{microtype}
\usepackage{tabularx}
\usepackage[utf8]{inputenc}
\usepackage{hyperref}
\usepackage{xcolor}
\usepackage{amsmath,amsthm,amssymb,mathtools,bbm}
\usepackage[foot]{amsaddr}
\usepackage{graphicx}
\usepackage{upgreek}
\usepackage[normalem]{ulem} 
\usepackage{mathrsfs}
\usepackage{ytableau} 
\usepackage{tikz}
\usepackage{braids}
\usetikzlibrary{arrows}
\usetikzlibrary{braids}
\usepackage{float}
\usepackage[all,2cell]{xy}
\UseAllTwocells 
\usepackage{blkarray}
\newcommand{\ppm}[1]{\textcolor{red}{#1}}
\newcommand{\ppmx}[1]{}
\newcommand{\ppmm}[1]{#1}
\newcommand{\ecr}[1]{\textcolor{orange}{#1}}
\newcommand{\ecrb}[1]{\textcolor{blue}{#1}}
\newcommand{\ft}[1]{\textcolor{purple}{#1}}
\newcommand{\ftx}[1]{}

\newcommand{\PI}[1]{\Pi^{N}}

\newcommand{\Match}{{\mathsf{Match}}}
\newcommand{\Matcha}{{\mathsf{Matcha}}}
\newcommand{\Matchg}{{\mathsf{Matchg}}}
\newcommand{\PermMat}{{\mathsf{PermMat}}}
\newcommand{\MonMat}{{\mathsf{MonMat}}}
\newcommand{\Invol}{{\mathsf{Invol}}}
\newcommand{\Vect}{{\mathsf{Vect}}}
\newcommand{\MonFun}{{\mathsf{MonFun}}}
\newcommand{\MnFun}{{\mathsf{MnFun}}}
\newcommand{\Fun}{{\mathsf{Fun}}}

\newcommand{\GTMat}{{\mathsf{GTMat}}}
\newcommand{\UMat}{{\mathsf{UMat}}}
\newcommand{\OMat}{{\mathsf{OMat}}}
\newcommand{\id}{\mathrm{id}}
\newcommand{\Id}{\mathrm{Id}}
\newcommand{\End}{\mathrm{End}}
\newcommand{\Aut}{\mathrm{Aut}}
\newcommand{\HomYB}{\Hom_{YB}}
\newcommand{\EndYB}{\End_{YB}}
\newcommand{\AutYB}{\Aut_{YB}}

\newcommand{\Bcat}{\mathsf{B}}
\newcommand{\Mat}{\mathsf{Mat}} 

\newcommand{\Sym}{\Sigma} 
\newcommand{\F}{\mathsf{F}}  
\newcommand{\X}{\mathsf{Y}} %
\newcommand{\FF}{\mathsf{F}}  
\newcommand{\N}{\mathbb{N}}   
\newcommand{\unit}{\textbf{1}}
\newcommand{\One}{\unit}
\newcommand{\ul}[1]{\underline{#1}}

\newcommand{\ignore}[1]{}

\newtheorem{theorem}{Theorem}[section]
\newtheorem{proposition}[theorem]{Proposition}
\newtheorem{lemma}[theorem]{Lemma}
\newtheorem{corollary}[theorem]{Corollary}
\newtheorem{conjecture}[theorem]{Conjecture}

\theoremstyle{definition}
\newtheorem{defin}[theorem]{Definition}
\newtheorem{example}[theorem]{Example}

\newtheorem{remark}[theorem]{Remark}
\newtheorem{question}[theorem]{Question}

\newtheorem{problems}[theorem]{Strategy}

\newtheorem{para}[theorem]{}

\newcommand{\beq}{\begin{equation}}
\newcommand{\eq}{\end{equation}} 

\newcommand{\Ob}{\mathrm{Ob}}

\newcommand{\C}{{\mathbb C}}

\newcommand{\Cc}{\mathcal{C}}
\newcommand{\Dd}{\mathcal{D}}
\newcommand{\cO}{\mathcal{O}}

\newcommand{\mat}{\begin{bmatrix}}
\newcommand{\tam}{\end{bmatrix}} 
\newcommand{\matt}[1]{\left[ \begin{array}{#1}}
\newcommand{\ttam}{\end{array} \right]}
\newcommand{\Hom}{\mathrm{Hom}}

\newcommand{\bra}[1]{\langle #1 |} 
\newcommand{\ket}[1]{| #1 \rangle}

\newcommand{\aaa}{{\mathsf{a}}}

\newcommand{\two}{\square\!\square}

\newcommand{\oneone}{\!\!\begin{array}{c} \square \vspace{-.081in}\\ 
\vspace{-.1in}\square \vspace{.08in} \end{array}\!\!}

\newcommand{\oneonex}{\!\!\begin{array}{c} \square \vspace{-.081in}\\ 
\vspace{-.1in}\blacksquare \vspace{.08in} \end{array}\!\!}
\newcommand{\smat}{\left(\begin{smallmatrix}}
\newcommand{\stam}{\end{smallmatrix}\right)}

\newcommand{\otimesb}{\otimes}

\newcommand{\footnotex}[1]{}  

\newcounter{minidef}[section]

\newcommand{\mdef}{\refstepcounter{theorem} 
\medskip \noindent ({\bf \thetheorem}) }
\newcounter{minicapt}

\newcommand{\ob}{Obj}  
 
\newcommand{\catC}{{\mathcal C}}
\newcommand{\catT}{{\mathcal T}}

\newcommand{\sss}{\mathsf{s}}  
\newcommand{\ttt}{\mathsf{t}}

\newcommand{\Matop}{{\mathsf{M\overline{a}t}}}     
\newcommand{\optimes}{\overline{\otimes}}
\newcommand{\Mmm}{\mathcal{M}}  
\newcommand{\Farr}{\overline{\F}}   
\newcommand{\Fff}{\mathcal{T}}   
\newcommand{\PP}{\mathtt{P}}    
\newcommand{\Jarmo}{\mathfrak{J}}  

\newcommand{\Bb}{{\mathcal{B}}}    
\newcommand{\CDd}{{\mathcal{C}}}   
\newcommand{\otimess}[1]{\stackrel{#1}{\otimes}}

\newcommand{\YangBaxterobject}{Yang-Baxter object}
\newcommand{\YBo}{\YangBaxterobject} 
\newcommand{\YB}{YB}   
\newcommand{\MoonFun}{{\mathsf{MoonFun}}}  
\newcommand{\mop}{\otimes\circ}    
\newcommand{\Bcatmop}{\Bcat^{\mop}}  

\newcommand{\YBone}{YB^1}

\newcommand{\Fio}{\F_{\iota}}  
\newcommand{\Feta}{\F_{\zeta}}  
\newcommand{\Fox}{\F_{\mathsf{ox}}}  

\title{
A Categorical Perspective on 
Braid Representations
}
 \author{Paul P. Martin $^1$}  
 \address{$^1$ School of Mathematics, University of Leeds}
\address{$^2$ Department of Mathematics, Texas A\&M University}
\address{$^3$ School of Mathematics, University of Bristol}
\address{$^4$ Heilbronn Institute of Mathematics}
\author{Eric C. Rowell $^{1,2}$}
\author{Fiona Torzewska $^{3,4}$}
\date{\today}
\thanks{The work of PPM was partially supported by the EPSRC Programme Grant 
EP/W007509/1. 
The research of ECR was partially supported by a Royal Society Wolfson Visiting Fellowship and US NSF grant DMS-2205962. The authors thank C\'esar Galindo and Ilaria Colazzo for valuable discussions.}
\begin{document}

\begin{abstract} 
    We study categories 
   whose objects are the {\em braid representations}, that is, 
strict monoidal functors 
$F : \Bcat \rightarrow \Mat$    
from the braid category $\Bcat$ to the category of matrices $\Mat$. 
{Braid representations are equivalent to} 
solutions to the (constant) Yang--Baxter equation, so our work provides a categorical framework for their classification.

Our first technical result (Theorem \ref{th:strictifyingmonfuns}) implies that the assumption of strictness causes no loss of generality: \emph{any} monoidal functor with source B is monoidally equivalent to a strict such functor.

We develop the structure of the category
$\MonFun(\Bcat,\Mat)$, whose objects are braid representations and whose morphisms are monoidal natural transformations.  Although $\MonFun(\Bcat,\Mat)$ is non-additive, we show that it does admit a rigid monoidal structure
(Theorem \ref{th:YBMonoidal}).  We introduce notions of quotient, sub- and simple objects in this setting, and prove an array of structural results such as a version of Schur's Lemma (Theorem \ref{thm:schur}), and a characterisation of objects that are both sub- and quotient objects (Corollary \ref{thm:endoissub/quo}).  This structural framework, for example, provides a new categorical characterisation of charge-conserving solutions.

Classification is, generally,  
up to a suitable notion of isomorphism. 
So a major 
part of the
contribution here is to 
introduce, compare, and contrast notions of isomorphism
{for braid representations}. 
We 
give various properties exposing the implications of different choices of equivalence and describe some relationships among them (Theorems \ref{thm: outer auto match}, \ref{thm:flipliftstoequiv}, \ref{th:ds is infinity}, Conjectures \ref{conj: Matcha outer},  \ref{conj: X is natural}).

An extensive range of key examples and counterexamples illustrate the framework developed here; and we cast a number of partial classifications in categorical terms, providing   
proof-of-principle vindication for our methods.
\end{abstract}

\maketitle
\newpage 
\tableofcontents

\newpage 

\section*{Glossary}

Here we use the following notations (explained further later in the paper):
\\
If $\catC$ is a category then $\catC^\circ$ is the opposite category. 
\\
If $\catC$ is a strict monoidal category then $\catC^{\otimes\circ}$ is the {\em monoidal} opposite category. 
\\
If $\catC$ is a category and $k$ a commutative ring, then $k\catC$ is the $k$-linearisation of $\catC$. 

\medskip \medskip

\newcommand{\glosss}[2]{$#1$ & #2  \\ }  

\begin{tabular}{l|l}
\glosss{\Sym}{strict monoidal category of symmetric groups (see \ref{de:Sym})} 
\glosss{\Bcat}{  strict monoidal category of braids {(see \ref{de:Bcat})}}  
\glosss{\Bcatmop}{category $\Bcat$ but with opposite monoidal product}
$\Mat$      &  strict monoidal category of matrices over given ring \\ & \hspace{.3cm} with Ab convention Kronecker product \\
\glosss{\Matop}{category $\Mat$ but with aB Kronecker product}
\glosss{\MonMat}{category of monomial matrices m{(see \ref{def:monomial/perm})}}
\glosss{\PermMat}{category of permutation matrices {(see \ref{def:monomial/perm})}}
\glosss{\GTMat}{set of group-type matrices (see \ref{def:group-type})}
\glosss{\Invol}{set of involutive matrices (see Def. \ref{de:invol})}
$\Mat^N$     & full subcategory of $\Mat$ monoidally generated by object $N\in \N$ (see \ref{de:MatN}) \\
\glosss{\Matcha^N}{
subcategory of $\Mat^N$ of additive charge conserving matrices (see \ref{de:Matcha})}
\glosss{\Match^N}{ 
subcategory of $\Matcha^N$ {of charge conserving matrices (see \ref{de:Match})}} 
\glosss{\Matchg^N}{subcategory of $\Mat^N$ of cc-with-glue matrices (see \ref{de:ccwg})}
\glosss{\Vect}{category of vector spaces over a given field} 
& \\
\glosss{\MnFun(A,B)}{category of monoidal functors  $F:A\rightarrow B$ and natural transformations (\S\ref{S:intro})}
\glosss{\MoonFun(A,B)}{  category of strict monoidal functors 
$F:A\rightarrow B$ and natural transformations (\S\ref{S:intro})}

\glosss{\MoonFun^N(\Bcat,\catC)}{subcat of $\MoonFun(\Bcat,\catC)$ of functors with $F(1)=N$} 
$\MonFun(A,B)$   &  subcategory of $\MoonFun(A,B)$ of monoidal natural transformations \\ 
\glosss{YB(\catC)}{synonym for $\MonFun(\Bcat,\catC)$}
\glosss{YB^a(\X)}{full subcategory of all $(a,R)\in YB(\X)$ for $\X\subset\Mat^N$}
& \\
\glosss{\boxtimes}{lashing product (see \ref{de:lashing})}
$\Fio :\Bcat\rightarrow \Bcat$  &  strict monoidal functor (SMF) flipping braid over/under conventions \\ 
\glosss{\Fox:\Bcat\rightarrow \Bcatmop}{SMF flipping braids laterally (see \ref{pa:Bflipp})}
$\Feta :\Bcat\rightarrow \Bcat^\circ$  &  SMF flipping braids 
vertically (see \ref{pa:Bopp}) \\  
\glosss{\Mmm_n}{$n$-strand ribbon half-twist braid (see \ref{pa:Bflipp})}
\glosss{\Fff_n}{for fixed $N$, the level-$n$ `flip' matrix (see \ref{de:flip})}
\glosss{\PP}{for fixed $N$, $\PP=\Fff_2$}
\glosss{\overline{F}:\Mat\rightarrow\Matop}{SMF swapping Kronecker convention (see \ref{de:Farr} and \cite{MR1X})}
\glosss{}{} 
\end{tabular}

\medskip

\section{Introduction}\label{S:intro}

{There has been interest}
in braids and their relevance to 
understanding 
the physical universe for millennia (see e.g. \cite{Olivelle98,Epple98} for examples from the 7th century B.C.E. philosopher G\=arg\={\i} V\=achaknav\={\i} and the 19th century C.E. mathematician Carl Friedrich Gauss). Mathematically, braids on $n$ strands form a group $B_n$.
{It is known 
\cite{Artin}
that this group is isomorphic to a group} defined abstractly via generators $\sigma_1,\ldots,\sigma_{n-1}$ satisfying relations
\begin{enumerate}
    \item[B1] $\sigma_i\sigma_{i+1}\sigma_i=\sigma_{i+1}\sigma_i\sigma_{i+1}$ for $1\leq i\leq n-2$
    \item[B2] $\sigma_{i}\sigma_j=\sigma_{j}\sigma_i$ for $|i-j|\neq1$.
\end{enumerate}  
(the isomorphism takes one of the elementary exchanges of the first two strands to $\sigma_1$, and so on). 
An important manifestation of braids in 
{computational}  
physics is through the constant Yang-Baxter equation: \beq   
\label{eq:YBE0}
(R \otimes I) (I\otimes R )(R \otimes I) 
= (I\otimes R ) (R \otimes I) (I\otimes R )
\eq 
where $R$ is an invertible operator on $V^{\otimes 2}$ for some finite dimensional complex vector space $V$, and $I=\Id_V$ 
{(and we have passed to the usual strictification $V^{\otimes 3}$)}.  
In practice $V=\C^N$ for some $N$, so we can fix the standard ordered basis, 
and the tensor product, 
and then $R$ is a matrix. 
Originally, such Yang-Baxter operators $R$ were used in statistical mechanics to describe the interaction of particles in 1+1d physical systems
{with computationally favourable scattering properties}
(see e.g. \cite{LiebMattis,Baxter82}).  
Moreover, given such an $R$ one obtains, for each $n\in\N$, a representation $\rho_n^R$ of the braid group $B_n$ on $n$ strands, 
{on $V^{\otimes n}$,}
by defining 
\beq 
\label{eq:rhorep}
\rho_n^R(\sigma_i)=I_V^{\otimes i-1}\otimes R\otimes I_V^{\otimes n-i-1}  .
\eq 
There is a uniform way to close braids to describe links, and 
{a direct connection to low-dimensional topology was made in \cite{Turaev1988}, showing how to obtain  link invariants from a given Yang-Baxter operator.}

Over time, 
{aspects of} category theory have been embraced by physicists 
- see e.g. \cite{Martin91,Nayaketal,gainutdinov2016fusionbraidingfiniteaffine,ZiniWang2018}.\footnote{Although somewhat unenthusiastically by some at first, as Moore and Seiberg \cite{MooreSeiberg89} describe category theory as ``an esoteric subject noted for its difficulty and irrelevance." }  One key reason that category theory is an appropriate framework for 
physical modelling 
was observed by Kapranov and Voevodsky \cite{KapranovVoevodsky}: \emph{In any category, it is unnatural and
undesirable to speak about equality
of two objects}--just as different particles are never equal, but can 
{usefully}
be regarded as equivalent 
{if having equal}
responses to  {certain} measurements. 
More specifically, for example, Morita equivalence in category theory can be seen as a sister to the phenomenon of the thermodynamic limit (equivalencing different large systems) in computational statistical mechanics. 

Computations in quantum physics are typically linear algebraic, so that the category $\Mat$ of matrices is of particular utility. This is a category whose objects are natural numbers $n\geq 1$ and the morphisms from $m$ to $n$ 
can be taken to be  
$n\times m$ matrices. 
\ppmm{More fundamentally, physical laws are observed (and hence in modelling assumed) to be translation invariant. In category-theoretic terms this corresponds to a monoidal structure.}
The category $\Mat$ has a natural strict monoidal structure: on objects this is  multiplication and on morphisms it is the Kronecker product. 
{This 
particular structure 
corresponds to independent event probabilities being composed multiplicatively
- another tenet of physical modelling.}
Thus \ppmm{these strict monoidal} categories provide a structure in which to `do' physics. 
\ppmm{Independence of reference frame in physics allows for a role for non-strict monoidal categories, 
and there is a frame-independent version of $\Mat$ known as $\Vect$ (which we will return to at the end).}

Incorporating braids into this framework is facilitated by the braid category $\Bcat$  \cite{MacLane}.  This is the category whose objects are natural numbers $k\geq 0$ with morphisms from $k$ to $k$ consisting of the braid group $B_k$ on $k$ strands, and no morphisms between distinct $k_1\neq k_2$.  This category is strictly monoidal, with tensor product on objects given by $+$ and on morphisms by {a chosen} juxtaposition of braids.

{Observe in particular  that the category $\Bcat$ is monoidally generated by the object 1 and an elementary braid  $\sigma \in \Bcat(2,2)$ 
(and its inverse).
Let $\catC$ be a strict monoidal category. It follows that a strict  monoidal functor $F:\Bcat \rightarrow \catC$ is completely determined by the value of $F(1) $ in $Ob(\catC)$, $F(1)=c$ say, and invertible $F(\sigma)$ in $\catC(c\otimes c,c \otimes c)$.
Conversely a necessary condition for  a formal assignment of an $F(1)$ and $F(\sigma)$ to yield a {strict} monoidal functor is that $F(\sigma)$ obeys the corresponding form of (\ref{eq:YBE0}). 
Indeed it follows from the Artin presentations of the braid groups {given in B1, B2 above} that this condition is also sufficient.
}

For any $N^2\times N^2$ matrix solution $R\in \Mat(N^2,N^2)$ to the Yang-Baxter equation \eqref{eq:YBE0}
we obtain a functor $F_R:\Bcat\rightarrow \Mat$ 
by setting $F_R(1)= N$
and setting $F_R(\sigma)=R$.
Of course it follows that  
 $F_R(k)=N^k$
and  $F_R(\beta)=\rho_n^R(\beta)$ for $\beta\in B_n$.
{From (\ref{eq:rhorep})}
the functor $F_R$ has an additional feature: it is a (strict) monoidal functor%
--\emph{i.e.} we have $F_R(j)\cdot F_R(k)= F_R(j+k)$. 
{Such a functor $F: \Bcat\rightarrow\Mat$ is called a {\em braid representation}. So classifying matrix solutions to the Yang--Baxter equation is equivalent to classifying braid representations.}

 Classifying matrix solutions to the Yang-Baxter equation (\ref{eq:YBE0}) is a problem noted for being easily formulated yet impervious to direct methods of solution.  
 And recasting solutions as 
 strict 
 monoidal functors  
 does not, of itself, make the problem any easier. 
 (Indeed one may then be drawn to relaxing the strictness condition, in principle making the problem harder.) 
 {Without some controlling notion of equivalence classes, classification would amount to giving a construction for every $R$-matrix --- certainly an overwhelming amount of information in general.}
 A complete enumeration of
 the set of 
 solutions up to some dimension without further context would 
 {therefore}
 not likely be of widespread interest--a more useful result could incorporate  natural equivalences and symmetries, 
 \ppmm{and any topology on this set,}
 and then might further stratify the solutions into natural families. 
 An application of the Polya principle, informed by the discussion above, suggests that we 
 should 
 generalise the problem by studying the 
 {\em category} 
 $\MnFun(\Bcat,\Mat)$ 
 of monoidal functors from $\Bcat$ to $\Mat$ {and natural transformations between them}.
 The additional structures afforded by taking this perspective may yield such stratifications while providing a useful context in which to study the original problem.  We also see potential for generalising this approach to functors from $\Bcat$ to less familiar targets, with the goal of understanding $\Bcat$ more deeply by viewing it through a diversity of lenses. 
 
 An immediate reduction is to only consider the subcategory of \emph{strict} monoidal functors, which we will denote by $\MoonFun(\Bcat,\Mat)$.  
 This is justified,
 \ppmm{even with Polya looking over our shoulder,}
by the fact (proved in Theorem \ref{th:strictifyingmonfuns}) that 
{every} 
monoidal functor from $\Bcat$ to $\Mat$ is naturally isomorphic
(even 
{monoidally} naturally isomorphic) to a strict one. 
To specify an $F\in\MoonFun(\Bcat,\Mat)$ is the same thing as providing a matrix solution to the Yang-Baxter equation--so  
{this is, up to a choice of basis, the problem we started with}.
{Except that now of course it is cast as (higher) representation theory.}

It is 
commonplace (see e.g. \cite{Vendraminetal22,Hietarinta1992} and cf. e.g. \cite{MR1X}) 
{and (categorically meaningful)}  
to exclude 
{\em morphisms} that do not preserve monoidal structure, corresponding to the subcategory $\MonFun(\Bcat,\Mat)$ of $\MoonFun(\Bcat,\Mat)$ with \emph{monoidal} natural transformations.  This restriction is structurally and computationally 
simplifying of certain aspects of the problem 
(\ppmm{see, e.g.,} {Theorem \ref{thm: outer auto match}, and Examples \ref{ex: gaussian factorisation} and \ref{ex:rank 3 ST}} below),  
and we develop much of our theory in this setting.

\newcommand{\sproblem}{strategy}
\newcommand{\sproblems}{strategies}

Casting
the problem in this categorical framework leads naturally to 
strategies such as 
 restricting the targets of  functors $F:\Bcat\rightarrow \Mat$  to 
 special families of matrices.  
 For example, we may consider monoidal subcategories of $\Mat$, as in Figure \ref{fig:catinclusions} (see Section \ref{ss:subcats} for notation). Indeed, the main inspiration for this article is the recent work \cite{MR1X} classifying $N^2\times N^2$ charge conserving solutions to the Yang-Baxter for all $N$ using this categorical perspective (i.e. functors $F\colon \Bcat \to \Match^N$ with $F(1)=1$). 
 
\ppmm{
Restricting to a subcategory $\X$ 
(or indeed any form of restriction) 
raises the possibility  
that the braid representations classified 
are not a transversal 
(up to the appropriate equivalences for the original category)
of {\em all} braid representations. 
So with every $\X$ there comes the question of transversality. 
}
This motivates the  
\sproblem:
 \begin{problems}\label{prob:1}
     Fix a subcategory $\X$ of $\Mat$. Classify functors in $\MonFun(\Bcat,\X)$, up to appropriate equivalences;
     \ppmm{and determine the \ppmm{extent of} 
     transversality of this subset in $\ob(\MonFun(\Bcat,\Mat) )$}.
 \end{problems}
Here  the `appropriate' equivalences must be specified, 
 incorporating the symmetries of the category $\X$.  
This aspect is treated generally in \S\ref{s:equivandauto}; 
\ppmm{and several existing illustrative examples are reviewed in \S\ref{s:classifications}.}

\begin{figure}
    \centering
  
\begin{tikzpicture}[every node/.style={midway}]
  \matrix[column sep={6em,between origins}, row sep={2em}] at (0,0) {

    \node(A) {$\PermMat$}  ; & \node(B) {$\MonMat$}; & \node(C) {$\Mat$};  \\

    \node(D) {$\PermMat^N$}  ; & \node (E) {$\MonMat^N$}; & \node(F) {$\Mat^N$}; \\

      \node(G) {}  ; & \node (H) {$\Matchg^N$}; & \node(I) {$\Matcha^N$}; \\ 
       \node(K) {}  ; & \node (L) {$\Match^N$}; & \node(M) {}; \\ 
  };
  \draw[right hook->] (D) -- (A) node[anchor=south]{};

  \draw[right hook->] (E) -- (B) node[anchor=south]{};

  \draw[right hook->] (F) -- (C) node[anchor=south]{};

  \draw[right hook->] (A) -- (B) node[anchor=west]{};
    \draw[right hook->] (B) -- (C) node[anchor=west]{};
      \draw[right hook->] (D) -- (E) node[anchor=west]{};
  \draw[right hook->] (E) -- (F) node[anchor=west]{};
    \draw[right hook->] (I) -- (F) node[anchor=south]{};
     \draw[right hook->] (L) -- (I) node[anchor=south]{};
     \draw[right hook->] (H) -- (F) node[anchor=south]{};
     \draw[right hook->] (L) -- (H) node[anchor=south]{};
\end{tikzpicture}
    \caption{
    A lattice of subcategories $\catC$ of $\Mat$ 
    (thus inducing a lattice of functor subcategories $\MoonFun(\Bcat,\catC)$ of $\MoonFun(\Bcat,\Mat)$).}
    \label{fig:catinclusions}
\end{figure}

The category $\MonFun(\Bcat,\Mat)$ 
can be endowed with 
additional categorical structure that facilitates a more general perspective 
of the \sproblem~\ref{prob:1}.  
Firstly, we will see that there is a natural monoidal structure on $\MonFun(\Bcat,\Mat)$ itself. There is also a natural notion of subobjects and quotient objects in $\MonFun(\Bcat,\Mat)$.  Thus, one could start with any 
{object, i.e. any}
monoidal functor, $F\in\MonFun(\Bcat,\Mat)$ and consider the monoidal subcategory generated by $F$ and its subobjects.  {This motivates generalising the method 
to consider:
\begin{problems}\label{prob:2}
     Find subcategories of $\MonFun(\Bcat,\Mat)$ amenable to classification; \ppmm{classify them and determine their transversality in $\ob(\MonFun(\Bcat,\Mat) )$}.
\end{problems}
In particular, we can study subcategories associated with restricting the 
`target' to a subset, rather than a subcategory of $\Mat$. 
\ppmm{Here we {usually} understand the target to mean a possible matrix type specifically for $R =F(\sigma)$.} 
For example, functors associated with \emph{group-type} solutions \cite{AS} to the Yang-Baxter equation provide such a subcategory: group-type matrices \ref{def:group-type} 
are not closed under composition and hence do not {themselves} form a subcategory of $\Mat$, yet functors with group-type targets do form a
{monoidal subcategory}
 {of   $\MonFun(\Bcat,\Mat)$}. 
Similarly involutive (i.e., $R^2=\Id$) matrices do not form a subcategory of $\Mat$, 
but monoidal functors with involutive targets form a (monoidal) subcategory of $\MonFun(\Bcat,\Mat)$ and  have been classified 
up to (non-monoidal) natural isomorphism 
\cite{LechnerPennigWood}.

\newcommand{\EF}{{\mathfrak{E}}}

Note that there is a natural functor 
$\EF: \Bcat \rightarrow \Sym$ 
that takes 
$g_i \in B_n$ to the underlying permutation in $\Sym_n$. 
This means that every 
strict monoidal 
functor $F:\Sym \rightarrow \Mat$ is also a braid representation $F:\Bcat \rightarrow \Mat$, and is furthermore one with $R=F(\sigma)$ obeying $R^2 = \Id$ 
as above. 
This is one among several properties of $\Bcat$ itself which inform the classification programme, as we will review in \S\ref{s:equivandauto}. 

\medskip 

The main goal of this article is to lay the groundwork for applying categorical techniques to the classification problem.  Briefly, this goal has the following main prongs:  1) a categorical justification for restricting to strict monoidal functors 2) a discussion of restricted targets, i.e. subcategories of $\MonFun(\Bcat,\Mat)$ 3) development of the fundamental structure and properties of $\MonFun(\Bcat,\Mat)$ and its subcategories 4) a discussion of equivalences and symmetries.

\section{Categorical Preliminaries and Strictification}

In this section we first recall some basic notions of category theory.
This is mainly to establish notation, 
but we also present the key    
result Theorem \ref{th:strictifyingmonfuns} showing that we may replace all monoidal functors between two strict categories by strict monoidal functors, provided that the source object monoid is free on finitely many generators 
(such as in $\Bcat$ and $\Mat^N$
- as defined in the Introduction, see also \ref{de:Bcat} and \ref{de:MatN}, but not $\Mat$). 
We then discuss some basic notions 
and conventions 
related to the categories we are initially interested in, namely the braid category $\Bcat$ and the category of matrices $\Mat$.

If $\catC$ is a category then we write 
$\ob(\catC)$ for the class of objects. However we may simply write an object
$X \in \catC$ for an object in $\catC$. For $X,Y\in\ob(\catC)$ the set of morphisms from $X$ to $Y$ is denoted 
 $\Hom_{\catC}(X,Y)$, or $\catC(X,Y)$,
 or perhaps 
$\Hom(X,Y)$ if $\catC$ is fixed. 
We define
\beq  \label{eq:deEnd} 
\End_C(X) = \Hom_C(X,X)
\eq 
and $\Aut_C(X) \subset \End_C(X)$ as the subgroup of self-isomorphisms. 

\mdef  \label{de:sssttt}
Quite generally, given a category $\catC$, we define 
functions 
$\sss,\ttt : \Hom_{\catC}(-,-) \rightarrow \Ob(\catC)$ such that $\sss(M)$ is the source object and 
$\ttt(M)$ the target object. 
See \ref{de:Farr} for an example.

\begin{defin}\label{def:natural transf}
Let $F,G:\catC\rightarrow\Dd$ be functors.
A {\em natural transformation} $\eta:F\rightarrow G$ 
is  
the following: for each $X\in\ob(\catC)$ we have a morphism $\eta_X\in\Hom_{\Dd}(F(X),G(X))$
 such that for any $\catC$-morphism $f:X\rightarrow Y$ we have $F(f)\eta_X=\eta_Y G(f)$. 
 
 The natural transformation $\eta$ is a natural {\em isomorphism} if, in addition, all the $\eta_X$'s are isomorphisms. 
 \end{defin}
 
\begin{defin}
A \emph{monoidal category} \((\catC,\otimes,\unit,\alpha,\rho,\lambda)\) consists of:
\begin{itemize}
  \item a category \(\mathcal{C}\),
  \item a bifunctor \(\otimes : \mathcal{C} \times \mathcal{C} \to \mathcal{C}\),
  \item a unit object \(\unit \in \mathcal{C}\),
  \item natural isomorphisms (for all objects \(X,Y,Z \in \mathcal{C}\)):
  \[
    \alpha_{X,Y,Z} : (X \otimes Y) \otimes Z \xrightarrow{\sim} X \otimes (Y \otimes Z) \quad\text{(associator),}
  \]
  \[
    \lambda_X : \unit \otimes X \xrightarrow{\sim} X, \quad \rho_X : X \otimes \unit \xrightarrow{\sim} X \quad \text{(unitors),}
  \]
\end{itemize}
satisfying the pentagon and triangle coherence axioms
\ppmm{(see e.g. \cite{EGNO,kassel-})}.
    A monoidal category $\Cc$ is \emph{strict} if the associators and unitors are identities, i.e., if for all objects $X,Y,Z$ in $\Cc$ we have $$(X \otimes  Y ) \otimes  Z = X \otimes  (Y \otimes  Z),\quad X \otimes  \unit = X = \unit \otimes  X.$$ 
\end{defin}

Note that in prescribing a strict monoidal category we may omit most of the tuple needed to prescribe a monoidal category. Thus we will write in the form 
$\Cc = (\Cc, \otimes, \unit)$, 
where $\otimes$ is replaced by the appropriate monoidal product;
and $\unit$ by the monoidal unit object.


\medskip 

The following definition 
adapts from 
\cite[Definition 4.21]{EGNO} to the case of strict categories:

\begin{defin}\label{def:monoidal functor}
    Let  $\Bb = (\Bb, \otimess{\Bb}, \One_\Bb), \; \CDd = (\CDd, \otimess{\CDd}, \One_\CDd)$  be strict monoidal categories.  A \emph{monoidal functor} from $\Bb$ to $\CDd$ is a pair $(F,J)$ where 
    \begin{enumerate}
        \item[1.] $F:\Bb\rightarrow\CDd$ is a functor such that $F(\One_\Bb)$ is isomorphic to $\One_\CDd$,
        \item[2.]  $J$ is a natural isomorphism between the functors 
$\Theta,\Phi:\Bb \times\Bb \rightarrow\CDd$ 
given by 
$\Theta:(X,Y)\mapsto F(X)\otimess{\CDd} F(Y)$ and 
$\Phi:(X,Y)\mapsto F(X\otimess{\Bb} Y)$, and 
\item[3.] $J$ satisfies
\beq    \label{eq:Jcoherence}
J_{X \otimess{\Bb} Y,Z}(J_{X,Y}\otimess{\CDd} \id_{F(Z)})
\; = \;    J_{X,Y\otimess{\Bb} Z}(\id_{F(X)}\otimess{\CDd} J_{Y,Z}).
\eq 
    \end{enumerate}
where  $J_{X,Y}$ denotes the component morphism $J_{X,Y}: F(X)\otimess{\Cc} F(Y) \rightarrow F(X\otimess{\Bb} Y)$ 
and 
$\id_A$ denotes the identity morphism on object $A$ in any setting.    

\medskip
A monoidal functor  $(F,J)$ is \emph{strict} if $J_{X,Y}=\id_{F(X)\otimes F(Y)}$ for all objects $X,Y$. {In this case we will write just $\FF$ for the pair $(\FF,\Id)$.} 
\end{defin}

\begin{remark}\label{rem:J0} Here we only require the existence of a isomorphism between $F(\unit_\Bb)$ and $\unit_\Cc$, rather than making {this isomorphism}
an explicit {part of the defining data}, as in \cite{EGNO}. 
We note that it is common in category theory texts to define a monoidal functor 
{as a triple $(F,J, J_0)$}
-   
specifying 
{$F$, $\{J_{X,Y} \}_{X,Y}$ and  
}
an isomorphism $J_0: F(\One_\Bb) \to \One_\CDd $, 
and giving conditions that this isomorphism must satisfy, {namely 
\[
J_{\unit_\Bb,X}\circ J_0\otimes \id_{F(X)}=\id_{F(X)}\quad \text{ and } \quad 
J_{X,\unit_\Bb}\circ \id_{F(X)}\otimes J_0=\id_{F(X)}
\] 
for the case where both source and target  are strict.} 
(see e.g. \cite[Sec.XI.2]{MacLane}). However, there is a unique isomorphism satisfying the conditions, and the isomorphism can be inferred from the pair $(F,J)$ see e.g. \cite[Section 2.4]{EGNO}. 
\end{remark}

\begin{example} \label{exa:J0}
\ppmm{If $\Bb = \Bcat$ and $\Cc = \Mat$ then $\unit_\Bb = 0$ and $\unit_\Cc = 1$ 
and $J_0 \in \Mat(1,1)$ is just a number. 
} {Moreover, $J_{m,0}=\id_{F(m)}\otimes J_0^{-1}$, a corresponding scalar multiple of the identity matrix.}
\end{example}
\begin{defin}\label{de:mon nat tranf}
A natural transformation $\eta:(F,J)\rightarrow (G,K)$ between monoidal functors is 
\ppmm{a natural transformation between functors $F$ and $G$, as in Def.\ref{def:natural transf}; and is}
{\em monoidal} if, in addition to the axioms in Definition \ref{def:natural transf}, we have 
$$
\; K_{X,Y} \;\; \eta_X\otimes \eta_Y \;\; = \;\; \eta_{X\otimes Y} \; J_{X,Y}    .
$$

\end{defin}

\begin{remark}
    Monoidal functors, as we have defined them, are often called \emph{strong} monoidal functors, since the $J_{X,Y}$ are isomorphisms.  We take the convention that if a monoidal functor is not strong we will use the term \emph{lax}, so that the unmodified term always means strong.\end{remark}
   
    \subsection{Strictification of monoidal functors} 
    $\;$  
    
    \medskip
    
\ppmm{Let $(F,J): \Cc \rightarrow \Dd$ be a monoidal functor between strict monoidal categories. We show here that 
if $\Ob(\Cc)$ is a finitely-generated free monoid 
then $(F,J)$ is} 
    \ppmm{monoidally naturally isomorphic}
    to a strict monoidal functor.

    On this basis, and in the interests of reducing functor classification problems to  tractable problems, we will {later} assume all functors \ppmm{$(F,J): \Cc \rightarrow \Dd$ to be classified} are strict monoidal, \ppmm{whenever $\Ob(\Cc)$ is a finitely generated free monoid (of course this includes our braid category case)}.

    \medskip
\ignore{
Let $\Cc = (\Cc, \otimes, \One_\Cc)$ be a strict monoidal category,
with $\Ob(\Cc)$  a finitely-generated free monoid. Specifically put 
$\Ob(\Cc) = \langle X_1,\ldots, X_p \rangle$, the free monoid on $p$ generators.
A word $s = s_1 s_2\cdots s_n$ in $\{1,2,\ldots,p\}^n$ represents an object by 
$s \mapsto X_{s_1} \otimes X_{s_2} \otimes \cdots\otimes X_{s_n}$, where the empty word represents $\unit_\catC$.
\\
Given a strict monoidal category $\Dd = (\Dd,\otimes, \One_\Dd)$,
let  
$(\FF,J)$ be 
a monoidal functor  (as in \ref{def:monoidal functor}) from $\Cc$ to $\Dd$.
Define a map 
$\FF^{st} : \Ob(\Cc) \rightarrow \Ob(\Dd)  \;$
by $\; \FF^{st} (\One_\Cc) = \One_\Dd$ and 
\beq \label{de:FFst} 
\FF^{st} ( s_1 s_2 \cdots s_n ) = \FF(s_1) \otimes \FF(s_2) \otimes \cdots\otimes \FF(s_n)
\eq 
that is
$$
\FF^{st} (  X_{s_1} \otimes X_{s_2} \otimes \cdots\otimes X_{s_n} ) 
=  \FF(X_{s_1}) \otimes \FF( X_{s_2} ) \otimes \cdots\otimes \FF( X_{s_n} )  . 
$$
Observe that $J_{s_1,s_2} \otimes \id_{\FF^{st}(s_3\cdots s_n)}$
is a morphism from  $\FF^{st}(s) $ to $ \FF(s_1 s_2) \otimes \FF^{st}(s_3 \cdots  s_n) $.
Fixing this object $s$ for a moment,  
write $J_j$ for the morphism 
$\; J_{s_1\cdots s_j,s_{j+1}}\otimes \id_{\FF^{st}(s_{j+2}\cdots s_n)}$,
which is from  $  \FF(s_1\cdots s_j) \otimes \FF^{st}(s_{j+1}\cdots s_n)$ to 
$\FF(s_1 \cdots s_{j+1}) \otimes \FF^{st}(s_{j+2}\cdots s_n)$. 
Thus  the composite morphism 
$\; \Psi_s =   J_{n-1} \circ J_{n-2} \circ \cdots \circ J_1  $
is a morphism 
\[
\Psi_s : \FF^{st}(s) \rightarrow \FF(s).
\]
We set $\Psi_{\One_\Cc} = J_0$ where $J_0$ is as in Remark~\ref{rem:J0}.

\newcommand{\FFF}{\FF^{st}}

\begin{example}
    
Consider a case in which the number of generators $p=1$ (such as $\Cc =\Bcat$ from (\ref{de:Bcat})), 
and $\Ob(\Dd) $ is $\langle 1 \rangle = (\N_0, +) \;$ (such as $\Dd = \Mat^N$ from (\ref{de:MatN})). 
Consider also $\FF(1) =1$ (essentially, note, WLOG).
Then $\FFF(0)=0$, 
$\FFF(1) = \FF(1)$,  
$\FFF(11) = \FF(1) \otimes \FF(1)\cong \FF(1\otimes 1 )= \FF(11)$ --- note that if $\Dd$ is skeletal then this is necessarily an equality.
\end{example}

\begin{theorem}\label{th:strictifyingmonfuns}
Let $\Cc = (\Cc, \otimes, \One_\Cc)$, and $\Dd = (\Dd,\otimes, \One_\Dd)$ be strict monoidal categories.
Suppose that 
$\Ob(\catC)$ is a free monoid
on symbols 
$\{ X_1, X_2, \ldots, X_p \}$ (any $p \in \N$). 
Let $(\FF,J)$ be a monoidal functor  
from $\Cc$ to $\Dd$
as in (\ref{def:monoidal functor}).
 Then\begin{enumerate}
     \item[(I)]
the object map $\FF^{st}$ 
from (\ref{de:FFst})
and the morphism map given on each morphism 
 $f: X \rightarrow Y$ in $\Cc$ by 
 $\FF^{st}(f) =  \Psi^{-1}_{Y} \FF(f) \Psi_{X} $ yields a strict monoidal functor 
 $\FF^{st} : \Cc \rightarrow \Dd$. 
\item[(II)] The morphisms $\Psi_{X}$ assemble to a \ecr{monoidal?} natural isomorphism from $\FF^{st}$ to $(\FF,J)$. 
\end{enumerate}
\end{theorem}

\begin{proof}
(I)
To see that $\FF^{st}$ is a functor observe that \[\FF^{st} (\id_X)= \Psi^{-1}_{X} \FF(\id_X) \Psi_{X} = \Psi^{-1}_{X} \id_{F(X)} \Psi_{X} = \id_{F^{st}(X)},\] and for any $f\colon X\to Y$ and $g\colon Y\to Z$ in $\Cc$,
\[\FF^{st}(gf) = \Psi^{-1}_{Z} \FF(gf) \Psi_{X}= \Psi^{-1}_{Z} \FF(g)\Psi_{Y}^{-1}\Psi_{Y}\FF(f) \Psi_{X}=\FF^{st}(g)\FF^{st}(f).\]
For any pair of morphisms $X=X_{s_1} \otimes \cdots\otimes X_{s_n}$ and $X' =X'_{s_1} \otimes \cdots\otimes X'_{s_n} $ we have 
$ \FF^{st} ( X \otimes X') = \FF^{st} (  X_{s_1}  \otimes \cdots\otimes X_{s_n} \otimes  X'_{s_1} \otimes \cdots\otimes X'_{s_n} ) 
=  \FF(X_{s_1}) \otimes \cdots\otimes \FF(X_{s_n}) \otimes \FF(X'_{s_1} )\otimes \cdots\otimes \FF(X'_{s_n} ) = \FF^{st}(X)\otimes \FF^{st}(X')$,
so we may choose the family of natural transformations $\FF^{st}(X\otimes Y)$ to $\FF^{st}(X)\otimes \FF^{st}(Y)$ to all be the identity.
All required identities are then trivially satisfied since associators and unitors in both $\Cc$ and $\Dd$ are also identities.
   \\ (II)
   To see that the $\Psi_X$ assemble to a natural isomorphism from $\FF^st$ to $(\FF,J)$, observe that, for any morphism $f\colon X\to Y$ in $\Cc$,
   $\Psi_Y\FF^{st}(f)=
   \Psi_Y\Psi^{-1}_{Y} \FF(f) \Psi_{X}=
   \FF(f)\Psi{X}$.
\end{proof}
}

Let $\Cc = (\Cc, \otimes, \One_\Cc)$ be a strict monoidal category,
with $\Ob(\Cc)$  a finitely-generated free monoid. 
Specifically put 
$\Ob(\Cc) = \langle X_1,\ldots, X_p \rangle$, the free monoid on $p$ generators.
A word $s = s_1 s_2\cdots s_n$ in $\{1,2,\ldots,p\}^n$ represents an object by 
$s \mapsto X_{s_1} \otimes X_{s_2} \otimes \cdots\otimes X_{s_n}$, where the empty word represents $\unit_\catC$.
\\
Given a strict monoidal category $\Dd = (\Dd,\otimes, \One_\Dd)$,
let  
$(\FF,J)$ be 
a monoidal functor  (as in \ref{def:monoidal functor}) from $\Cc$ to $\Dd$. {In our notation we have $\FF(s)=\FF(X_{s_1}\otimes\cdots\otimes X_{s_n}).$}
Define a map 
$\FF^{st} : \Ob(\Cc) \rightarrow \Ob(\Dd)  \;$
by $\; \FF^{st} (\One_\Cc) = \One_\Dd$ and 
\beq \label{de:FFst} 
\FF^{st} ( s_1 s_2 \cdots s_n ) = \FF(s_1) \otimes \FF(s_2) \otimes \cdots\otimes \FF(s_n)
\eq 
that is 
$$
\FF^{st} (  X_{s_1} \otimes X_{s_2} \otimes \cdots\otimes X_{s_n} ) 
=  \FF(X_{s_1}) \otimes \FF( X_{s_2} ) \otimes \cdots\otimes \FF( X_{s_n} )  . 
$$

\ppmm{We next construct a morphism   $ \Psi_s : \FF^{st}(s) \rightarrow \FF(s) $ for each $s \in \Ob(\Cc)$.}
Set $\Psi_{\One_\Cc} = J_0$ where $J_0$ is as in Remark~\ref{rem:J0}, and $\Psi_{X_i}= \id_{\FF(X_i)}$ for any generator $X_i$. For a word $s=s_1 s_2\cdots s_n$, let $s'= s_1 s_2\cdots s_{n-1}$ be the word with the last term dropped. We then iteratively construct the $\Psi_s$ by letting 
\beq \label{eq:dePsis} 
\Psi_s = J_{s',s_{n}}\circ (\Psi_{s'} \otimes \Psi_{s_{n}})   .
\eq 
Note that each $\Psi_s$
is 
indeed 
a morphism 
\[
\Psi_s : \FF^{st}(s) \rightarrow \FF(s).
\] 

\newcommand{\FFF}{\FF^{st}}

\begin{example}
    
Consider a case  $(\FF,J)$ in which the number of generators $p=1$ (such as $\Cc =\Bcat$ from (\ref{de:Bcat})), 
and $\Ob(\Dd) $ is $\langle 1 \rangle = (\N_0, +) \;$ (such as $\Dd = \Mat^N$ from (\ref{de:MatN})). 
Consider also $\FF(1) =1$ (essentially, note, WLOG%
). 
Here of course $\One_\Cc =0$ and $\One_\Dd =0$. 
Then $\FFF(0)=0$, 
$\FFF(1) = \FF(1)$,  
$$
\FFF(11) \stackrel{\ref{de:FFst}}{=} \FF(1) \otimes \FF(1)   
   \cong    
   \FF(1\otimes 1 )= \FF(11)
$$ 
--- note that if $\Dd$ is skeletal then this is necessarily an equality: $\FFF(11)=\FF(11)$. 
And indeed $\FFF(1^n) \cong \FF(1^n)$ for all $n$, 
with equality if $\Dd$ is skeletal.
\ignore{{ 
(since assuming 
$\FFF(1^{n-1}) \cong \FF(1^{n-1})$ we get 
$\FFF(11..1) = \FFF(1^n) =\FF(1) \otimes \FF(1) \otimes ..\otimes  \FF(1) 
=   \FF(1) \otimes (\FF(1) \otimes .. \otimes  \FF(1) )  
\cong   \FF(1) \otimes \FF(1^{n-1}) \cong  \FF(1^n)$)
}}
\end{example}

\begin{lemma}\label{lem:nattransismonoidal}
    Let $\Cc = (\Cc, \otimes, \One_\Cc)$, and $\Dd = (\Dd,\otimes, \One_\Dd)$ be strict monoidal categories
such
that 
$\Ob(\catC)$ is a finitely-generated free monoid.
Let $(\FF,J)$ be a monoidal functor  
from $\Cc$ to $\Dd$
as in (\ref{def:monoidal functor}). 
Let $s$ and $t$ be words 
representing a pair of objects in $\Cc$.
Then the morphisms $\Psi_s$ constructed above obey 
\[
\Psi_{s\otimes t} = J_{s,t}\circ(\Psi_{s}\otimes \Psi_t).
\]
\end{lemma}
\begin{proof}
We first show that the given identity is satisfied when $t$ is the empty word.
We have, using that $\Cc$ is strict and the identities given in Remark \ref{rem:J0}, \[\Psi_{s\otimes \One_\Cc} = \Psi_{s} = J_{s,\One_\Cc} \circ (\id_{\FF(s)}\otimes J_{0}) \circ \Psi_{s}  = J_{s,\One_\Cc} \circ (\Psi_s \otimes J_0).\]

It follows immediately from the definition of the morphisms $\Psi$ that the identity is satisfied when $t$ is a word of length $1$.

We now prove the result by induction on the length of $t$. In particular, for $t=t_1\cdots t_n$ suppose that the identity is true for $t'=t_1\cdots t_{n-1}$.
Now we have, where the second equality uses the induction assumption, and the fourth uses condition 3 of Def.~\ref{def:monoidal functor},
\begin{align*}
\Psi_{s\otimes t} &   \;  \stackrel{(\ref{eq:dePsis})}{=}  \;
J_{s\otimes t',t_{n}}\circ (\Psi_{s\otimes t'} \otimes \Psi_{t_{n}})
= J_{s\otimes t',t_{n}}\circ ((J_{s,t'}\circ(\Psi_{s}\otimes \Psi_{t'})) \otimes  \Psi_{t_{n}}) \\ 
&= J_{s\otimes t',t_{n}}\circ (J_{s,t'}\otimes \id_{t_{n}}) \circ (\Psi_{s}\otimes \Psi_t'\otimes  \Psi_{t_{n}})\\&
= J_{s,t'\otimes t_{n}} \circ (\id_{s}\otimes J_{t',t_{n}})\circ (\Psi_{s}\otimes \Psi_{t'} \otimes  \Psi_{t_{n}}) \\ 
&=J_{s,t'\otimes t_{n}} \circ (\Psi_{s}\otimes \Psi_{t}).
\end{align*}
\end{proof}

\begin{theorem}\label{th:strictifyingmonfuns}
Let $\Cc = (\Cc, \otimes, \One_\Cc)$, and $\Dd = (\Dd,\otimes, \One_\Dd)$ be strict monoidal categories, 
with $\Cc$ 
such 
that 
$\Ob(\catC)$ is a finitely generated free monoid.
Let $(\FF,J)$ be a monoidal functor  
from $\Cc$ to $\Dd$
as in (\ref{def:monoidal functor}).
 Then\begin{enumerate}
     \item[(I)]
the object map $\FF^{st}$ 
from (\ref{de:FFst})
and the morphism map given on each morphism 
 $f: X \rightarrow Y$ in $\Cc$ by 
 $\FF^{st}(f) =  \Psi^{-1}_{Y} \FF(f) \Psi_{X} $ yields a strict monoidal functor 
 $\FF^{st} : \Cc \rightarrow \Dd$. 
\item[(II)] The morphisms $\Psi_{X}$ assemble to a monoidal natural isomorphism from $\FF^{st}$ to $(\FF,J)$. 
\end{enumerate}
\end{theorem}

\begin{proof}
(I)
To see that $\FF^{st}$ is a functor observe that \[\FF^{st} (\id_X)= \Psi^{-1}_{X} \FF(\id_X) \Psi_{X} = \Psi^{-1}_{X} \id_{F(X)} \Psi_{X} = \id_{F^{st}(X)},\] and for any $f\colon X\to Y$ and $g\colon Y\to Z$ in $\Cc$,
\[\FF^{st}(gf) = \Psi^{-1}_{Z} \FF(gf) \Psi_{X}= \Psi^{-1}_{Z} \FF(g)\Psi_{Y}^{-1}\Psi_{Y}\FF(f) \Psi_{X}=\FF^{st}(g)\FF^{st}(f).\]
For any pair of {objects} $X=X_{s_1} \otimes \cdots\otimes X_{s_n}$ and $X' =X'_{s_1} \otimes \cdots\otimes X'_{s_n} $ 
\begin{multline*}\FF^{st} ( X \otimes X') = \FF^{st} (  X_{s_1}  \otimes \cdots\otimes X_{s_n} \otimes  X'_{s_1} \otimes \cdots\otimes X'_{s_n} ) 
= \\ \FF(X_{s_1}) \otimes \cdots\otimes \FF(X_{s_n}) \otimes \FF(X'_{s_1} )\otimes \cdots\otimes \FF(X'_{s_n} ) = \FF^{st}(X)\otimes \FF^{st}(X'). \end{multline*}
Moreover, for morphisms $f\colon X\to Y$ and $g\colon X' \to Y'$, using Lemma~\ref{lem:nattransismonoidal},
\begin{align*}
\FF^{st}(f\otimes g) &= \Psi_{Y\otimes Y'}^{-1} \circ \FF(f\otimes g)\circ \Psi_{X\otimes X'}\\
&= (\Psi_{Y}^{-1} \otimes \Psi_{Y'}^{-1}) \circ J^{-1}_{Y,Y'} \circ \FF(f\otimes g) \circ J_{X,X'}\circ(\Psi_{X}\otimes \Psi_{X'}) \\
&=  (\Psi_{Y}^{-1} \otimes \Psi_{Y'}^{-1}) \circ
(\FF(f)\otimes \FF(g)) \circ (\Psi_{X}\otimes \Psi_{X'})\\
&= \FF^{st}(f)\otimes \FF^{st}(g)
\end{align*}
so we may choose the natural transformation $\FF^{st}(X\otimes Y)$ to $\FF^{st}(X)\otimes \FF^{st}(Y)$ to have all components identities.
All required identities are then trivially satisfied since associators and unitors in both $\Cc$ and $\Dd$ are also identities.
   \\ (II)
   To see that the $\Psi_X$ assemble to a natural isomorphism from $\FF^{st}$ to $(\FF,J)$, observe that, for any morphism $f\colon X\to Y$ in $\Cc$,
   $\Psi_Y\FF^{st}(f)=
   \Psi_Y\Psi^{-1}_{Y} \FF(f) \Psi_{X}=
   \FF(f)\Psi_{X}$.
   To see that the natural transformation is in addition monoidal observe that looking at Def.~\ref{de:mon nat tranf} and letting $(F,J)$ there be the current $(F^{st},\id)$ and $(G,K)$ there be $(\FF,J)$ gives precisely the identity proved in Lemma~\ref{lem:nattransismonoidal}.
\end{proof}

  \ppmm{First we note that the axioms are not spurious:} {Without the freeness condition {on the object monoid} the conclusions can fail. A non-free but finitely generated example \cite[Section 2.6]{EGNO} is to take a non-trivial 2-cocycle on a finite group $G$ to build a non-strict monoidal structure on the identity functor on $Vec_G$, which cannot be strictified.  Moreover, {Theorem \ref{th:strictifyingmonfuns} is not vacuous: strictness is not automatic. There are examples of non-strict functors between categories satisfying the hypotheses of Theorem \ref{th:strictifyingmonfuns}:}}
\begin{example}\label{ex:non-strict}
Let \(R  \in \Mat(N^2,N^2)\) be a YBO
\ppmm{(for definiteness consider the flip operator in rank $N=2$)}, 
and 
\(
F_R:\Bcat \rightarrow \Mat
\) the corresponding strict monoidal functor.

We construct from this, 
\ppmm{and a certain infinity of choices,} a non-strict monoidal functor
\(
(G,J,J_0):\Bcat \rightarrow \Mat.
\)

Choose isomorphisms \(T_n:N^n\to N^n\)  
for each $n$ as follows. Take 
\[
T_0=\id_0,\qquad T_1=\id_1,
\] and 
 take \(T_n\in\Aut(N^n)\) arbitrary
 for
\(n\neq 2\)  
For each $n$ define a function $G:\Bcat(n,n) \rightarrow \Mat(N^n,N^n)$ by 
\[
G(\beta)=T_n^{-1}F_R(\beta)T_n
\qquad \text{for } \beta\in \Bcat(n,n) = B_n
\]
and 
define the following morphisms in $\Mat$:
\[
J_0=\id_0,\quad
J_{m,n}
=
T_{m+n}^{-1}\,(T_m\otimes T_n).
\]

With these choices, \((G,J,J_0)\) is a monoidal functor.
However it is generally not strict.  In particular,
\(
G(\sigma)=T_2^{-1}RT_2
\) 
need not itself satisfy the Yang--Baxter equation on the nose. 
\ppmm{In our example we can choose $T_2$ to diagonalise the flip operator - and the diagonal form is not an $R$-matrix, as it would be if $(G,J,J_0)$ was strict.}
\end{example}


\subsection{The categories $\Bcat$, $\Mat$ and $\Mat^N$}    \label{ss:Bcat}
$\;$  

\ignore{{
The two categories {we} will mainly focus on are $\Bcat$ and $\Mat$. \ppm{[-or rather $\Mat^N$?]}\ecr{[I think no. Eventually we often restrict to $\Mat^N$, but this is to stratify.  $YB(\Mat^N)$ is not monoidal, in particular.  Subobjects also require us to look at the full $YB(\Mat)$.]}
}}

Recall that a strict monoidal category is {\em natural} if the object monoid is (isomorphic to) the monoid $(\N_0 , +)$, i.e. the object monoid is freely generated by a single object, usually denoted 1. 
A {\em natural functor} is a strict monoidal functor between natural categories 
such that $F(1)=1$  (see e.g. \cite{MR1X}). 

\mdef  \label{de:Bcat}
The category $\Bcat$ is the category of braids, in the sense of Mac~Lane \cite[\S XI.4]{MacLane}.
The category $\Bcat$ 
is thus a skeletal, diagonal, strict monoidal category.  
The object monoid is freely generated by a single object
(so $\Bcat$ is natural); 
and $\Bcat(n,n)$ is the braid group $B_n$. 
\\ 
The monoidal product on morphisms can be viewed as a
juxtaposition by means of the left-to-right embedding  $B_n\times B_m\rightarrow B_{n+m}$ (see e.g. \cite[\S13.1]{Martin91}). For example, fixing $k$ then 
$\sigma_i=1^{\otimesb i-1}\otimesb\sigma\otimesb 1^{\otimesb k-i-1}$ is the usual
element in $B_k$ braiding the $i$-th and $i+1$st strands. 

\medskip 

The following elementary properties of $\Bcat$ will be crucial here:
\\ (I)
The object monoid 
of $\Bcat$ 
can be taken to be $(\N_0, +)$.
\\ (II)
The subcategory generated as a monoidal category 
by the elementary braid $\sigma \in \Bcat(2,2)$
and its inverse $\sigma'$ is the whole of $\Bcat$.
\\ (III)
Indeed, $\Bcat$ can be presented as a 
strict 
monoidal category  generated 
by two (mutually inverse) morphisms in $\Bcat(2,2)$. 
This presentation requires only two relations:
\beq  \label{eq:inv}
\sigma \sigma' =1
\eq  
and the `constant Yang-Baxter equation'
\beq \label{eq:YBE}
(\sigma \otimesb 1) (1\otimesb \sigma )(\sigma \otimesb 1) 
= (1\otimesb \sigma ) (\sigma \otimesb 1) (1\otimesb \sigma )
.
\eq
\ignore{
Thus the standard generators $\sigma_i\in\Bcat(k,k)=B_k$ (the braid group on $k$-strands) may be denoted $\sigma_i=1^{\otimesb i-1}\otimesb\sigma\otimesb 1^{\otimesb k-i-1}$.
}
\ignore{ 
The monoidal product on morphisms can be viewed as a
juxtaposition by means of the left-to-right embedding  $B_n\times B_m\rightarrow B_{n+m}$ (see e.g. \cite[\S13.1]{Martin91}) so that, for example, for 
{$\sigma=$}
$\sigma_1\in\Bcat(2,2)$ we have $\sigma_1\otimesb\sigma_1=\sigma_1\sigma_3\in\Bcat(4,4)$. 
}

Note that for    {$\sigma=$}    $\sigma_1\in\Bcat(2,2)$ we have $\sigma_1\otimesb\sigma_1=\sigma_1\sigma_3\in\Bcat(4,4)$. 
But if the left $\sigma_1$ in $\sigma_1 \otimes \sigma_1$
is in $\Bcat(3,3)$ and the right is in $\Bcat(2,2)$ then we would obtain $\sigma_1\sigma_4\in\Bcat(5,5)$, for example.


\mdef  \label{pa:Binv}
Observe from (\ref{eq:inv},\ref{eq:YBE}) that there is a symmetry in the construction of $\Bcat$ under swapping $\sigma$ and $\sigma'$. 
That is, there is an involutive strict monoidal functor $\Fio: \Bcat\rightarrow \Bcat$ given by the assignment 
\beq 
\Fio(\sigma)=\sigma' .
\eq 

\mdef   \label{pa:Bopp}
Observe that $\Bcat$ is isomorphic to its opposite. 
The functor $\Feta: \Bcat \rightarrow \Bcat^o$ is the identity map on objects
and takes each morphism 
to its inverse 
(rather than inverting `generator-wise' as in (\ref{pa:Binv})).
(This lifts to the linearised categories in the natural way.) 

\mdef 
Properties (\ref{pa:Binv}) and (\ref{pa:Bopp}) mean in particular that if $F(\sigma)=R$ gives a representation then both 
$F'(\sigma) =R^{-1}$ and,
if $F: \Bcat\rightarrow \Mat$,  also 
$F''(\sigma) = R^{tr}$ (matrix transpose) give a representation. 

See \S\ref{s:equivandauto}  {\em et seq} for how these constructions inform the classification programme. 

\mdef  \label{pa:Bflipp}
There is also an important geometric isomorphism functor 
$\Fox : \Bcat \rightarrow \Bcatmop$
- the map on morphisms is (if braids are seen as passing down the page) to read braids from right to left, i.e. to flip about a vertical axis.

Note that there is an `inner' realisation of $\Fox$: 
conjugation in each group $B_n$ by the half-twisted full-$n$-strand ribbon braid 
- this braid is denoted $M$ in \cite[\S5.7.2]{Martin91}, but here it will be convenient to have a notation that records $n$, so we write $\Mmm_n$.

\medskip 

\mdef  \label{de:Sym}
We write  
$\Sym$  for the `symmetric' category. It 
is directly analogous to the category $\Bcat$ above but is the category of symmetric groups 
$\Sym_n$. (So if we write the generating elementary transposition as 
$\varsigma \in \Sym(2,2)$ then we have 
the additional relation $\varsigma^2 = Id_2$.) 

\medskip

\ignore{{
\ecr{\sout{What follows is} \ppm{See Appendix~\ref{which?} for [?]} a ham-handed attempt to formulate $\Mat$ in a convenient way.  In my mind I am think of it as its fully faithful embedding into $Vec_{basis}$ of f.d. vector spaces with a chosen (ordered) basis.  The latter can be made monoidal, but of course one has to chose the right ordering.  Which we do.  We could formulate it in the formal way, perhaps?  It is convenient to think of the morphisms as matrices being applied to vectors.}
}}

\mdef  \label{de:Mat}
We write  
$\Mat = (\Mat,\otimes,1)$  for  
the skeletal, strict monoidal linear category of matrices over
a given commutative ring $k$ 
(by default we take $k=\C$), with object monoid  $(\N,\cdot)$
{(meaning $m\otimes n = m.n$ on objects)}, 
and 
with   
Kronecker product for the monoidal product of morphisms - see below. 
(Recall that the Kronecker product corresponds to tensor product of linear maps with respect to a chosen basis.)  

Here our convention for morphism sets is that morphisms in $\Mat(m,n)$ are $n \times m$ matrices, i.e. with $n$ rows (a departure, note, from the 
conventions in \cite{MR1X}).
The row labels for matrices in $\Mat(m,n)$ are $1,2,...,n$; and columns $1,2,...,m$.

In particular $\Mat(m,1)$ is here the space of $m$-component  
row vectors.  
If we write $\bra{i} \in \Mat(m,1)$ we mean the elementary 
row vector with entry 1 in the $i$-th position.  Similarly $\ket{i} \in \Mat(1,n)$. 
Given a matrix $M \in \Mat(m,n)$ we continue to write $M_{ij}$ for the entry in row $i$, column $j$. Thus $M_{ij} = \bra{i} M \ket{j}$. 

\ignore{{
There are some conventions to establish. 
Here
we regard $\Mat$ as the standard skeletonization of $Vec$, the category of finite dimensional vector spaces \sout{linear algebra perspective} and identify objects in $\Mat$ with \ecr{the space of} column vectors \ecr{i.e., $\C^N$}. \ppm{[-I have no idea what this means! there are uncountably many column vectors. And this seems to conflict with what you just said above.]}
so that 
\ppm{[``so''? I don't understand the logic yet.]}
the morphism set $\Mat(M,N)$ consists of all $N\times M$ matrices, acting on the left. 
\ppm{[``acting''? Isn't this a category? Can we delete the last clause? Or separate it and clarify a bit, if it passes back to the standard lin alg point?]}

In practice the standard basis for $\C^K$, which label the rows in $\Mat(K,N)$, 
\ppm{[-I am completely lost. Some typos, and some more explanation needed? $M$ is $K$?]}
will be denoted $\{\ket{j}: 1\leq j\leq K\}$, and the tensor product $\ket{j_1\cdots j_n}:=\ket{j_1}\otimes\cdots\otimes\ket{j_n}$, i.e., words of length $n$ in the symbol set $\underline{K}=\{1,2,\ldots,K\}$. 
\ppm{[I have no idea what is going on here. Is it possible that this is jumping the gun and we are already in $\Mat^K$ here?]}
}}

\mdef On morphisms, i.e. matrices, we take the  
the $Ab$ convention 
\cite[Ch.3 sec.1]{Murnaghan}
for the Kronecker product, so that, for example, 
\beq
A \otimes B 
= \mat a_{11} b_{11} & a_{12} b_{11} &... \\ 
a_{21} b_{11} & a_{22} b_{11} & ... \\
\vdots  & \vdots 
\tam
\; \mbox{ and in particular } \;\;\;
A \otimes \matt{cc} 1&0 \\ 0&1 \ttam  = \matt{c|c} A & \\ \hline  & A \ttam 
\eq 
For comparison the $aB$ convention gives:
\beq
A \optimes B 
= \mat a_{11} b_{11} & a_{11} b_{12} &... \\ 
a_{11} b_{21} & a_{11} b_{22} & ... \\
\vdots & \vdots 
\tam
\eq 
The choice of $Ab$ convention corresponds to 
the order of the basis  of $\C^K\otimes \C^L$ 
(needed, since this product passes to the same object in $\Mat$ as $\C^{K.L}$, which already has ordered  basis) being $\ket{1}\otimes\ket{1},\ket{2}\otimes\ket{1},\ldots,\ket{K}\otimes\ket{1},\ket{1}\otimes\ket{2},\ldots,\ket{K}\otimes\ket{L}$.  And similarly for $3$ or more tensor factors: we take the 'standard' ordered bases and then use the reverse lexicographic ordering induced by the ordering on the tensor factors.

\mdef  \label{de:Matmop}
As in \cite{MR1X} it will be convenient to write $\Matop = (\Mat,\optimes,1)$ 
for the monoidal category with the aB-convention Kronecker product, i.e. 
$A \optimes B = B \otimes A$. 

\ignore{{\mdef
For $N,M \in \N$
let $P_{N,M}$ be the $NM\times NM$ matrix 
that changes the ordered basis 
\ppm{[no ordered basis here! say: row/column order]}\ecr{[I am trying to efficiently describe a matrix.  Would it help to write out an example, say with $N=2,M=3$? I put it in, to be deleted later, of course.]}
$\ket{11},\ket{21},\cdots$ to $\ket{11},\ket{12},\cdots$. \ppm{[-these things seem undefined in $\Mat$, but I don't think we'll ever use them, so let's just leave it.]}
Then $A\optimes BP_{M,N}=P_{N,M}A\otimes B$.  \ppm{[I dont understand this yet. but we dont use it so I'll move on.]} \ecr{ For example, the case of $N=2$ and $M=3$ has \[ P_{2,3}=\left[ \begin {array}{cccccc} 1&0&0&0&0&0\\ \noalign{\medskip}0&0&1&0
&0&0\\ \noalign{\medskip}0&0&0&0&1&0\\ \noalign{\medskip}0&1&0&0&0&0
\\ \noalign{\medskip}0&0&0&1&0&0\\ \noalign{\medskip}0&0&0&0&0&1
\end {array} \right]
 \]
which sends $[\ket{11},\ket{21},\ket{12},\ket{22},\ket{13},\ket{23}]$ to $[\ket{11},\ket{12},\ket{13},\ket{21},\ket{22},\ket{23}]$}
In fact, $P_{N,M}$ gives $\Mat$ the structure of a (symmetric) braided monoidal category, see (\ref{ybmat is monoidal}).}}

\mdef  \label{de:MatN}
Recall from \cite{MR1X} that $\Mat^N$ denotes 
the full (hence linear) monoidal subcategory of $\Mat$ monoidally generated
by the object $N$ in $\Mat$.  
The object monoid of $\Mat^N$ may be taken to be $(\N_0,+)$, 
but corresponds 
(via exponentiation)
to the multiplicative submonoid $\{1,N,N^2,\ldots\}$ in $\Mat$. 
Thus morphisms in $\Mat^N(j,k)$ are $N^k\times N^j$ matrices. 

Note that $\Mat^1$ makes sense, but by default we will assume that $N>1$ here. 

Observe that a SMF $\; F:\Bcat \rightarrow \Mat$ is essentially the same as  a SMF 
$\; F':\Bcat \rightarrow \Mat^{F(1)}$ with $F'(1) =1$.  

Note that in this sense {\em every} braid representation can be regarded as a natural functor to some $\Mat^N$. 
Possible interest in functors that are {\em not} natural arises only once has restricted 
the target category to one that is not full. 
We will have examples of this later. 


Another notational difference from $\Mat$ is that 
the index set for rows of a matrix in $\Mat^N(j,k)$ is the set 
$\underline{N}^k$ of words of 
length $k$ in the symbol set $\underline{N}=\{1,2,\ldots,N\}$. 
Our convention 
is to take 
{\em revlex} order on words 
(for example 11, 21, 31, ..., $N1$, 12, 22, 32, ...).
This yields that $|w\rangle \otimes |v\rangle =| wv\rangle$ 
(concatenation of words). 

\newcommand{\mm}{{\mathbf{m}}}
\newcommand{\ff}{{\mathsf{f}}}

\mdef \label{de:alma-order} 
Fix $N,k \in \N$. 
Recall that for $S$ a set then $S^*$ is the set of all words in the symbol set $S$. 
We define a partial order on  $\underline{N}^k$ as follows. 
Firstly define
\[
\Lambda^N_n \; = \; \{  \mm \in \N_0^N \; | \;  \mm.(1,1,...,1)=n \}
\]
- the set of {\em weak compositions} of $n$ into $N$ parts. 
Then define $\ff :  \underline{N}^k \rightarrow  \Lambda^N_k $ by 
setting $\ff(w)_i = | \{ j | w_j = i \}|$. 
Define $\overline{\ff} : \N_0^N \rightarrow \underline{N}^*$ by
$ 
\overline{\ff} (\lambda) =  ( N,N,...,N, \; N\!-\!1, ..., 1,1,...,1 )
$
where the number of $i$'s is $\lambda_i$. 
We define an order $(\Lambda^N_k , <)$  by ordering $\underline{N}^k$ in revlex,
applying $\ff$ to this sequence, then $\lambda < \mu$ if $\lambda$ first appears 
before $\mu$ in the sequence $\ff(\underline{N}^k)$. 
This then induces a partial order on  $\underline{N}^k$ via $\ff$ in the obvious way. 
We can extend to $\underline{N}^*$ by  taking words of different lengths to be incomparable. 

\mdef  \label{de:flip}
Observe that for fixed $N$ there is a permutation matrix at each level $n$ that takes revlex order to lex order 
(11,12,13,..., $1N$, 21, 22, 23, ...). We write $\Fff_n$ for this matrix. 
For example with $N=2$
\beq
\Fff_2  = \mat 1 \\ &&1 \\ &1 \\ &&& 1 \tam ,   \hspace{.5cm} 
\Fff_3 = 
\begin{blockarray}{ccccccccc} 
 & 111 & 211 & 121 & 221 & 112 & 212 & 122 & 222 \\ 
\begin{block}{c[cccccccc]}
 111 &  1 &&&&&& \\ 
 211 &   &&&& 1 \\ 
 121 & && 1 \\
 221 & &&&&&& 1 \\ 
 112 & &1   \\
 212 &  &&&&& 1  \\
 122 &  &&&1  \\
 222 & &&&&&&& 1
 \\
\end{block}  
\end{blockarray}
\eq 

This word index set conforms with our  convention 
of denoting column vectors as $\ket{w}$ for $w$ a word of length $j$ in $\ul{N}$.  Then for $M\in\Mat^N(j,k)$ we write the coefficients as $M_{v,w}$ where $M\ket{v}=\sum_{w\in \ul{N}^k}M_{v,w}\ket{w}$.

\mdef\label{matrix units}  {We denote the corresponding matrix units by $E_{v,w}$, where $E_{v,w}\ket{ w}=\delta_{v,w}\ket{v}$.  Thus \(E_{ij,kl}\) denotes the matrix unit sending the basis vector \(\ket{kl}\) to \(\ket{ij}\).}

\mdef   \label{de:Farr}
We write $\Farr:\Mat^N \rightarrow \Matop^N$ for the functor flipping the Kronecker convention. We have 
\beq
\Farr (M) = \Fff_{\sss(M)} M \Fff_{\ttt(M)}
\eq

\section{Effective narrowing of the classification problem}\label{ss:subcats}

Here we focus on narrowing the classification problem by narrowing the final target. 
Another significant form of narrowing is to require that functors factor through some simpler intermediate category 
(for example using the natural functor ${\mathfrak S}: \Bcat\rightarrow\Sym$ and a classification of symmetric representations), 
but we will address this in \S\ref{s:equivandauto}. 
 
As noted in the Introduction, and in \ref{de:MatN},
restriction of the target of $F:\Bcat\rightarrow\Mat\;$ from $\Mat$  to $\Mat^N$ for some $N$ is no restriction at all. It provides,  rather, the `universe' of the first natural step in classification - classification simply according to the image $F(1)$.
With a view towards \sproblems~\ref{prob:1} and \ref{prob:2}, we describe some properly restricted targets. 
As we will see in Section \ref{s:classifications}, these have emerged in pursuit of 
the aim of 
materially simplifying the classification problem,
while retaining important classes of braid representations. 
The corresponding problems  
will be revisited in section \ref{s:subcats of MonFun} after we have developed some tools. 

The current status of each of these classification problems  
will be described in Section \ref{s:classifications}.
Here we restrict ourselves to a few remarks needed to set the scene. 

\ignore{{
\ppm{[
A section specifically on restrictions of the classify braid reps problem. - In addition to restricting the target, `source-side' restrictions such as restricting eigenvalues of $R$ are very important, for example? [I think we did have a functioning version of this sec at some point (although my specs might be rose-tinted) - I'll have a search back.]]}\ecr{[substantially rearranged by Paul and Eric 12/5/25.]}

\subsection{Subcategories and subsets of $\Mat$}\label{ss:subcats}   $\;$

As noted in the Introduction, and in \ref{say-where},
restriction of the target of $F:\Bcat\rightarrow\Mat\;$ from $\Mat$  to $\Mat^N$ for some $N$ is no restriction at all.
Now we describe some properly restricted targets, some of which  
materially simplify the classification problem,
without discarding important classes of braid representations, as we shall see. }}

The main motivating example here is the category type $\Match^N$, for $N \in \N$. 
We recall the definition in \ref{de:Match}.
Each such ansatz amounts to an ansatz for the R-matrix $F(\sigma)$. 
That is, the initial ansatz, assuming $F(1)=1$,  
is that $R$ can be a generic matrix in, in this case,  
$\Match^N(2,2)$ - hence with a corresponding number of nominally free parameters. 
The generic ansatz is then narrowed by solving the YBE, 
passing to a solution variety. 
Of course the ansatz does not have to be characterised categorically - this kind of
condition is an enrichment strategy motivated by the great difficulty of the 
classification problem in general. 
A simple example is Hietarinta's approach in \cite{HietarintaUpperTri}, 
where the ansatz is that the `unchecked' R-matrix is upper-triangular. 
In this case the ansatz is sufficient to complete classification in ranks $N=2,3$. 
But higher ranks remain open. This contrasts instructively with $\Match^N$, 
where classification has been completed in all ranks when $F(1)=1$. 
(In fact in \ref{de:ccwg} below we will include a discussion of the `pushout' of these two approaches.) 

The $\Match^N$ ansatz is well-motivated for a number of intrinsic reasons, before we get to the bonus that it allows solution of the braid representation classification problem when $F(1)=1$. 
There is a strong link between quantum spin chains, quantum groups, 
quantum Schur-Weyl duality and braid representations, and $\Match^N$ with $F(1)=1$ captures this, so that many well-studied representations are of this form, as we will see. 
Formulating this class categorically (or even just in terms of `charge conservation' - where matrix row and column indices are spin-chain configurations, organised by a suitable notion of total charge)
naturally suggests various generalisations, and we will start here with one of these. 

\begin{defin}  \label{de:Matcha}
Let $N \in \N$. A matrix $M$ in $ \Mat^N(j,k)$ is {\em additive charge conserving} if $M_{v,w} \neq 0$ implies that $\sum_i v_i = \sum_i w_i$ 
(recall from \ref{de:MatN} that $v \in \ul{N}^k$, see \cite{HMR1}). 
\end{defin}
In \cite{Almateari24} it is shown that, 
for each $N$, 
additive charge conserving matrices form a linear strict monoidal subcategory
of $\Mat^N$. 
We denote this subcategory $\Matcha^N$. 

\begin{defin}  \label{de:Match}
A matrix $M\in\Mat^N(n,m)$ is {\em charge conserving}
if 
$M_{w,w'} = \langle w | M | w' \rangle \neq 0$ 
implies that 
$w$ is a perm of $w'$. 
That is $w = \sigma w'$ for some $\sigma\in\Sigma_n$,
where the symmetric group 
$\Sigma_n$ acts by place permutation
(note in particular that this requires $n=m$).
\end{defin}

The subset of $\Mat^N$ of charge conserving (cc) 
matrices
forms a diagonal linear monoidal subcategory (see for example 
\cite[Lem~{3.7I}]{MR1X}) denoted $\Match^N$. 

\mdef 
Notice that $\Match^N$ is a subcategory of $\Matcha^N$, since the condition $\sum_i v_i = \sum_i w_i$ is clearly satisfied if $v$ is a permutation of $w$.  Note also that the morphisms in $\Match^N$ are all square matrices.

\medskip 

Observe that the algebras $\Match^N(n,n)$ are, in themselves, semisimple (see also \cite{Almateari24}). 
Choosing a semisimple target for representation theory does not imply semisimple 
representations, since the image may be smaller than the target 
(this holds in ordinary 
\ppmm{linear/Artinian}
representation theory and, piecemeal, in our higher representation theory). 
But as we will see in \S\ref{s:classifications}, such a target is restrictive on non-semisimple representations. 
Enlarging the target generally makes the classification problem harder, so it is 
useful to have a target which is close to $\Match^N$ (in a sense that we will explain 
in (\ref{de:ccwg}-\ref{pa:Qfunctor}))
but retains key properties - in the sense that the maximal semisimple quotient of an {\em ordinary} representation is a representation retaining key properties (such as operator spectrum/characters).  

\begin{defin}  \label{de:ccwg}
Let $N,n,m \in \N$. 
A matrix $M\in\Mat^N(n,m)$ is {\em charge conserving with glue (ccwg)}
if 
$M_{w,w'} = \langle w | M | w' \rangle \neq 0$ 
implies that 
$w < w'$ (where $(\underline{N}^*,<)$ is the order defined in \ref{de:alma-order}).  
\end{defin}

The subset of $\Mat^N$ of ccwg 
matrices forms a diagonal linear monoidal subcategory 
denoted $\Matchg^N$. 
(The term `glue' comes from the informal term for radical elements in non-semisimple ordinary representation theory.) 

\mdef 
Notice that $\Match^N$ is a subcategory of $\Matchg^N$. 

Notice also that the number of nominally free parameters in the ccwg ansatz is greater 
than in Hietarinta's upper-triangular ansatz \cite{HietarintaUpperTri}; but that this does not of itself 
imply that the ansatz will embrace more solutions. 

\mdef   \label{pa:Qfunctor}
Regarding $\Match^N(n,n)  \hookrightarrow \Matchg^N(n,n)$ simply as vector spaces, then 
$\Matchg^N$ of course has the 
complementary 
pure-glue subspace (all entries allowed non-zero in $\Match^N$ are zero). 
We thus have linear maps   
${\mathcal Q}:\Matchg^N(n,n) \rightarrow  \Match^N(n,n)$ by taking the quotient by the pure-glue subspace 
(or equivalently by zeroing the pure-glue entries). 
A key result is: \\ 
The map ${\mathcal Q} $ map extends to a functor 
${\mathfrak{Q}}:  \Matchg^N \rightarrow \Match^N $. 
\\ 
This means in particular that every ccwg braid representation restricts (in this sense) to a cc braid representation. 
We thus have a partial classification of ccwg representations according to their restriction. 
(This is in contrast with, for example, $\Matcha^N$. Here there is a corresponding vector space map, but no functor.)

\medskip  \medskip 

In the above examples we first restrict the object monoid 
(this is simply the choice of $N$), thus passing from $\Mat$ to $\Mat^N$ --- as noted, this is no restriction at all, 
provided that we eventually consider all $N$. 
And then we restrict the set of morphisms, 
to a subset which must be closed under composition. 
Again this is natural, since $\Bcat$ is closed. 

Notice that our default target $\Mat$ is linear. 
That is, its morphism sets are $\C$-vector spaces. 
This is beneficial in the same way as it is in ordinary representation theory - 
giving us access to the rep theory of the group algebras of the collections of 
groups (in $\Bcat$) that we are studying. 
The variants 
$\Mat^N$, $\Match^N$, $\Matcha^N$ and $\Matchg^N$ are also linear categories. 

One may, 
instead of restricting the object monoid first,  
restrict the morphisms first as in the following:

\mdef 
We may consider the (not linear) subcategory $\UMat$ consisting of unitary matrices in $\Mat$, i.e., $U\in\Mat$ such that $UU^\dag=\Id$ where $U^\dag$ is the usual conjugate transpose.
That is to say, the set of unitary matrices is closed under the matrix multiplication, but not addition. 

Notice that in $\UMat$ every morphism is an automorphism, i.e. there are no morphisms between $n,m$ if  
{$n \neq m$}; and every morphism is invertible. 
Hence $\UMat$ is a monoidal groupoid. The subcategory $\OMat$ of (real) orthogonal matrices form a subcategory.

\begin{defin}\label{def:monomial/perm}
     A matrix $M$ is called \emph{monomial} if there is exactly one non-zero entry in each row and column. Notice that monomial matrices are closed under composition. Let $\MonMat$ denote the (not linear) subcategory of $\Mat$ of monomial matrices.   A matrix $P$ is a \emph{permutation} matrix if it is monomial and each non-zero entry is $1$. {Let $\PermMat$ denote the subcategory of $\MonMat$ of permutation matrices.}
     \end{defin}
     Notice that $\MonMat$ and $\PermMat$ are also monoidal groupoids.

Other collections of matrices that do not form monoidal subcategories of $\Mat$ are worthy of consideration as well. 
A prime example are group-type matrices:
\begin{defin}\label{def:group-type} [c.f. \cite{AS}]
    Let $N\in\N$ and $R\in\Mat^N(2,2)$.  Then $R$ is of \emph{group-type} 
   if, with respect to the standard basis $\{\ket{i}\}_{i=1}^N$ of $k^N$
   there are $g_i\in GL(k^N)$ such that for all $i,j$ \[R\ket{ij}=g_i\ket{j}\otimes \ket{i}.\]
\end{defin}

Group type matrices are not closed under composition, and so do not form a category.  However, we may still employ the same notation as above and denote by $\GTMat^N$ the $N^2\times N^2$ group-type matrices, and consider these as targets.

\mdef  \label{de:invol}
Another class of matrices considered as targets
that are not closed under composition is that of \emph{involutive}  matrices, i.e. those satisfying $A^2=I$. 
 We will denote these matrices by $\Invol$ generally, and $\Invol^N$ for the $N^2\times N^2$ cases.  In section \ref{sss:invol} we will see that this may also be interpreted as restricting the source.

\ignore{{
\subsection{On cases with $F(1) \neq 1$}

As noted in the Introduction, the motivating examples here are \ecr{objects $F$ in} \sout{those of form} $\MonFun(\Bcat,\Match^N)$,
which have all been completely classified, subject only to the condition $F(1)=1$. \ecr{[Does the introduction need more specific language?]}
It is instructive for a number of reasons  to explain what happens when 
this condition is lifted.
The most obvious new case here is $F(1)=2$, and this example is indeed illuminating, as we
show next.  

\newcommand{\x}{x}
\setcounter{MaxMatrixCols}{20}

Here is the `shape' of an $R$-matrix targeting $\Match^2$ with $F(1)=2$:
\beq  \label{eq:F12shape}
\matt{cccc|cccc|cccc|cccc} 
\x &   &   &   &   &&&&&&&&   \\
   &\x &\x &   & \x &   &&& \x &&& \\
   &\x &\x &   & \x &   &&& \x &&& \\
   &   &   &\x &    &      \x & \x & & & \x&\x &  & \x&&& \\ \hline 
   &\x &\x &   & \x &   &&& \x &&&  \\
   &   &   &\x &    &      \x & \x & & & \x&\x &  & \x&&& \\ 
   &   &   &\x &    &      \x & \x & & & \x&\x &  & \x&&& \\ 
   &&&&&&& \x  &&&& \x && \x & \x\\ \hline
   &\x &\x &    &\x  & &&& \x &&& \\ 
   &   &   &\x &    &      \x & \x & & & \x&\x &  & \x&&& \\ 
   &   &   &\x &    &      \x & \x & & & \x&\x &  & \x&&& \\ 
   &&&&&&& \x  &&&& \x && \x & \x\\ \hline
   &   &   &\x &    &      \x & \x & & & \x&\x &  & \x&&& \\    
   &&&&&&& \x  &&&& \x && \x & \x\\ 
   &&&&&&& \x  &&&& \x && \x & \x\\ 
   &&&&&&&&&&&&&&&\x 
\ttam 
\eq 
The natural comparison here is with $F(1)=1$ at rank $N=4$, which is completely understood.
In this case, for comparison, the shape is

\newcommand{\xxx}{} 

\[
\matt{cccc|cccc|cccc|cccc} 
\x &   &   &   &   &&&&&&&&   \\
   &\x &\xxx &   & \x &   &&& \xxx &&& \\
   &\xxx &\x &   & \xxx &   &&& \x &&& \\
   &   &   &\x &    &      \xxx & \xxx & & & \xxx&\xxx &  & \x&&& \\ 
\hline 
   &\x &\xxx &   & \x &   &&& \xxx &&&  \\
   &   &   &\xxx &    &      \x & \xxx & & & \xxx&\xxx &  & \xxx&&& \\ 
   &   &   &\xxx &    &      \xxx & \x & & & \x&\xxx &  & \xxx&&& \\ 
   &&&&&&& \x  &&&& \xxx && \x & \xxx\\ 
\hline
   &\xxx &\x &    &\xxx  & &&& \x &&& \\ 
   &   &   &\xxx &    &      \xxx & \x & & & \x&\xxx &  & \xxx&&& \\ 
   &   &   &\xxx &    &      \xxx & \xxx & & & \xxx&\x &  & \xxx&&& \\ 
   &&&&&&& \xxx  &&&& \x && \xxx & \x\\ 
\hline
   &   &   &\x &    &      \xxx & \xxx & & & \xxx&\xxx &  & \x&&& \\    
   &&&&&&& \x  &&&& \xxx && \x & \xxx \\ 
   &&&&&&& \xxx  &&&& \x && \xxx & \x \\ 
   &&&&&&&&&&&&&&&\x 
\ttam 
\]

Note that in this case $F(1)=1$ lies `inside' $F(1)=2$ but is substantially smaller.}}

\section{Categories of Functors from $\Bcat$}

Recall from  
Section \ref{S:intro} that our aim  
\ppmm{in principle is to classifying strict monoidal}
functors from $\Bcat$ to $\Mat$. 
This immediately raises three questions that we  
address  further 
in this Section: 
Why restrict to strict monoidal functors?;
What notion of equivalence should classification be up to?; and 
How to overcome the intractability of the problem?

\ppmm{As discussed in \S\ref{ss:subcats}, to sidestep the intractability of the problem, we have}  
in mind the possibility of   
\ppmm{restricting the target $\Mat$ to} 
something 
for which these functors are 
more amenable to classification - motivated 
by the success of this 
tactic  
for example in \cite{MR1X}.

{First let us explain why there is no point in considering the problem with the monoidal structures forgotten.}
The category $\Fun(\Bcat,\Mat)$ of  functors from $\Bcat$ to $\Mat$ with morphisms being the natural transformations is too general
to be interesting: 
such a functor is 
{simply a free choice of a representation} 
$\rho_n$ for each $n$, and the natural transformations are intertwining maps.  Since there are no morphisms in $\Bcat$ between distinct objects, one does not have any cohesion between the $\rho_n$.  
We {avoid} this shortcoming 
{(the requirement effectively being to classify all braid {\em group} representations)}
by requiring the functors to respect the monoidal structure.

{So then the next question is the {\em extent} to which the monoidal structure should be respected (specifically, why {\em strict} monoidal functors?).}

Theorem \ref{th:strictifyingmonfuns} applies to the source category $\Bcat$, since the object monoid $\ob(\Bcat)=\N_0$  is free on $\{1\}$.  
So, up to monoidal natural isomorphism, we lose nothing by restricting to strict monoidal functors. 

\medskip 

Next we turn to the question of good notions of equivalence. 
Here we will articulate this in 
terms of morphisms in a category having our braid representations as objects.  {From a representation-theoretic point of view, the following is appropriate:}

\mdef \label{de:MoonFun}
Fix  strict monoidal categories $\Bb$ and $\catC$. 
The category 
$\MoonFun(\Bb,\catC)$ is the category of \emph{strict} monoidal functors $F$ from $\Bb$ to $\catC$, and morphisms the  natural transformations $\eta:F\rightarrow G$ for $F,G\in\MoonFun(\Bb,\catC)$. 

Now consider in particular  functors $F,G \in \MoonFun(\Bcat,\Cc)$. 
Suppose that $F(1)=X$ and $F(\sigma)=\chi\in\catC(X^{\otimes 2},X^{\otimes 2})$ and similarly $G(1)=Y$ and $G(\sigma)=\gamma$.  These completely determine $F$ and $G$.  How does one specify a natural transformation $\eta:F\rightarrow G$?  
By Definition \ref{def:natural transf} 
its components are  morphisms $\eta_n\in\catC(X^{\otimes n},Y^{\otimes n})$ such that $\eta_nF(\beta)=G(\beta)\eta_n$ for $\beta\in B_n$. 

\ignore{{For classification one is interested in particular in the natural isomorphisms, which give us one useful notion of equivalence for braid representations 
- and quite generally for the objects in any functor category.
Note that this 
\ppm{[-what is `this' here?]}
gives us a natural hierarchy of, in-principle weaker 
than natural isomorphism 
(i.e. more inclusive),  
equivalences among such braid representations. For each $p\in\N$ we say two braid reps are $p$-equivalent {\em in} $\MoonFun(\Bcat,\catC)$  if there exist morphisms $\eta_n$ as above for all $n \leq p$.
For example with $\catC=\Mat$ then $F,G$ are 1-equivalent here if $F(1)=G(1)$; and the condition for 2-equivalence here is that $F(\sigma)$ and $G(\sigma)$ are equivalent matrices. 
(We will restate this in more rarefied language in \S\ref{ss:p}.)\ecr{[in fact we define it again there.  To me this feels out of place here, especially if we are thinking about space.]}
}}

\mdef \label{de:MonFun} We denote by $\MonFun(\Bcat,\catC)$  the subcategory 
with 
morphisms the {\em monoidal} natural transformations as in Definition \ref{de:mon nat tranf}.  Since both $\Bcat$ and $\catC$ are strict this means that $\eta:F\rightarrow G$ satisfies \[\eta_{m+n}=\eta_m\otimes\eta_n.\] This is a powerful  constraint: in particular $\eta$ is determined by $\eta_1:F(1)\rightarrow G(1)$.  Conversely, any $f\in\catC(F(1),G(1))$ that satisfies $(f\otimes f)F(\sigma)=G(\sigma)(f\otimes f)$ uniquely determines a morphism in $\MonFun(\Bcat,\catC)$ by defining $\eta_n=f^{\otimes n}$.

\ignore{{\ecrb{\begin{proposition}\label{prop:monoidal strictification}
    Let $F:\Bcat\rightarrow\Mat$ 
    \sout{be} \ppm{give} 
    a monoidal functor \ppm{$(F,J)$ as in \ref{def:monoidal functor}}
    with coherence maps $J_{m,n}:F(m)\otimes F(n)\cong F(m+n)$ and $J_0:\unit\cong F(0)$.  
    \sout{Set $F(1)=N$.}  \ppm{[-backwards? and not needed in this form of Prop.]}
    Then there exists a strict monoidal functor $G:\Bcat\rightarrow \Mat$ and a monoidal natural isomorphism $\eta:F\Rightarrow G$.
\end{proposition}
}
\begin{proof}
\ecrb{ 
\ppm{Set $N = F(1)$. Thus $F(n)  = N^n$.}
    We first define morphisms $\eta_n:F(n)\cong N^n$ recursively by $\eta_0:=J_0,\ecr{\text{ actually } J_0^{-1}?} \; \eta_1:=\id_N$ and 
    \begin{equation}  \label{eq:defetan}
    \eta_{n+1} := (\eta_n\otimes \eta_1)\circ J_{n,1}^{-1}
    \end{equation} 
    for $n\geq 0$ \ecr{maybe $n\geq 1$} (any isomorphism $\eta_1$ will do). 
    \ppm{For example $\eta_2 = (\id_N \otimes \id_N) \circ J^{-1}_{1,1}$  [...sorry I am being slow here! I'll get there eventually.] - the second subscript in $ J^{-1}_{1,1}$ is not $1_\Bb = 0$ [er, right?] so ...  
    Next, 
    $\eta_3 =(((\id_N \otimes \id_N ) \circ J^{-1}_{1,1})\otimes \id_N)\circ J^{-1}_{2,1}$}
    \medskip
    \\ 
    Next we define a functor $G:\Bcat\rightarrow \Mat$ by $G(n)=N^n$ and $G(\beta)=\eta_nF(\beta)\eta_n^{-1}$ for $\beta\in B_n$.  To see that is is a functor we note that conjugation preserves compositions.  To verify that $G$ is strict monoidal and to assemble the $\eta_n$ into a monoidal natural isomorphism we need the following:
    \begin{equation}\label{eq:mon nat}
        \eta_{m+n}=(\eta_m\otimes \eta_n)\circ J_{m,n}^{-1}
    \end{equation} for all $n,m\geq 0$ 
    \ppm{(in other words, for all $l$ and all 2-part compositions of $l=m+n$ we have $\eta_l = \eta_{m+n} = (\eta_m \otimes \eta_n)\circ J^{-1}_{m,n}$)}
    which we prove by induction on $n$ for an arbitrary fixed $m\geq 0$.  
    The $n=0$ case 
    \ppm{requires us to rewrite $\eta_m$. We have    
    $$ 
    \eta_m = \eta_m \otimes \id_{1_\Bcat} = (\eta_m\otimes J_0)\circ (\id_{N^m}\otimes J_0^{-1}) 
    \hspace{1.5in} $$ $$ \hspace{1.5in}  
    \stackrel{\ref{rem:J0}}{=}   (\eta_m\otimes J_0)\circ J^{-1}_{m,1_\Bb}  = (\eta_m \otimes J_0) \circ J^{-1}_{m,0} = (\eta_m \otimes \eta_0) \circ J^{-1}_{m,0} 
    $$ 
(recall we have a momentarily awkward notation so that $\One_\Bcat =0 \neq 1$)
    as required.   Indeed this is essentially vacuous.    
}}

\ecrb{ 
    The $n=1$ case is the definition of $\eta_n$ from (\ref{eq:defetan}).  Suppose this hold for some $n\geq 1$, and consider $\eta_{m+n+1}$: using the definition of $\eta$, the induction hypothesis and the coherence equation \eqref{eq:Jcoherence} we calculate:
    \begin{eqnarray*}
    && \eta_{m+n+1}=   (\eta_{m+n}\otimes \eta_1)\circ J_{m+n,1}^{-1}=\\
    && ([(\eta_m\otimes \eta_n)\circ J_{m,n}^{-1}]\otimes \eta_1)\circ J^{-1}_{m+n,1}=(\eta_m\otimes \eta_n\otimes \eta_1)\circ (J_{m,n}^{-1}\otimes\id_N)\circ J_{m+n,1}^{-1}\\
    &&(\eta_m\otimes \eta_n\otimes \eta_1)\circ (\id_{N^m}\otimes J_{n,1}^{-1})\circ J_{m,n+1}^{-1}=(\eta_m\otimes\eta_{n+1})\circ J_{m,n+1}^{-1}
    \end{eqnarray*} as required.
Now we can see that $G$ is strict monoidal, using the naturality of $J_{m,n}$: \begin{eqnarray*}
&&G(\alpha\otimes\beta)=\eta_{m+n}F(\alpha\otimes\beta)\eta_{m+n}^{-1}=\\
   && \eta_{m+n}(J_{m,n}F(\alpha)\otimes F(\beta)J_{m,n}^{-1})\eta_{m+n}^{-1}=\eta_m\otimes \eta_n F(\alpha)\otimes F(\beta) (\eta_m\otimes \eta_n)^{-1}\\
   && \hspace{2in}=G(\alpha)\otimes G(\beta).
\end{eqnarray*}  
Now by construction $\eta$ defines a natural isomorphism from $F$ to $G$, so we only need to verify that it is monoidal.
For this we note that \eqref{eq:mon nat} is precisely the monoidal condition in Definition \ref{de:mon nat tranf}.   }
\end{proof}
}}
  
  \ignore{{
  , as they are only constrained by $F(\beta)\eta_n=\eta_nG(\beta)$ for $\beta\in B_n$.  In particular there is no \emph{a priori} relationship among the $\eta_n$ for distinct $n$, so one must provide a such a morphism $\eta_n:F(\beta)\rightarrow G(\beta)$ for each $n$. Thus while is often difficult to determine whether or not two functors are naturally isomorphic (when they are not already monoidally naturally isomorphic), there are examples of such situations, see, eg. Example \ref{ex:9x9unitaries} and Theorem \ref{th:ds is infinity}. }}

\ignore{{
\ecr{[kill this remark.]}
\begin{remark}  \label{rem:restrict2strict}
  What do we lose by restricting to strict monoidal functors?  In the case $\catC=\Mat$ we have from Def.~\ref{def:monoidal functor} that a general monoidal functor $(F,J):\Bcat\rightarrow \Mat$ has $F(1)=N$ and $F(\sigma)=R$ with 
  the infinite set of invertible matrices $J_{m,n}:N^m\cdot N^n\rightarrow N^{m+n}$ 
  satisfying:
    \[
J_{m+n,k}(J_{m,n} \otimes \mathrm{Id}_{N^k}) = J_{m,n+k}(\mathrm{Id}_{N^m} \otimes J_{n,k}) \; \forall m,n,k.
\] 
This is all one needs, all other images are determined in terms of $F(\sigma)$ and the $J_{m,n}$, e.g. for $\sigma_1\in \Bcat(3,3)$,
$F(\sigma_1)=J_{2,1}(F(\sigma)\otimes \id_N )J_{2,1}^{-1}$.  Of course  finding non-trivial solutions to this infinite system of matrix equations could be tedious.}}

\ignore{{\ft{[At the moment we aren't really using that $(F,J)$ is connected to a strict monoidal functor via a natural transformation. Of course this turns out to not help, but a priori it could, and I think the fact that it doesn't is exactly what we want to convey in this remark.]}\ecr{[feel free to add to this remark.  I am not sure how to say what you want to convey.]}
\end{remark}

}}

\subsection{The categories $\YB(\catC)$} 
\label{ss:YB}

We continue with $\catC$ a strict monoidal category. 
We can construct a category isomorphic to $\MonFun(\Bcat, \catC)$ shifting the focus from functors as objects to braidings on objects 
in $\catC$ via the category $YB(\catC)$ defined as follows (see e.g. \cite{street2012,kassel-}
for proof of well-definedness):

\begin{defin}\label{de:YBC}
Fix a strict monoidal category $\catC$. 
(I) A \emph{\YangBaxterobject} of $\catC$ is 
a pair $(X,\chi)$ where $X\in Obj(\catC)$ and $\chi\in \catC (X\otimes X,X\otimes X)$ is an invertible morphism satisfying the Yang-Baxter equation in 
$\catC(X^{\otimes 3}, X^{\otimes 3})$ cf. \eqref{eq:YBE}: 
\begin{equation}\label{def:ybobj}
    (\chi \otimes \id_X) (\id_X\otimes \chi )(\chi \otimes \id_X) 
= (\id_X\otimes \chi ) (\chi \otimes \id_X) (\id_X\otimes \chi ).
\end{equation}

\ignore{{In other words $(X,\chi)$ is a \YBo\ iff $\chi$ gives a braid representation. \ecr{[a braid representation is a functor?  $\chi$ is not.]} \ppm{-but it gives one, right? Not sure why we are hiding this.}}}

(II)
The category 
$YB(\catC)$ has 
\ppmm{Yang--Baxter objects in $\catC$ as objects; and }
 morphism sets $YB(\catC) ( (X,\chi), (Y,\gamma)   ) \subseteq \catC(X,Y)$ 
consisting of those $f\in\catC(X,Y)$
such that 
\beq   \label{eq:intertwine}
\gamma(f\otimes f)=(f\otimes f)\chi   .
\eq 
\end{defin}

\ignore{{
\ppm{[paul q to self: is it defn or prop/defn? a set and a set of sets dont always make a cat.]}
\end{defin}

\mdef   \label{de:YBo}
Recall from the Introduction, or from 
(\ref{de:Bcat}),
that the monoid $\ob(\Bcat)$ is generated by 1; and the category is monoidally generated by $\sigma$ (and inverse).  
Hereafter we sometimes refer to a braid representation $F:\Bcat\rightarrow \catC$, 
i.e. an object in $\MonFun(\Bcat,\catC)$,
when expressed as the pair $(F(1),F(\sigma))$,
i.e. as 
an object in $YB(\catC)$,
as a {\em \YangBaxterobject}.
\ppm{[-All OK now? Then we can delete the next para?]}
\\
\ppm{[-here is a (meta) proposal: how about we leave terms/notations used in Kassel to have the meanings as in Kassel, and modify the term a bit somehow if we use it for something slightly different. So for example leave Y-B object to have the meaning as in Kassel; and perhaps also have YB(C) as in Kassel. This might mean quite a lot of superficial changes in the doc here, but easy ones, and with the obvious reward of less chance of confusion? OTOH, perhaps there are a couple of lemmas that are true that make everything consistent - also happy to go that way.]}

\medskip 

It is shown in \cite[Lemma XIII.3.5]{kassel-} 
 that from a Yang-Baxter object 
 $(X,\chi)\in YB(\catC)$ 
 one obtains a strict monoidal functor $F\in\Ob(\MonFun(\Bcat,\catC))$ 
  {by setting} $F(1)=X$ and $F(\sigma)=\chi$.  
 \ppm{[-(guessing some meanings above) isn't this direction just obvious from \ref{de:YBC} and properties of $\Bcat$ already established? the morphism aspect is maybe harder?:] We claim also that given $f \in YB(\catC)(....)$ then $f\otimes f$ intertwines the two braid reps. \ppm{[Nope, this direction again seems fairly obvious.]}}

\mdef \label{pa:MF-YB}
    In \cite[Theorem XIII.3.3]{kassel-} it is shown that $\MonFun(\Bcat,\catC)$ and $YB(\catC)$ are  equivalent \ft{isomorphic?}\ecr{Kassel shows equivalence (fully faithfully, essentially surjective.)} categories.  
    \ppm{[-the deep thing here is that all natural transformations can be expressed in such a simple way?]}
\\
\ppm{[One problem for me here is that the cited ref deals with a different case  - a weaker assumption on $\catC$, and only asserts equivalence. Can we have a discussion?]}

\mdef 
In fact, in     \cite[Theorem XIII.3.3]{kassel-}   
it is shown that the category of monoidal functors from $\Bcat$ to another (not-necessarily strict) monoidal category $\catC$ is equivalent \ft{isomorphic?} to the category of Yang-Baxter objects in $\catC$, 
meaning objects which satisfy a non-strict version of \eqref{def:ybobj};
and their `intertwiners', meaning an up-to-isomorphism version of (\ref{eq:intertwine}).  
\ppm{[-just so I understand - this is irrelevant for us, right? (not advocating to omit it, just trying to check my understanding.)]}

It is possible to derive the special case of Proposition \ref{pr:strictifyingmonfuns} with source $\Bcat$ by combining the results of \cite[Section XIII.3]{kassel-}.  In particular, we see that to specify an $F\in\MonFun(\Bcat,\catC)$ is to present a pair $(X,\chi)\in YB(\catC)$.  
}}

\mdef  \label{pa:identifyYBMonFun}
\emph{From now on, we use $YB(\catC)$ and $\MonFun(\Bcat,\catC)$ interchangeably.}

\ignore{{
    \ppm{[If it is only equivalence then are there not situations in which interchanging could be dangerous? Certainly we will only do it when it is safe to do so, but should we say a tad more here?]}\ecr{[I believe Fiona will make this clear.  The point is that we have \emph{defined} $\MonFun$ to be strict, and so it is actually an isomorphism.  equivalence does not require strictness.]}}}

\medskip

This 
identification
is justified 
as follows. 
\\ 
Firstly from \cite[Section XIII]{kassel-} it can be shown that:
For any monoidal category $\catT$ (strict or otherwise) there is a  \emph{equivalence} between: 1) the category of monoidal functors from $\Bcat$ to $\catT$ and 2) the category of Yang-Baxter objects of $\catT$, with the obvious non-strict generalisation of Yang-Baxter objects.  

For a strict monoidal category $\catC$ observe that the above equivalence can be promoted to an isomorphism via 
a functor 
$\Xi:\MonFun(\Bcat,\catC)\rightarrow YB(\catC)$ 
given by $\Xi(F)= (F(1),F(\sigma))$ on objects and $\Xi(\eta)=\eta_1$ on morphisms. The inverse functor $\Xi^{-1}$ may be constructed on objects using \cite[Lemma XIII.3.5]{kassel-}, and the characterisation of monoidal natural transformations in $\MonFun(\Bcat,\catC)$  from \ref{de:mon nat tranf} above.

\subsection{More on the category $\YB(\Mat)$}\label{why not mon}

In much of what follows we focus on developing the category $\MonFun(\Bcat,\Mat)\cong YB(\Mat)$.

{Notationally, since we mainly consider $\YB(\catC)$ in case $\catC = \Mat$, we will simply write 
$\EndYB(N,R)$ for $\End_{\YB(\Mat)}(N,R)$ as in (\ref{eq:deEnd}); and similarly for $\Aut_{\YB}$ and so on. 
We will also write $\Aut(N)$ for $\Aut_{\Mat}(N)$.} 

The isomorphism class of an object $(N,R)$ in $YB(\Mat)$ is precisely \[\{(N,Q^{\otimes 2}R(Q^{-1})^{\otimes 2}):Q\in \Aut(N) \}.
\] 
Similarly the endomorphisms $Q \in \EndYB(N,R)$ are the $Q\in\Aut(N)$ such that $Q\otimes Q$ commutes with $R$.

To specify a morphism $(M,S)\rightarrow (N,R)$ 
\ppmm{in $\YB(\Mat)$}
we only need a single matrix $A\in\Hom(M,N)$ 
\ppmm{$=\Mat(M,N)$}
with $(A\otimes A)S=R(A\otimes A)$.  If $R$ and $S$ are given it is computationally straightforward to find all such $A$.  This is explored further in section \ref{ss:subobjects}.

\subsection{More on the 
category $\MoonFun(\Bcat,\Mat)$}\label{why not moon}

In much of what follows we focus on 
$\MonFun(\Bcat,\Mat)\cong YB(\Mat)$, at the expense of neglecting $\MoonFun(\Bcat,\Mat)$.  Let us pause to discuss a few differences.

While any morphism in 
$\YB(\Mat)$ 
remains a morphism in $\MoonFun(\Bcat,\Mat)$, a general morphism is somewhat more involved to describe. Indeed, suppose $\eta:F\rightarrow G$ is an isomorphism in $\MoonFun(\Bcat,\Mat)$.  Define $R=F(\sigma)$ and $S=G(\sigma)$.  
Then $F(1)=G(1)$, 
and $\eta_2R\eta_2^{-1}=S$.  
Setting $N=F(1)$, 
we also have \[\eta_3(R\otimes \id_N)\eta_3^{-1}=(S\otimes \id_N)=(\eta_2R\eta_2^{-1}\otimes \id_N)\]  so that $(\eta_2^{-1}\otimes \id_N)\eta_3$ commutes with $(R\otimes\id_N)$ and analogously $(\id_N\otimes\eta_2^{-1})\eta_3 $ commutes with $(\id_N\otimes R)$. The $\eta_n$ for $n>3$ \ppmm{obey similar identities}.

Thus, to compute the isomorphism class of $F$ 
in  $\MoonFun(\Bcat,\Mat)$ 
one  
could
first find all matrices $\eta_2\in\Aut(N^2)$ such that $\eta_2R\eta_2^{-1}$ satisfies the Yang-Baxter equation.  Then 
\ppmm{for each such $\eta_2$}
find \ppmm{an} $\eta_3$ constrained as above (if it exists) and continue, to find \ppmm{an} $\eta_n$ for each $n$.  
If such a sequence can be found, then  
the representation given by  $\eta_2R\eta_2^{-1}$ is in the same class as $F$.
\\ 
This may turn out to be a finite problem for a given $R$, but we do not know.  We do construct some isomorphisms in $\MoonFun(\Bcat,\Mat)$ that are not already isomorphisms in $YB(\Mat)$, although this construction method uses $\EndYB(N,R)$, see Theorem \ref{th:ds is infinity}.  Even to describe the endomorphisms $\eta$ of $F\in\MoonFun(\Bcat,\Mat)$ with $F(\sigma)=R$ we must determine all $\eta_n$ that self-intertwine $\rho_n^R$. This amounts to understanding the decomposition of $\rho_n^R$.  These tasks can be further complicated or simplified if we restrict the target of $F$, see  
e.g. \cite{LechnerPennigWood} summarised in section \ref{sss: invol ybo}.

\section{Structure of 
$\MonFun(\Bcat,\catC)\cong YB(\catC)$ }\label{sec: structure}

 In this section standard categorical concepts, structures and properties such as monoidal product, subobjects, quotient objects, and rigidity will be explored.  We will focus on the case of main interest to us, target $\catC=\Mat$, only using the general case when appropriate. Crucially, the lashing product  
(Definition \ref{de:lashing})
 renders the category $YB(\Mat)$ a monoidal category, so that we may identify monoidal subcategories amenable to classification as in \sproblem \ref{prob:2}. {For example, one could fix a particular object $(N,R)\in YB(\Mat)$ and try to classify the monoidal subcategory it generates.}

\subsection{Cabling Products}\label{ss:cabling}

In this section we review some ways of obtaining new objects in $YB(\catC)$ from given ones, at first for general monoidal categories $\catC$ and then for braid representations, i.e., for $\catC=\Mat$.  
{Nominally this section is off-topic, being construction rather than classification. But any construction which takes braid representations as input and produces braid representations as output allows us to ask, for example, the question: 
What minimal subset of braid representations generates all braid representations by application of this construction? (A notable example would be the notion of direct sum in artinian algebraic representation theory, allowing us to classify representations by classifying indecomposable representations.) }
Another perspective is that rather than being construction, it (and by `it' we now mean a certain form of cabling) is structure: indeed a key observation is that it provides the structure of a monoidal category on $YB(\catC)$ when $\catC$ is braided, crucially the case for $\catC=\Mat$.
A bigger issue with `reviewing some ways' is that it involves a choice of `ways' to review, without providing a value system for comparing and contrasting with other possible ways. From this perspective, constructing using direct sum in rep theory is easily verified to be crucial, and indeed canonical. In the present case, however, our justification lies mainly with local utility (as we shall see); and the fact that higher rep theory is generally harder and less canonical than ordinary rep theory, so we have to start somewhere! With these caveats, we now proceed.

\mdef  \label{pr:cable}
Fix a strict monoidal category $\catC$ and let $(X,\gamma)\in YB(\catC)$ be a \YangBaxterobject, as defined in \ref{de:YBC}, 
and $k\geq 1$ an integer.
Inspired \ppmm{by} the classical cabling of braids \ignore{{(see e.g. \cite[p.164]{CrowellFox63})}} there is a construction (see e.g. \cite{Wenzl90PJM}) of another \YangBaxterobject\ $(X^{\otimes k},\gamma^{(k)})$ by $k$-cabling $\gamma$.

\medskip 

For example, suppose $(X,\gamma)\in YB(\catC)$. 
 Define $\gamma_1=\gamma\otimes \id_X\otimes \id_X, \gamma_2=\id_X\otimes \gamma\otimes \id_X$ and $\gamma_3=\id_X\otimes \id_X\otimes \gamma$. Then we have the $2$-cabling $\gamma^{(2)}:=\gamma_2\gamma_1\gamma_3 \gamma_2$. See Example \ref{ex:projection} below for an instance of this.

\medskip 

Suppose in addition $\catC$ is itself braided \cite{JoyalStreet}, i.e.,  equipped with a natural isomorphism $c_{X,Y} \colon X\otimes Y\cong Y\otimes X$ satisfying (the hexagon) equations:
\[c_{X,Y\otimes Z}=(\id_Y\otimes c_{X,Z})(c_{X,Y}\otimes \id_Z)\]

\[c_{X\otimes Y, Z}=(c_{X,Z}\otimes \id_Y)(\id_X\otimes c_{Y,Z})\]

Then we can define another type of product on objects in $YB(\catC)$, blending two Yang-Baxter objects by means of the categorical braiding.\footnote{This procedure will not run afoul of Deuteronomy 22:11 as long as we don't use it to make clothes.}

\begin{defin}   \label{de:lashing}
Let $\catC$ be a strict braided monoidal category with braiding $c$ in the sense of \cite[Chapter XI]{MacLane}.  We define the \emph{lashing product} of a pair of objects $(X,\chi)$, $(Y,\gamma)\in YB(\catC)$
  to be  $(X\otimes Y,\chi\boxtimes \gamma)$ where 
\[\chi\boxtimes\gamma:=(\id_X\otimes c_{X,Y}^{-1}\otimes \id_Y)(\chi\otimes \gamma)(\id_X\otimes c_{Y,X}\otimes \id_Y)\in\catC((X\otimes Y)^{\otimes 2},(X\otimes Y)^{\otimes 2}).\]
Reading from bottom to top, $\chi\boxtimes\gamma$ can be visualised as, the following, where the under/over crossings represent $c_{Y,X}$ and $c_{X,Y}^{-1}$:

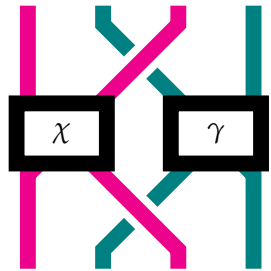
\begin{center} \begin{tikzpicture}[
braid/.cd,
number of strands=4,
line width=6pt,
strand 1/.style={magenta},
strand 2/.style={teal},
strand 3/.style={magenta},
strand 4/.style={teal},
gap=0.1,
control factor=0,
nudge factor=0,
name=symbols,
] 

\pic at (1,0) {braid={s_2^{-1}s_1-s_3 s_2}};
\node[draw, fill=white, minimum width=12mm, minimum height=8mm] 
        at (1.45, -1.7) {$\chi$};
\node[draw, fill=white, minimum width=12mm, minimum height=8mm] 
        at (3.50, -1.7) {$\gamma$};
\end{tikzpicture}
\end{center}

\end{defin}

\begin{theorem}\label{th:YBMonoidal}
   For any braided strict monoidal category $\catC$, $YB(\catC)$ is a strict monoidal category under the lashing product, with monoidal unit $(\unit,\id_\unit)$.
\end{theorem}
\begin{proof}  The functoriality of the braiding $c$ can be used to verify the Yang-Baxter equation (\ref{eq:YBE}) as follows, which shows the lashing product is a bifunctor on $YB(\catC)$:
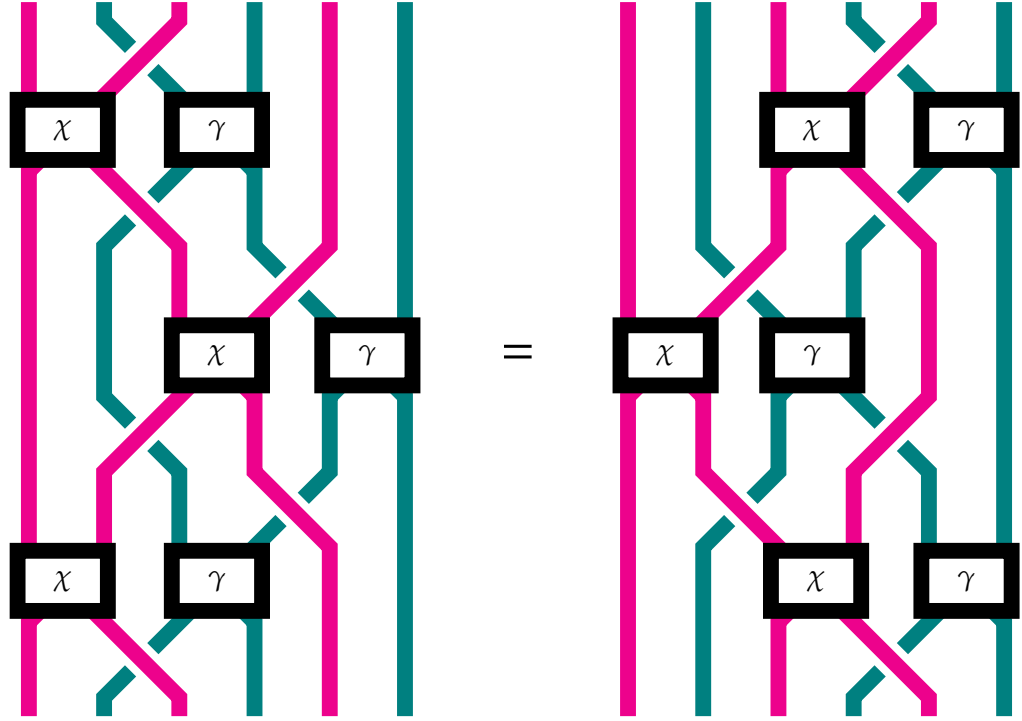
\begin{center}\begin{tikzpicture}[
braid/.cd,
number of strands=6,
line width=6pt,
strand 1/.style={magenta},
strand 2/.style={teal},
strand 3/.style={magenta},
strand 4/.style={teal},
strand 5/.style={magenta},
strand 6/.style={teal},
gap=0.1,
control factor=0,
nudge factor=0,
name=symbols,
] 
\pic {braid={s_2^{-1}s_1-s_3 s_2s_4^{-1}s_3-s_5s_2^{-1}s_4 s_1-s_3s_2}};
\node[draw, fill=white, minimum width=12mm, minimum height=8mm] 
        at (0.45, -1.7) {$\chi$};
\node[draw, fill=white, minimum width=12mm, minimum height=8mm] 
        at (2.50, -1.7) {$\gamma$};

\node[draw, fill=white, minimum 
width=12mm, minimum height=8mm] 
        at (0.45, -7.7) {$\chi$};
\node[draw, fill=white, minimum 
        width=12mm, minimum height=8mm] 
        at (2.50, -4.7) {$\chi$};

        \node[draw, fill=white, minimum width=12mm, minimum height=8mm] 
        at (4.50, -4.7) {$\gamma$};

               \node
        at (6.50, -4.7) {\huge{=}};

        \node[draw, fill=white, minimum 
        width=12mm, minimum height=8mm] 
        at (2.50, -7.7) {$\gamma$};

\end{tikzpicture}$\quad\quad$
\begin{tikzpicture}[
braid/.cd,
number of strands=6,
line width=6pt,
strand 1/.style={magenta},
strand 2/.style={teal},
strand 3/.style={magenta},
strand 4/.style={teal},
strand 5/.style={magenta},
strand 6/.style={teal},
gap=0.1,
control factor=0,
nudge factor=0,
name=symbols,
] 

\pic at (7,0) {braid={s_4^{-1}s_3-s_5 s_4s_2^{-1}s_1-s_3s_4^{-1}s_2 |s_3-s_5s_4}};

\node[draw, fill=white, minimum width=12mm, minimum height=8mm] 
        at (9.45, -1.7) {$\chi$};
\node[draw, fill=white, minimum width=12mm, minimum height=8mm] 
        at (11.50, -1.7) {$\gamma$};

\node[draw, fill=white, minimum 
width=12mm, minimum height=8mm] 
        at (9.5, -7.7) {$\chi$};
\node[draw, fill=white, minimum 
        width=12mm, minimum height=8mm] 
        at (7.50, -4.7) {$\chi$};

        \node[draw, fill=white, minimum width=12mm, minimum height=8mm] 
        at (9.45, -4.7) {$\gamma$};

        \node[draw, fill=white, minimum 
        width=12mm, minimum height=8mm] 
        at (11.50, -7.7) {$\gamma$};

\end{tikzpicture}
\end{center}

 For $(X,\chi),(Y,\gamma)$ and $(Z,\zeta) \in YB(\catC)$ one must check that $(\chi\boxtimes\gamma)\boxtimes \zeta =\chi\boxtimes(\gamma\boxtimes \zeta)$.  This is another exercise in diagrammatic calculus, which relies crucially on the functoriality of the braiding in $\catC$.
\end{proof}

\mdef \label{ybmat is braided}
If $\catC$ is \emph{symmetrically} braided, i.e., $c_{X,Y}^{-1}=c_{Y,X}$ then $YB(\catC)$ is itself braided: the braiding \[A_{(X,\chi),(Y,\gamma)}:(X\otimes Y,\chi\boxtimes \gamma)\cong (Y\otimes X,\gamma\boxtimes\chi)\] is the morphism $c_{X,Y}=c_{Y,X}^{-1}\in\catC(X\otimes Y,Y\otimes X)$.  One must check that \[(c_{X,Y}\otimes c_{X,Y})\chi\boxtimes \gamma =\gamma\boxtimes \chi(c_{X,Y}\otimes c_{X,Y}).\] 
This is again most easily seen by drawing the picture: the morphism $c_{X,Y}\otimes c_{X,Y}$ conjugates $\chi\otimes \gamma$ over to a morphism with the over/under crossings switched and the $\chi$ and $\gamma$ interchanged.  But since $c$ is symmetric this is the same as $\gamma\boxtimes\chi$. 

\mdef \label{ybmat is monoidal} The category $\Mat$ introduced in (\ref{de:Mat}) has the structure of a (symmetric) braided monoidal category, as is well-known: the braiding is the standard flip $P_{N,M}:N\otimes M\rightarrow M\otimes N$. { This implies that if $A\in\Mat(M_1,M_2)$ and $B\in\Mat(N_1,N_2)$ then we have, by functoriality, \[P_{M_2,N_2}(A\otimes B)=(B\otimes A)P_{M_1,N_1}.\]} Thus for any  $(N,R),(M,S)\in YB(\Mat)$ one may define $(N,R)\boxtimes (M,S)=(NM,[R\boxtimes S])$ as the lashing product coming from $P_{N,M}$. This particular case of lashing has been studied recently in \cite{Chouraqui}, where it is called the Tracy-Singh product.

\begin{example}\label{ex: gaussian factorisation}
    In \cite{GalindoRowell} the Gaussian solutions to the Yang-Baxter equation were studied, following \cite{GoldschmidtJones1989}. These are unitary YBOs arising from quantum tori with $q$ a $d$th root of unity. \cite[Lemma 3.5]{GalindoRowell} shows that the Gaussian YBOs admit a kind of prime-power factorisation. The lashing product gives an interpretation of this factorisation, which we illustrate in a concrete example.
    
Define (the qudit Pauli $X$ and $Z$ gates):
\[
  X_d e_j = e_{j+1}, \qquad
  Z_d e_j = \zeta_d^j e_j,
\]
 with indices taken modulo \(d\) and where \(\zeta_d=e^{2\pi i/d}\). 
Set
\(
  P_d = X_d \otimes Z_d. 
\)

\smallskip
We identify \(\C^6\cong \C^2\otimes \C^3\) by the Chinese remainder map
\[
C:  e_j \mapsto e_{j\bmod 2}\otimes e_{j\bmod 3}.
\]

Define \( Q_3=X_3\otimes Z_3^2=\overline{P_3}\),  \(\omega=\zeta_3\) and  \(\eta=\zeta_{12}\). We have
\[
  R=I-iP_2,
  \qquad
  S=I+\omega Q_3+\omega Q_3^2
\] are YBOs, i.e., $(2,R),(3,S)\in\YB(\Mat)$.
The \(d=6\) Gaussian solution
\[
  G=\sum_{k=0}^{5}\eta^{k^2}P_6^k 
\]
 is the lashing product of $R$ and $S$ after using the CRT map above, i.e. $(6,G)\cong(2,R)\boxtimes (3,S)$ via $(C\otimes C)G(C\otimes C)^{-1}=R\boxtimes S.$
\ignore{{
\begin{proof}
Write
\[
  R=\sum_{a=0}^1 r_a P_2^a,
  \qquad
  S=\sum_{b=0}^2 s_b Q_3^b,
\]
where
\[
  (r_0,r_1)=(1,-i)=(1,\eta^9),
  \qquad
  (s_0,s_1,s_2)=(1,\omega,\omega)=(1,\eta^4,\eta^4).
\]
The lashing product is, after the middle-factor swap,
\[
  R_2^-\otimes S_3^{\mathrm{tw}}
  =
  \sum_{a=0}^1\sum_{b=0}^2 r_a s_b\, P_2^a\otimes Q_3^b.
\]
Since \(P_6\) corresponds to \(P_2\otimes Q_3\), its \(k\)-th power corresponds to
\[
  P_2^k\otimes Q_3^k.
\]
The exponent \(k\in\{0,\dots,5\}\) is determined by the congruences
\[
  k\equiv a\pmod 2,
  \qquad
  k\equiv b\pmod 3.
\]
Therefore the coefficient of \(P_6^k\) is
\[
  c_k=r_{k\bmod 2}s_{k\bmod 3}.
\]
This gives
\[
  (c_0,c_1,c_2,c_3,c_4,c_5)
  =
  (1,-i\omega,\omega,-i,\omega,-i\omega).
\]
Since \(i=\eta^3\) and \(\omega=\eta^4\), this is
\[
  (1,\eta,\eta^4,\eta^9,\eta^4,\eta),
\]
which is precisely \((\eta^{k^2})_{k=0}^5\).
\end{proof}
}}
\end{example}

\medskip

\subsection{Subobjects and quotient objects in $YB(\Mat)$}\label{ss:subobjects}

By Definition \ref{de:YBC}, morphism sets in $YB(\Mat)$ are of the following form:
\[\HomYB((M,S),(N,R))=\{Q\in\Mat(M,N): (Q\otimes Q) S=R(Q\otimes Q)\}.\]    
Since $\HomYB((M,S),(N,R))$ is \emph{not} generally closed under addition, $YB(\Mat)$ is not an abelian category.  Moreover, $YB(\Mat)$ does not have a terminal object or initial object so we cannot use the general categorical notion of simple.  However, the morphisms are matrices so we can still meaningfully discuss these notions in this context.   

\mdef   \label{de:Aut}
Let $F: \Bcat\rightarrow\Mat$ be a braid representation, determined by $F(1)=N$ and $F(\sigma)=R$. 
Note that  $\EndYB(N,R)$ consists of those $Q\in\Mat(N,N)$ such that $Q\otimes Q$ commutes with $R$. 
Meanwhile $\AutYB(N,R)$ denotes the invertible endomorphisms.

If $(N,R)$ are as above, then by default we will write $\Aut(N,R)$ for $\AutYB(N,R)$, and similarly for $\End(N,R)$. 

\mdef 
In any category $\catC$ one defines a \emph{monomorphism} to be a morphism $f\in\catC(X,Y)$ that is left-cancellable, i.e., for any $g,h\in \catC(Y,Z)$, if $f\circ g=f\circ h$ then $g=h$.  In the category $\Mat$ a morphism $Q\in\Mat(M,N)$ with $M\leq N$ is left-cancellable precisely when it is full rank, i.e. rank $M$. 
In this case $Q$ has  
{left inverses}, one of which is the Moore-Penrose pseudo-inverse defined by \beq\label{MPinv} Q^+:=(Q^\dag Q)^{-1}Q^\dag.\eq    
Thus the same holds for monomorphisms in $YB(\Mat)$.   Similarly, \emph{epimorphisms} are right-cancellable, so that $P\in\Mat(N,M)$ with $N\geq M$ is an epimorphism if $P$ is full rank, and again, the same is true for epimorphisms in $YB(\Mat)$.

    Suppose that $Q\in\Hom((M,S),(N,R))$ and $Q'\in\Hom((M,T),(N,R))$ are monomorphisms with target $(N,R)$.  Then $Q$ and $Q'$ are \emph{equivalent} if there exists an isomorphism $A\in\Hom((M,T),(M,S))$ such that $Q'=QA$.  For epimorphisms $P\in\Hom((N,R),(M,S))$ and $P'\in\Hom((N,R),(M,T))$ we require such an isomorphism $A\in\Hom((M,S),(M,T))$ such that $AP=P'$.

\begin{defin}
    A \emph{subobject} of $(N,R)\in YB(\Mat)$ is an equivalence class $[Q,M,S]$ of triples $(Q,M,S)$ where $Q$ is a monomorphism from $(M,S)$ to $(N,R)$.  A \emph{quotient object} of $(N,R)$ is an equivalence class $[M,S,P]$ where $P$ is an epimorphism from $(N,R)$ to $(M,S)$.
\end{defin}
If $[Q,M,S]$ is a subobject of $(N,R)$ we will say that $(M,S)$ `appears as a subobject' of $(N,R)$.

\begin{example}
    A 1-dimensional subobject of $(N,R)$ corresponds to an equivalence class of triples $(v,1,\lambda)$ where $v$ is a (column) vector such that $v\otimes v$ is a (right) eigenvector of $R$ of eigenvalue $\lambda$.  In this case $(w,1,\lambda)\in[v,1,\lambda]$ if $w=\alpha v$ for some $\alpha\neq 0$.  Similarly, $1$-dimensional quotient objects correspond to left eigenvectors of the form $x\otimes x$.
\end{example}

The classical linear algebraic notion of invariant subspaces can be incorporated into this language as follows:
\begin{lemma}\label{lem: invariant is subobject}
   Let $(N,R)\in YB(\Mat)$, and let
    $W\otimes W$ be a right (resp. left) $R$-invariant subspace of $\C^N\otimes \C^N$ with $\dim(W)=M$.  Then there is a {subobject} $[Q,M,S]$ (resp. quotient object $[M,S,P]$)  such that $Q(\C^M)=W$.
\end{lemma}
\begin{proof}  
    Let $[w_1,\ldots,w_M]$ be an ordered basis of $W\subset \C^N$, and let $Q$ be the {full rank} $N\times M$ matrix with columns $w_i$.  To be explicit, $Q$ is a monomorphism from $M$ to $N$.  The Moore-Penrose pseudo-inverse $Q^+$ satisfies $Q^+Q=\Id_M$, and $QQ^+|_{W}=\Id_W$.  In particular the restriction of $Q^+$ to $W$ is injective.  It follows that $S:=(Q^+\otimes Q^+)R(Q\otimes Q)$ is the restriction of $R$ to $W\otimes W$. 

    The left/quotient version is completely analogous.
\end{proof}

While some $(M,S)$ may appear as both a subobject and quotient object of $(N,R)$, this doesn't always happen. 
\begin{example}\label{ex:rightnotleft}
The \YangBaxterobject\  $(2,R)$ with $R=\left[ \begin {array}{cccc} 5&0&0&0\\ \noalign{\medskip}0&3&2&0
\\ \noalign{\medskip}0&5&0&0\\ \noalign{\medskip}0&5&2&-2\end {array}
 \right]$ has $(1,[-2])$ appearing as a subobject, but not as a quotient object.  This can be seen by observing that $\ket{22}$ is a right eigenvector, but there is no left eigenvector of the form $v\otimes v$ with eigenvalue $-2$. This $R$ is isomorphic to $R'=\left[ \begin {array}{cccc} 5&0&0&0\\ \noalign{\medskip}0&3
&2&0\\ \noalign{\medskip}0&5&0&0\\ \noalign{\medskip}7&0&0&-2
\end {array} \right]$ by the isomorphism 
$A=\left[\begin {array}{cc} 1&0\\ \noalign{\medskip}-1&2\end {array}
 \right],$  where $R'$ is the form appearing in the classification of \cite{Hietarinta1992}, see section \ref{ss:jarmo}.

\end{example}

\begin{defin}
    We will say $(N,R)\in YB(\Mat)$ is \emph{simple} if it has no proper subobjects nor quotient objects, i.e., the only such object is $(N,R)$ itself up to isomorphism. 
  
\end{defin}

 The definition of simple objects should be compared with that of {indecomposable objects}, described in \cite[Definition 2.5]{etingof_et_al} in the set-theoretical setting.  The general definition is that $(N,R)\in YB(\Mat)$ is \emph{right decomposable} if there exists a non-trivial decomposition $\C^N=V\oplus W$ such that $W$ and $V$ are each right $R$-invariant subspaces, and is otherwise \emph{right indecomposable}.  Left indecomposable is analogously defined.  By Lemma \ref{lem: invariant is subobject} $(N,R)$ is right decomposable if there are subobjects $[Q_1,S_1,M_1],[Q_2,S_2,M_2]$ of $(N,R)$ such that $\C^N=Q_1(\C^{M_1})\oplus Q_2(\C^{M_2})$.  Note, however, that $(N,R)$ is not determined by $[Q_1,S_1,M_1]$ and $[Q_2,S_2,M_2]$--there may be many choices of $(N,R')\in YB(\Mat)$ with complementary subobjects $[Q_1,S_1,M_1]$ and $[Q_2,S_2,M_2]$.  The behaviour of $R'$ on $(V\otimes W)\oplus (W\otimes V)$ is additional information. Moreover, it is possible to be right decomposable but left indecomposable, as Example \ref{ex:rightnotleft} illustrates: 
 the subobjects $[(1 \; 0\;0\;0)^T, 1,[5]]$ and $[(0\; 0\;0\;1)^T, 1,[-2]]$ give a right decomposition, whilst there is only one left eigenvector of the form $v\otimes v$, and a decomposition requires two.

\begin{defin} Given $(M_1,S_1),(M_2,S_2)\in YB(\Mat)$, any  $(M_1+M_2,R)\in YB(\Mat)$ that is right decomposable with corresponding subobjects $[Q_1,M_1,S_1],[Q_2,M_2,S_2]$ is called a \emph{lift} of $(M_1,S_1),(M_2,S_2)$.  Setting $V_i:=Q_i(\C^{M_i})$, a lift $(M_1+M_2,R)\in YB(\Mat)$ is \emph{two-sided} if in addition we have $(V_1\otimes V_2)\oplus (V_2\otimes V_1)$ is $R$-invariant.
\end{defin}
\begin{remark}
    {Note that by the block structure of a two-sided lift $R$ on the $V_i\otimes V_i$ and $(V_1\otimes V_2)\oplus (V_2\otimes V_1)$ subspaces, it is clear that $R^T$ has the same block structure. 
 In particular the spaces $V_i\otimes V_i$ are $R^T$ invariant, hence yield quotient objects as well. }
\end{remark}

Classifying lifts (two-sided or otherwise) of $(M_1,S_1),(M_2,S_2)\in YB(\Mat)$ can be computationally difficult. One always exists: there is a natural notion of direct sum of objects, found in \cite{LechnerPennigWood}, see also \cite{Hietarinta1993}.  We expand their definition slightly, for any $R\in\Mat^N(2,2)$ and $S\in\Mat^M(2,2)$ and $\mu\neq 0$:
\[(R\boxplus_\mu S)\ket{ij}=\begin{cases}
    R\ket{ij}, & 1\leq i,j\leq N\\
    S(\ket{i-N}\otimes\ket{j-N}), & N+1\leq i,j\leq N+M\\
    \mu\ket{ji}, & \textrm{else}.
\end{cases}
\]

Extending the results of \cite{LechnerPennigWood} we have:
\begin{lemma}
If $(N,R),(M,S)\in YB(\Mat)$  then $(N+M,R\boxplus_\mu S)\in YB(\Mat)$.\footnote{It is not hard to see that this construction applies more generally to any $YB(\catC)$ where $\catC$ is itself braided, much like the lashing product.  We will not need this generality.} 
\end{lemma}
\begin{proof}
One only needs to verify that $R\boxplus_\mu S$ satisfies \eqref{eq:YBE}.  This is done in the case of $\mu=1$ in \cite{LechnerPennigWood}, and the observation that \eqref{eq:YBE} is homogeneous allows for general $\mu\neq 0$. 
\end{proof}

\ignore{{  \mdef Consider the following inductive process: Let $CD_1$ consist of all $(1,[a])\in YB(\Mat)$.  Let $CD_2$ be the set of all two-sided lifts of $(1,[a_1])$ and $(1,[a_2])$. It is not hard to see that any $(2,R)\in CD_2$ is in $\Match^2(2,2)$.  Now  define $CD_N$ to be the set of all two-sided lifts of $(N-1,S)\in CD_{N-1}$ and some $(1,[b])\in YB(\Mat)$. Computationally this is a matrix completion problem: to find the matrices in $CD_N$ we must solve for $N^2$ variables.  \begin{question}
       What is the totality of such $CD_N$? 
   \end{question} 
    }}

 \begin{theorem}\label{thm:morphismissub/quo}
 Objects $(M,S)\in YB(\Mat)$ that are a quotient object of $(N_1,R_1)$ and a subobject of $(N_2,R_2)$ correspond to rank $M$ morphisms $A\in \Hom_{YB}((N_1,R_1),(N_2,R_2))$.
\end{theorem}
\begin{proof} First suppose that $A\in\Hom_{YB}((N_1,R_1),(N_2,R_2))$ has rank $M$, i.e. $R_2(A\otimes A)=(A\otimes A)R_1$.  Let $Q$ be an $N_2\times M$ matrix whose columns are a basis for the column space, $W$, of $A$.  Note that $W\otimes W$ is $R_2$-invariant.  By solving the equation $QX=A$ column by column (or using a left inverse of $Q$) we obtain an $M\times N_1$ matrix $P$ such that $QP=A$.  Note that $P$ has rank $M$.  It follows that $R_2(Q\otimes Q)=(Q\otimes Q)S$ for some matrix $S$.  By comparing ranks we see that $S$ is invertible, and, indeed, $S=Q^+\otimes Q^+ R_2 Q\otimes Q$, where $Q^+$ is any left pseudo-inverse of $Q$. Observe that \[(Q\otimes Q)(P\otimes P)R_1=(QP\otimes QP)R_1=R_2(QP\otimes QP)=(Q\otimes Q)S(P\otimes P) \] and $Q$ has a left inverse, so that $S(P\otimes P)=(P\otimes P)R_1$.  Therefore $S$ is a quotient object of $(N_1,R_1)$. 

For the converse, suppose that $[Q,M,S]$ is a subobject of $(N_2,R_2)$ and $[M,S,P]$ is a quotient object of $(N_1,R_1)$.  In particular $Q$ and $P$ are full rank.  The $N_2\times N_1$ matrix $QP$ has rank $M$, since $P$ is surjective and $Q$ is injective.  Moreover $(QP\otimes QP)$ intertwines $R_1$ and $R_2$, i.e. $QP\in\Hom_{YB}((N_1,R_1),(N_2,R_2))$, since \begin{eqnarray*}
    R_2(QP\otimes QP) &=& R_2(Q\otimes Q)(P\otimes P)=\\  (Q\otimes Q)S(P\otimes P) &=& (Q\otimes Q)(P\otimes P)R_1 = (QP\otimes QP)R_1.
\end{eqnarray*}
\end{proof}

\begin{theorem}[Schur's Lemma for $YB(\Mat)$]\label{thm:schur}
 Let $(N_1, R_1)$ and $(N_2, R_2)$ be simple objects in $YB(\Mat)$. Any non-zero morphism $A \in \Hom_{YB}((N_1, R_1), (N_2, R_2))$ is an isomorphism. 
\end{theorem}
\begin{proof}
 Let $A \in \Hom_{YB}((N_1, R_1), (N_2, R_2))$ be a non-zero morphism. Suppose $A$ has rank $M > 0$. By Theorem \ref{thm:morphismissub/quo}, $A$ factors as $A = QP$, where $P \in \Mat(M, N_1)$ and $Q \in \Mat(N_2, M)$ are full rank, and there is an intermediate object $(M, S) \in YB(\Mat)$ such that $[M, S, P]$ is a quotient object of $(N_1, R_1)$ and $[Q, M, S]$ is a subobject of $(N_2, R_2)$.

 Since $(N_1, R_1)$ is simple, it has no proper quotient objects. Because $A \neq 0$, the rank $M$ cannot be zero, which forces $M = N_1$. Therefore, the epimorphism $P$ is a full-rank $N_1 \times N_1$ square matrix, and hence an invertible matrix. By definition of morphisms in $YB(\Mat)$, this means $P$ is an isomorphism.

 Similarly, since $(N_2, R_2)$ is simple, it has no proper subobjects, forcing $M = N_2$. The monomorphism $Q$ is therefore a full-rank $N_2 \times N_2$ square matrix, and hence an isomorphism.

 It follows that $N_1 = N_2 = M$, and $A = QP$ is the composition of two isomorphisms, meaning $A$ is itself an isomorphism. 
\end{proof}

\begin{remark}
In standard linear algebra, the second half of Schur's Lemma proves that the endomorphisms of a simple object over $\mathbb{C}$ are just scalar multiples of the identity matrix, but that standard proof requires you to subtract $\lambda I$ from your matrix to find a kernel. This category is not additive, and so you cannot subtract morphisms. Thus, our version proves that every non-zero morphism is an isomorphism, but allows for exotic, non-scalar endomorphism groups—like the one observed in the simple Ising solution in \ref{ex:gauss2}.
\end{remark}

Applying Theorem \ref{thm:morphismissub/quo} to $(N_1,R_1)=(N_2,R_2)$ we immediately deduce:
\begin{corollary}\label{thm:endoissub/quo}
 Objects $(M,S)\in YB(\Mat)$ that are both subobjects and quotient objects of $(N,R)$ correspond to rank $M$ endomorphisms $A\in \End(N,R)$.
\end{corollary}
\ignore{{\begin{proof} First suppose that $A\in\End(N,R)$ has rank $M$, i.e. $R(A\otimes A)=(A\otimes A)R$.  Let $Q$ be an $N\times M$ matrix whose columns are a basis for the column space, $W$, of $A$.  Note that $W\otimes W$ is $R$-invariant.  By solving the equation $QX=A$ column by column (or using a left inverse of $Q$) we obtain an $M\times N$ matrix $P$ such that $QP=A$.  Note that $P$ has rank $M$.  It follows that $R(Q\otimes Q)=(Q\otimes Q)S$ for some matrix $S$.  By comparing ranks we see that $S$ is invertible, and, indeed, $S=Q^+\otimes Q^+ R Q\otimes Q$, where $Q^+$ is any left pseudo-inverse of $Q$. Observe that \[(Q\otimes Q)(P\otimes P)R=R(QP\otimes QP)=(Q\otimes Q)S(P\otimes P) \] and $Q$ has a left inverse, so that $S(P\otimes P)=(P\otimes P)R$.  Therefore $S$ is a quotient object of $R$ as well. 

For the converse, suppose that $[Q,M,S]$ is a subobject of $(N,R)$ and $[M,S,P]$ is a quotient object.  In particular $Q$ and $P$ are full rank.  The $N\times N$ matrix $QP$ has rank $M$, since $P$ is surjective and $Q$ is injective.  Moreover $(QP\otimes QP)$ commutes with $R$, i.e. $QP\in\End(N,R)$. 
\end{proof}}}

Corollary \ref{thm:endoissub/quo} suggests an algorithm for finding subobjects (and quotient objects) of a given $(N,R)$.  It is computationally straightforward to compute $\End(N,R)$, so one simply extracts a basis for the column space of some $A\in\End(N,R)$ of rank $M$, places them column-wise in a matrix $Q$, and computes $S:=(Q^+\otimes Q^+)R(Q\otimes Q)$.  Then $(Q,M,S)$ yields a subobject.  Similarly one constructs quotient objects by taking a basis of the row-space of such an $A$.
The following example illustrates this approach, using an object in $YB(\Mat)$ constructed via $2$-cabling.

\begin{example}\label{ex:projection}
{First note that the matrix}
\[R'=\left[ \begin {array}{cccc} 1&0&0&0\\ \noalign{\medskip}0&1+x&-x&0
\\ \noalign{\medskip}0&1&0&0\\ \noalign{\medskip}0&0&0&1\end {array}
 \right] \]  
{gives a braid representation for any $x$. 
This is case $f$ of the $\Match^2$ classification \cite{MR1X} with parameter $\alpha=1$, thus a simple 
deformation of the basic flip solution; 
dual to the usual action of $U_qsl_2$ on tensor space for a suitable $q$.
 In other words $(2,R')\in YB(\Mat)$, with $R'\in\Match^2(2,2)$. 
 The eigenvalues are $1,1,1,x$
- multiplicities corresponding to the dimensions of the spin-1 and spin-0 representations  on the $U_q sl_2$ side.} 
Define $R=R'_2R'_1R'_3R'_2$, i.e., the $2$-cabling ${R'}^{(2)}$ of $R'$
as in (\ref{pr:cable}).   
We have $(4,R)\in YB(\Mat)$, with eigenvalues and their multiplicities: 
$[[1, 5], [-x^2, 3], [-x, 4], [-x^3, 1], [x, 3]]$.  
Setting $y=x+1$ we have, explicitly (cf. \eqref{matrix units}):

{
\begin{equation}\label{eq:Rcable2}
\begin{aligned}
R={}&
E_{11,11}
+yE_{21,21}
-xyE_{21,31}
+x^2E_{21,12} \\
&+yE_{31,21}
-xyE_{31,31}
+x^2E_{31,13} \\
&+y^2(1-x)E_{41,41}
-xyE_{41,22}
+x^2yE_{41,32}
+x^2yE_{41,23} \\
&-x^3yE_{41,33}
+x^4E_{41,14} \\
&+E_{12,21}
+yE_{22,41}
-xE_{22,22} \\
&+yE_{32,41}
-xE_{32,23} \\
&+yE_{42,42}
-xyE_{42,43}
+x^2E_{42,24} \\
&+E_{13,31}
+yE_{23,41}
-xE_{23,32} \\
&+yE_{33,41}
-xE_{33,33} \\
&+yE_{43,42}
-xyE_{43,43}
+x^2E_{43,34} \\
&+E_{14,41}
+E_{24,42}
+E_{34,43}
+E_{44,44}.
\end{aligned}
\end{equation}}

\ignore{{
\begin{aligned}
R={}&
E_{1,1}
+yE_{2,2}-xyE_{2,3}+x^2E_{2,5}
+yE_{3,2}-xyE_{3,3}+x^2E_{3,9} \\
&+y^2(1-x)E_{4,4}
-xyE_{4,6}
+x^2yE_{4,7}
+x^2yE_{4,10}
-x^3yE_{4,11}
+x^4E_{4,13} \\
&+E_{5,2}
+yE_{6,4}-xE_{6,6}
+yE_{7,4}-xE_{7,10} \\
&+yE_{8,8}-xyE_{8,12}+x^2E_{8,14}
+E_{9,3} \\
&+yE_{10,4}-xE_{10,7}
+yE_{11,4}-xE_{11,11} \\
&+yE_{12,8}-xyE_{12,12}+x^2E_{12,15} \\
&+E_{13,4}+E_{14,8}+E_{15,12}+E_{16,16}.
\end{aligned}

\beq   
R={\footnotesize \left[ \begin {array}{cccccccccccccccc} 1&0&0&0&0&0&0&0&0&0&0&0&0&0&0
&0\\ \noalign{\medskip}0&y&-x y &0&{x}^{2}&0&0&0&0&0
&0&0&0&0&0&0\\ \noalign{\medskip}0&y&-x y &0&0&0&0&0
&{x}^{2}&0&0&0&0&0&0&0\\ \noalign{\medskip}0&0&0& y ^
{2}(1-x)&0&-x y &{x}^{2} y &0&0&{x}^{2} y &-{x}^{3}  y
 &0&{x}^{4}&0&0&0\\ \noalign{\medskip}0&1&0&0&0&0&0&0&0&0&0&0&0
&0&0&0\\ \noalign{\medskip}0&0&0&y&0&-x&0&0&0&0&0&0&0&0&0&0
\\ \noalign{\medskip}0&0&0&y&0&0&0&0&0&-x&0&0&0&0&0&0
\\ \noalign{\medskip}0&0&0&0&0&0&0&y&0&0&0&-x y &0&
{x}^{2}&0&0\\ \noalign{\medskip}0&0&1&0&0&0&0&0&0&0&0&0&0&0&0&0
\\ \noalign{\medskip}0&0&0&y&0&0&-x&0&0&0&0&0&0&0&0&0
\\ \noalign{\medskip}0&0&0&y&0&0&0&0&0&0&-x&0&0&0&0&0
\\ \noalign{\medskip}0&0&0&0&0&0&0&y&0&0&0&-x y &0&0
&{x}^{2}&0\\ \noalign{\medskip}0&0&0&1&0&0&0&0&0&0&0&0&0&0&0&0
\\ \noalign{\medskip}0&0&0&0&0&0&0&1&0&0&0&0&0&0&0&0
\\ \noalign{\medskip}0&0&0&0&0&0&0&0&0&0&0&1&0&0&0&0
\\ \noalign{\medskip}0&0&0&0&0&0&0&0&0&0&0&0&0&0&0&1\end {array}
 \right] }
\eq}}

Aside: As noted, $R'$ is CC with $F(1)=1$. 
Comparing $R$ with (\ref{eq:F12shape}) we see that the shape is the same. 
\\

{\em Claim}. Starting with $(2,R')\in YB(\Mat)$ with $R' \in \Match^2(2,2)$ and {2-}cabling we obtain $R \in$ {$\Match^2(4,4)$, so in particular we have $(4,R)\in YB(\Mat)$ with $R\not\in \Match^4(2,2)$.}

\medskip 

 The following rank $3$ matrix is in $\End(4,R)$: 
 \[A:= \left[ \begin {array}{cccc} 1&0&0&0\\ \noalign{\medskip}0&\frac{1}{1-x}&{\frac {x}{x-1}}&0\\ \noalign{\medskip}0&\frac{1}{1-x}&{\frac {x}{x-1}}&0\\ \noalign{\medskip}0&0&0&1
\end {array} \right] \]
Since the columns of $A\otimes A$ span a $9$-dimensional $R$-invariant subspace, we extract a basis and place them column-wise, and compute the pseudo-inverse: \[Q:=\left[ \begin {array}{ccc} 1&0&0\\ \noalign{\medskip}0&1&0
\\ \noalign{\medskip}0&1&0\\ \noalign{\medskip}0&0&1\end {array}
 \right], \quad Q^+= \left[ \begin {array}{cccc} 1&0&0&0\\ \noalign{\medskip}0&{\frac{1}{2
}}&{\frac{1}{2}}&0\\ \noalign{\medskip}0&0&0&1\end {array} \right].
 \]  Now we compute  
 $(Q^+\otimes Q^+)R(Q\otimes Q)$ and obtain (again, with $y=x+1$):
{
\[
\begin{aligned}
S={}&
E_{11,11}
+y(1-x)E_{21,21}
+x^2E_{21,12} \\
&-y^2(x-1)E_{31,31}
-xy(x-1)^2E_{31,22}
+x^4E_{31,13} \\
&+E_{12,21}
+yE_{22,31}
-xE_{22,22} \\
&+y(1-x)E_{32,32}
+x^2E_{32,23} \\
&+E_{13,31}
+E_{23,32}
+E_{33,33}.
\end{aligned}
\]}

\ignore{{
\[
\begin{aligned}
S={}&
E_{1,1}
+y(1-x)E_{2,2}
+x^2E_{2,4} \\
&-y^2(x-1)E_{3,3}
-xy(x-1)^2E_{3,5}
+x^4E_{3,7} \\
&+E_{4,2}
+yE_{5,3}
-xE_{5,5} \\
&+y(1-x)E_{6,6}
+x^2E_{6,8} \\
&+E_{7,3}
+E_{8,6}
+E_{9,9}.
\end{aligned}
\]

\[ S:= \left[ \begin {array}{ccccccccc} 1&0&0&0&0&0&0&0&0
\\ \noalign{\medskip}0&y(1-x)&0&{x}^{2}&0&0&0&0&0
\\ \noalign{\medskip}0&0&- y ^{2}
 \left( x-1 \right)
&0&-x y  \left( x-1 \right) ^{2}&0&{x}^{4}&0&0\\ \noalign{\medskip}0&1&0&0&0&0&0&0&0
\\ \noalign{\medskip}0&0&y&0&-x&0&0&0&0\\ \noalign{\medskip}0&0&0&0&0
&y(1-x)&0&{x}^{2}&0\\ \noalign{\medskip}0&0&1&0&0&0&0&0&0
\\ \noalign{\medskip}0&0&0&0&0&1&0&0&0\\ \noalign{\medskip}0&0&0&0&0&0
&0&0&1\end {array} \right] 
,
\]
}}
which is in $\Matcha^3(2,2)$.  
Observe that, $Q\in\Hom((3,S),(4,R))$ is a monomorphism so that $[Q,3,S]$ is a subobject of $(4,R)$.  

Setting $P= \left[ \begin {array}{cccc} 1&0&0&0\\ \noalign{\medskip}0& \frac{1}{1-x}&{\frac {-x}{1-x}}&0\\ \noalign{\medskip}0&0&0&1
\end {array} \right]$ we find, similarly, that $[3,S,P]$ is a quotient object of $(4,R)$. Note also that $PQ=A\in\End((4,R))$. 
{
By the classification in \cite{HMR1} the matrix $S$ above is in the same variety as the well-known $R$-matrix for $U_q\mathfrak{so}_3$ (see e.g., \cite{Wenzl90CMP} for the explicit form we use). This is unsurprising as the adjoint representation of $\mathfrak{sl}_2$ is the standard representation of $\mathfrak{so}_3$.}
\medskip

The following rank $2$ matrix $B:= \left[ \begin {array}{cccc} 1&0&0&0\\ \noalign{\medskip}0&{\frac {x}{
x-1}}&{\frac {-x}{x-1}}&0\\ \noalign{\medskip}0& \frac {1}{x-1}&\frac {-1}{x-1}&0\\ \noalign{\medskip}0&0&0&0
\end {array} \right]$ in $\End((4,R))$ yields, by a similar process as above, a subobject  $[Q',2,T]$ and a quotient object $[2,T,P']$ of $(4,R)$ where $T= \left[ \begin {array}{cccc} 1&0&0&0\\ \noalign{\medskip}0&0&{x}^{2}&0
\\ \noalign{\medskip}0&1&0&0\\ \noalign{\medskip}0&0&0&-x\end {array}
 \right],$ and $B=P'Q'$.

  There are also three  $1$-dimensional subobjects and quotient object of $(4,R)$: two corresponding to (left/right) eigenvalue $1$ and one with (left/right) eigenvalue $-x$. 
\end{example}

\begin{example}\label{ex:CCnotSimple}
    Objects $(N,R)\in YB(\Mat^N)$ with $N>1$ and $R\in\Match^N(2,2)$ are never simple: they have subobjects and quotient objects of the form $(M,S)$ for every $M\leq N$. Indeed, choose any $M$ distinct standard basis vectors $\mathcal{B}=[\ket{j_1},\ldots,\ket{j_M}]$ and let $(M,S)\in YB(\Mat)$ be the charge conserving object corresponding to the restriction of $R$ to the {span of $\mathcal{B}\otimes \mathcal{B}$}.   Define $Q\in\Mat(M,N)$ to be the matrix of the vectors in $\mathcal{B}$ taken column-wise.  Then $ S=(Q^+\otimes Q^+)R (Q\otimes Q)$.  Note that $S\in \Match^M(2,2)$. The same construction works for quotient objects. It is also clear that any diagonal matrix $A$ with eigenvalues $0,1$ is in $\End(N,R)$, so that {by Corollary \ref{thm:endoissub/quo} } the construction of subobjects and quotient objects from $\End(N,R)$ produces all of these examples.  

\ignore{{    In fact, the above shows that any $(1,R)\in YB(\Mat^N)$ with $R\in\Match^N(2,2)$ is left and right decomposable for any $N>1$.  Note that the corresponding $(1,R)\in YB(\Match^N)$ lie in $CD_N$--they are maximally decomposable in this sense.  They are properly contained in $CD_N$--the first examples where equality fails is for $N=3$.
}}
\end{example}

\begin{theorem}[Upper-Triangular Decomposition]\label{thm:upper_triangular}
 Let $(N, R) \in YB(\Mat)$ be a non-simple object, and let $[Q, M, S]$ be a proper subobject of $(N, R)$. Then $(N, R)$ is locally equivalent to an object $(N, R')$ such that, up to a permutation of the standard basis of $\mathbb{C}^N \otimes \mathbb{C}^N$, the matrix $R'$ is block upper-triangular of the form:
 \[
 R' \cong \begin{bmatrix} S & G \\ 0 & T \end{bmatrix}
 \]
 where $S \in \Mat(M^2, M^2)$ is the Yang-Baxter subobject acting on the invariant tensor subspace.
\end{theorem}
\begin{proof}
 By definition of a proper subobject, $0 < M < N$ and $Q$ is an $N \times M$ matrix of full column rank satisfying the intertwining relation:
 \[
 R(Q \otimes Q) = (Q \otimes Q)S.
 \]
 Let $W \subset \mathbb{C}^N$ be the $M$-dimensional column space of $Q$. Because the $M$ columns of $Q$ are linearly independent, the Basis Extension Theorem guarantees we may choose an invertible matrix $U \in \Aut(N)$ whose first $M$ columns are exactly the columns of $Q$. In the basis defined by $U$, the embedding matrix becomes $Q' = U^{-1}Q$, which takes the strict block form $Q' = \begin{bmatrix} I_M \\ 0 \end{bmatrix}$, where $I_M$ is the $M \times M$ identity matrix, and $0$ is the $(N-M) \times M$ zero matrix.

 We define the locally equivalent Yang-Baxter object $R' = (U^{-1} \otimes U^{-1})R(U \otimes U)$. Left-multiplying the intertwining relation by $(U^{-1} \otimes U^{-1})$ and inserting $(U \otimes U)(U^{-1} \otimes U^{-1})$ yields:
 \[
 R'(Q' \otimes Q') = (Q' \otimes Q')S.
 \]

 By construction, $Q'$ maps $\mathbb{C}^M$ into the subspace spanned by the first $M$ standard basis vectors of $\mathbb{C}^N$. Consequently, $Q' \otimes Q'$ maps $\mathbb{C}^M \otimes \mathbb{C}^M$ strictly into the subspace spanned by the basis vectors $|i\rangle \otimes |j\rangle$ with $1 \le i, j \le M$. 

 Let us apply a permutation matrix $P$ to the standard lexicographical basis of $\mathbb{C}^N \otimes \mathbb{C}^N$ to group these $M^2$ basis vectors first. In this ordered basis, the matrix $P(Q' \otimes Q')$ takes the strict block form:
 \[
 P(Q' \otimes Q') = \begin{bmatrix} I_{M^2} \\ 0 \end{bmatrix}.
 \]

 Conjugating $R'$ by $P$, we write the matrix $PR'P^{-1}$ in block form corresponding to this decomposition:
 \[
 PR'P^{-1} = \begin{bmatrix} \widetilde{R}_{11} & \widetilde{R}_{12} \\ \widetilde{R}_{21} & \widetilde{R}_{22} \end{bmatrix}.
 \]
 Substituting these block forms into the subobject condition $PR'P^{-1} P(Q' \otimes Q') = P(Q' \otimes Q')S$ yields:
 \[
 \begin{bmatrix} \widetilde{R}_{11} & \widetilde{R}_{12} \\ \widetilde{R}_{21} & \widetilde{R}_{22} \end{bmatrix} \begin{bmatrix} I_{M^2} \\ 0 \end{bmatrix} = \begin{bmatrix} I_{M^2} \\ 0 \end{bmatrix} S \implies \begin{bmatrix} \widetilde{R}_{11} \\ \widetilde{R}_{21} \end{bmatrix} = \begin{bmatrix} S \\ 0 \end{bmatrix}.
 \]
 Equating the blocks forces $\widetilde{R}_{21} = 0$ (proving that $PR'P^{-1}$ is block upper-triangular) and $\widetilde{R}_{11} = S$ (confirming that the restriction to this subspace is precisely the subobject $S$). Setting $G = \widetilde{R}_{12}$ and $T = \widetilde{R}_{22}$ completes the proof. 
\end{proof}

\begin{theorem}[Lower-Triangular Quotient Decomposition]\label{thm:lower_triangular}
 Let $(N, R) \in YB(\Mat)$ be a non-simple object, and let $[M, S, P]$ be a proper quotient object of $(N, R)$. Then $(N, R)$ is locally equivalent to an object $(N, R')$ such that, up to a permutation of the standard basis of $\mathbb{C}^N \otimes \mathbb{C}^N$, the matrix $R'$ is block lower-triangular of the form:
 \[
 R' \cong \begin{bmatrix} S & 0 \\ G & T \end{bmatrix}
 \]
 where $S \in \Mat(M^2, M^2)$ is the Yang-Baxter quotient object.
\end{theorem}
\begin{proof}
 By definition of a proper quotient object, $0 < M < N$ and $P$ is an $M \times N$ matrix of full row rank satisfying the intertwining relation:
 \[
 S(P \otimes P) = (P \otimes P)R.
 \]
 Because the $M$ rows of $P$ are linearly independent, we may choose an invertible matrix $U \in \Aut(N)$ whose first $M$ rows are exactly the rows of $P$. In the basis defined by $U$, the surjective matrix becomes $P' = PU^{-1}$, which takes the strict block form $P' = \begin{bmatrix} I_M & 0 \end{bmatrix}$, where $I_M$ is the $M \times M$ identity matrix, and $0$ is the $M \times (N-M)$ zero matrix.

 We define the locally equivalent Yang-Baxter object $R' = (U \otimes U)R(U^{-1} \otimes U^{-1})$. Right-multiplying the intertwining relation by $(U^{-1} \otimes U^{-1})$ and inserting $(U^{-1} \otimes U^{-1})(U \otimes U)$ yields:
 \[
 S(PU^{-1} \otimes PU^{-1}) = (PU^{-1} \otimes PU^{-1})(U \otimes U)R(U^{-1} \otimes U^{-1}),
 \]
 which simplifies to:
 \[
 S(P' \otimes P') = (P' \otimes P')R'.
 \]

 Let us apply a permutation matrix $\Pi$ to the standard lexicographical basis of $\mathbb{C}^N \otimes \mathbb{C}^N$ to group the $M^2$ basis vectors mapped by $P' \otimes P'$ first. In this ordered basis, the matrix $(P' \otimes P')\Pi^{-1}$ takes the strict block form:
 \[
 (P' \otimes P')\Pi^{-1} = \begin{bmatrix} I_{M^2} & 0 \end{bmatrix}.
 \]

 Conjugating $R'$ by $\Pi$, we write the matrix $\Pi R'\Pi^{-1}$ in corresponding block form:
 \[
 \Pi R'\Pi^{-1} = \begin{bmatrix} \widetilde{R}_{11} & \widetilde{R}_{12} \\ \widetilde{R}_{21} & \widetilde{R}_{22} \end{bmatrix}.
 \]
 Substituting these block forms into the quotient condition \[S(P' \otimes P')\Pi^{-1} = (P' \otimes P')\Pi^{-1} \Pi R'\Pi^{-1}\] yields:
 \[
 S \begin{bmatrix} I_{M^2} & 0 \end{bmatrix} = \begin{bmatrix} I_{M^2} & 0 \end{bmatrix} \begin{bmatrix} \widetilde{R}_{11} & \widetilde{R}_{12} \\ \widetilde{R}_{21} & \widetilde{R}_{22} \end{bmatrix} \implies \begin{bmatrix} S & 0 \end{bmatrix} = \begin{bmatrix} \widetilde{R}_{11} & \widetilde{R}_{12} \end{bmatrix}.
 \]
 Equating the blocks forces $\widetilde{R}_{12} = 0$ (proving that $\Pi R'\Pi^{-1}$ is block lower-triangular) and $\widetilde{R}_{11} = S$. Setting $G = \widetilde{R}_{21}$ and $T = \widetilde{R}_{22}$ completes the proof.
\end{proof}

\ignore{{ 

The physical purpose of braiding non-abelian anyons is to generate quantum entanglement, which is at odds with the notion of a sub-object which, in the $1$-dimensional case, implies the preservation of a pure tensor subspace of the form $v\otimes v$. Thus we have the following conjecture.

\begin{conjecture}
  All solutions coming from localising a BFC are simple.   
\end{conjecture}

\begin{remark}
    We certainly do not expect all simple solutions to be equivalent to solutions coming from BFCs. Those coming from BFCs are forced to be diagonalisable, and to have eigenvalues which are discrete roots of unity. Neither of these properties are a requirement for general solutions. However, we have no examples of simple solutions which do not come from BFCs...Is this true, what happens in the rank 2 case? 
\end{remark}
}}

\begin{lemma}\label{lem:cc_converse}
Let $(N, R) \in YB(\Mat)$. Suppose that for every subset of standard basis indices $X \subset \{1, \dots, N\}$ of size $|X| \le 2$, the subspace $W_X = \mathrm{span}(|x\rangle : x \in X)$ corresponds to both a subobject and a quotient object of $(N,R)$. Then $R \in \Match^N(2,2)$, i.e., $R$ is strictly charge conserving. 
\end{lemma}
\begin{proof}
 Let $M = |X|$. By the hypothesis and Lemma~\ref{lem: invariant is subobject}, there is an $N \times M$ full-rank injection matrix $Q_X$ whose columns span $W_X$, and an $M \times N$ full-rank projection matrix $P_X$ whose rows span the dual space $W_X^*$, satisfying the categorical definitions of subobjects and quotient objects (Definition 5.9).
 
 Because $W_X$ yields a subobject, there exists an $S_X$ such that $R(Q_X \otimes Q_X) = (Q_X \otimes Q_X)S_X$. This dictates that the column space of $Q_X \otimes Q_X$, which is exactly $W_X \otimes W_X$, is an $R$-invariant subspace. 
 
 Because $W_X$ yields a quotient object, there exists an $S_X'$ such that $S_X'(P_X \otimes P_X) = (P_X \otimes P_X)R$. The kernel of $P_X \otimes P_X$ is exactly the span of all standard tensor basis vectors not in $W_X \otimes W_X$. This relation dictates that the kernel is also $R$-invariant, meaning no state from outside $W_X \otimes W_X$ can map into $W_X \otimes W_X$.

 We evaluate these conditions on subsets of size 1 and size 2 to constrain the matrix elements of $R$. Following the notation in (2.24), we write $R_{v, w} = \langle v | R | w \rangle$.

 \textit{Step 1: Size-1 Subsets.}
 Let $X = \{i\}$ for some fixed $i \in \{1, \dots, N\}$. The associated tensor space is the 1-dimensional subspace spanned by $|ii\rangle$.
 Because $W_X$ yields a subobject, $R$ maps $|ii\rangle$ strictly into this subspace, so $R|ii\rangle = c |ii\rangle$ for some scalar $c$. Thus, $R_{v, ii} = 0$ for all words $v \neq ii$. 
 Because $W_X$ yields a quotient object, no state outside this subspace can map into it. Thus, for any input word $w \neq ii$, we must have $R_{ii, w} = 0$.
 
 \textit{Step 2: Size-2 Subsets.}
 Let $X = \{i, j\}$ with $i \neq j$. The associated tensor space is the 4-dimensional subspace spanned by $\{|ii\rangle, |ij\rangle, |ji\rangle, |jj\rangle\}$.
 Because $W_X$ yields a subobject, $R$ must map the mixed state $|ij\rangle$ strictly into this 4-dimensional space. Therefore, $R_{v, ij} = 0$ for any word $v \notin \{ii, ij, ji, jj\}$. 
 Thus, we may write $R|ij\rangle$ as the linear combination:
 \[
 R|ij\rangle = R_{ii, ij}|ii\rangle + R_{ij, ij}|ij\rangle + R_{ji, ij}|ji\rangle + R_{jj, ij}|jj\rangle.
 \]

 \textit{Step 3: Intersection of Constraints.}
 From Step 1, we established that for any input word $w \neq ii$, the matrix element $R_{ii, w} = 0$. Since the word $ij \neq ii$, it strictly follows that $R_{ii, ij} = 0$.
 By applying the identical Step 1 logic to the subset $Y = \{j\}$, we know $R_{jj, w} = 0$ for all $w \neq jj$. Since the word $ij \neq jj$, it follows that $R_{jj, ij} = 0$.

 Substituting these zeros into our linear combination forces:
 \[
 R|ij\rangle = R_{ij, ij}|ij\rangle + R_{ji, ij}|ji\rangle.
 \]
 
 Therefore, the only non-zero matrix elements for an input word $w = ij$ (with $i \neq j$) occur when the output word $v \in \{ij, ji\}$. For an input word $w = ii$, the only non-zero element occurs when $v = ii$. 
 In all cases, $R_{v,w} \neq 0$ implies that the word $v$ is a permutation of the word $w$. By Definition 3.2, $R$ is charge conserving.
\end{proof}

We can now derive a characterisation of charge conserving in terms of the existence of quotient and subobjects:

\begin{theorem}[Characterisation of Charge Conservation]\label{thm:cc_characterisation}
 Let $(N, R) \in YB(\Mat)$.
  Then $(N, R)$ is locally equivalent to a charge-conserving object, i.e., there exists a local equivalence to an object $(N, R')$ such that $R' \in \Match^N(2,2)$, if and only if there exists a decomposition of $\mathbb{C}^N$ into a direct sum of 1-dimensional subspaces $L_1 \oplus \dots \oplus L_N$ such that:
 \begin{enumerate}
     \item For every $i \in \{1, \dots, N\}$, the 1-dimensional space $L_i$ corresponds to both a subobject and a quotient object of $(N, R)$.
     \item For every pair $i \neq j$, the 2-dimensional space $L_i \oplus L_j$ corresponds to both a subobject and a quotient object of $(N, R)$.
 \end{enumerate}
\end{theorem}
\begin{proof}
First suppose $R$ is charge conserving, then Example~\ref{ex:CCnotSimple} gives the only if.

We now prove the converse.
 Because $L_1 \oplus \dots \oplus L_N = \mathbb{C}^N$, we may choose a non-zero vector $v_i \in L_i$ for each $i$ to form a basis $\{v_1, \dots, v_N\}$ of $\mathbb{C}^N$. Let $U \in \Aut(N)$ be the change-of-basis matrix whose $i$-th column is exactly the vector $v_i$. Thus, $U$ maps the standard basis vector $|i\rangle$ to the vector $v_i$.

 We define a new, locally equivalent Yang-Baxter object $R' = (U^{-1} \otimes U^{-1})R(U \otimes U)$. 

 By hypothesis, the subspace $L_i = \mathrm{span}(v_i)$ corresponds to a subobject of $R$. Thus, there exists a full-rank injection $Q_i$ (whose column spans $L_i$) such that $R(Q_i \otimes Q_i) = (Q_i \otimes Q_i)S_i$. 
 Under the change of basis $U$, define $Q'_i = U^{-1}Q_i$. The column of $Q'_i$ perfectly spans the standard basis subspace $W'_i = \mathrm{span}(|i\rangle)$.
 Substituting $Q_i = UQ'_i$ into the subobject intertwining relation and left-multiplying by $(U^{-1} \otimes U^{-1})$ yields:
 \[
 (U^{-1} \otimes U^{-1})R(U \otimes U)(Q'_i \otimes Q'_i) = (Q'_i \otimes Q'_i)S_i \implies R'(Q'_i \otimes Q'_i) = (Q'_i \otimes Q'_i)S_i.
 \]
 This proves that the standard basis subspace $W'_i = \mathrm{span}(|i\rangle)$ corresponds to a subobject of $(N, R')$. An identical algebraic argument applies to the quotient objects via the surjection $P'_i = P_i U$.
 
 Similarly, by hypothesis, the subspace $L_i \oplus L_j = \mathrm{span}(v_i, v_j)$ corresponds to a subobject and quotient object of $R$. Following the exact same conjugation by $U$, the subspace $U^{-1}(L_i \oplus L_j) = \mathrm{span}(|i\rangle, |j\rangle)$ must yield a subobject and quotient object for $R'$.

 Thus, for every standard basis subset $X \subset \{1, \dots, N\}$ of size $|X| \le 2$, the subspace $W_X = \mathrm{span}(|x\rangle : x \in X)$ yields both a subobject and a quotient object for the matrix $R'$. By Lemma \ref{lem:cc_converse}, this strictly forces $R'$ to be a charge-conserving matrix, i.e., $R' \in \Match^N(2,2)$. Since $R$ is locally equivalent to $R'$ via $U$, the proof is complete.
\end{proof}

\begin{example}
    No $(N,R)\in YB(\Mat)$ with $R\in\PermMat^N(2,2)$ is simple.  This is immediate from the fact that for $v:=\sum_{i =1}^N\ket{i}$ we have $v\otimes v$ is a (right) eigenvector for such an $R$, with eigenvalue $1$.  Thus $(1,[1])$ corresponds to a subobject.  Similarly $w:=\sum_{i =1}^N\bra{i}$ renders $(1,[1])$ a quotient object. 
\end{example}

The following is a non-trivial example of a simple object in $YB(\Mat)$.
\begin{example}\label{ex:gauss2}
  The `Ising' unitary Yang-Baxter operator  $R=\frac{1}{\sqrt{2}}\left[ \begin {array}{cccc} 1&0&0&1\\ \noalign{\medskip}0&1&1&0
\\ \noalign{\medskip}0&-1&1&0\\ \noalign{\medskip}-1&0&0&1\end {array}
 \right]$ is simple, as it has no left or right eigenvectors of the form $v\otimes v$.  Consistent with the above, we find that $\End(2,R)$ consists of scalar multiples of  $\left[\begin{array}{cc} 1&0\\0&\pm 1\end{array}\right]$ hence $(2,R)$ has no rank $1$ endomorphisms.

\end{example}

Continuing with the philosophy that computing $\End(N,R)$ is essential for finding sub- and quotient objects, we have the following for group-type objects:
\begin{example}
  Let $(N,R)$ group type, i.e. there are $\{g_i\}_{i=1}^N\in GL_N$ such that with respect to the standard basis $[x_1,\ldots,x_N]$ we have $R(x_i\otimes x_j)=g_i(x_j)\otimes x_i$.  In \cite[Lemma 3.1]{KadarMartinRowellWang17} it is shown that the Yang-Baxter equation is equivalent to the matrix equations: $g^{j,k}_i g_ig_j = g^{j,k}_i g_kg_i$ for all $i,j,k$ where $g_i(x_j)=\sum_k g^{j,k}_ix_k$.  Let $h\in\Mat(N,N)$ and suppose that $h\in\End(N,R)$.  Defining  $h(x_j)=\sum_k h^{j,k}x_k$ we find that we must have $h^{i,k} g_kh = h^{i,k} hg_i$ for all $i,k$.  In particular this holds for $h=g_j$, for any $j$, so that $\Aut(N,R)$ contains the group $G=\langle g_1,\ldots,g_N\rangle$ generated by the $g_i$.  
    
    \end{example}

\subsection{Rigidity}
An object in $YB(\Mat)$ is rigid if it has left and right
duals. Since $YB(\Mat)$ is braided
\eqref{ybmat is braided}, we do not need to distinguish
between left and right duals.
A dual to $(N,R)\in YB(\Mat)$ is an object $(M,S)$
together with two matrices
\[
coev_R=Q\in\Hom((1,\id),(NM,R\boxtimes S))
\]
and
\[
ev_R=P\in\Hom((MN,S\boxtimes R),(1,\id))
\]
such that the snake identities:
\[
(I_N\otimes P)(Q\otimes I_N)=I_N
\quad\text{and}\quad
(P\otimes I_M)(I_M\otimes Q)=I_M
\] hold.
Thus $Q\otimes Q$ must be a right eigenvector of
$R\boxtimes S$ with eigenvalue $1$, and $P\otimes P$
must be a left eigenvector of $S\boxtimes R$ with
eigenvalue $1$.
Here $I_N=I_R\in\End(N,R)$ and
$I_M=I_S\in\End(M,S)$.

\begin{lemma}
Every $(N,R)\in YB(\Mat)$ has dual
$(N,R^{-tr})$, where
$R^{-tr}=(R^{-1})^{tr}$.
In particular, if $R^{tr}R=I$, then $(N,R)$
is self-dual.
\end{lemma}

\begin{proof}
We use standard bra-ket notation.
Put $S=R^{-tr}$,
$Q=\sum_i\ket{ii}$, and $P=\sum_i\bra{ii}$.
Taking inverse transposes of the braid relation
shows that $S$ is a Yang--Baxter operator.

We must check that $Q\otimes Q$ intertwines
$\Id$ and $R\boxtimes S$, and that $P\otimes P$
intertwines $S\boxtimes R$ and $\Id$.
Put
\[
K=I_N\otimes P_{N,N}\otimes I_N,
\qquad
\Omega=\sum_{i,j}\ket{ijij}.
\]
Then
\[
K(Q\otimes Q)=\Omega,
\qquad
(P\otimes P)K=\Omega^{tr}.
\]
Since $RS^{tr}=S^{tr}R=I$, we have
\[
(R\otimes S)\Omega=\Omega,
\qquad
\Omega^{tr}(S\otimes R)=\Omega^{tr}.
\]
Using $R\boxtimes S=K(R\otimes S)K$ and $K^2=I$,
we obtain
\[
(R\boxtimes S)(Q\otimes Q)=Q\otimes Q,
\qquad
(P\otimes P)(S\boxtimes R)=P\otimes P.
\]

Next we compute
\[
(I_N\otimes P)(Q\otimes I_N)\ket{k}
=
\sum_{i,j}\delta_{j,i}\delta_{j,k}\ket{i}
=
\ket{k},
\]
and similarly
\[
(P\otimes I_N)(I_N\otimes Q)\ket{k}
=
\sum_{i,j}\delta_{j,k}\delta_{j,i}\ket{i}
=
\ket{k}.
\]
Thus the snake identities hold.
If $R^{tr}R=I$, then $S=R$, giving self-duality.
\end{proof}
Since permutation matrices satisfy $R^{tr}R=I$ we immediately have:
\begin{corollary}
Any $(N,R)\in YB(\Mat)$ with
$R\in\PermMat^N(2,2)$ is self-dual.
\end{corollary}

\begin{remark}
This notion of rigidity differs from Takeuchi's
rigidity of a Yang--Baxter operator, which requires
the canonical map
$R^\flat:V^*\otimes V\to V\otimes V^*$
to be invertible.
Takeuchi's condition is important for the Hopf
algebra version of the FRT construction
\cite{Takeuchi2000}.
\end{remark}

\section{Subcategories of $\MonFun(\Bcat,\Mat)$ and $\MoonFun(\Bcat,\Mat)$}\label{s:subcats of MonFun}

We can begin to consider the two main \sproblems\  described in Section \ref{S:intro}.  We first consider subcategories of $\MonFun(\Bcat,\Mat)$  that come from restricting the targets to a monoidal subcategory of $\Mat$, such as those in Figure \ref{fig:catinclusions}.  Nearly everything said in this section about $\MonFun(\Bcat,\Mat)$ can also be said about $\MoonFun(\Bcat,\Mat)$.  We will not belabour this point, but indicate where there are differences.

 Generally, for any collection of matrices $\catT$ in $\Mat$ one can define $\MonFun(\Bcat,\catT)$ to be the full subcategory of those functors $F\in\MonFun(\Bcat,\Mat)$ with $F(\sigma)\in\catT$.  Note that since $\catT$ is not assumed to have any particular structure (eg. closure under composition) we cannot generally require morphisms to be in $\catT$.   The objects {in $\MonFun(\Bcat,\Mat)$} can be identified with pairs $(N,R)\in YB(\Mat)$, and morphisms the same as those in $YB(\Mat)$.

\subsection{Monoidal Subcategory Targets}
In section \ref{ss:subcats} we described a number of monoidal subcategories of $\Mat$.  This was accomplished by restricting the object monoids, the morphisms and sometimes both.  Some of these are displayed in Figure \ref{fig:catinclusions}.  For such a monoidal subcategory $\X$  we obtain a category 
$\MonFun(\Bcat,\X)$, for which a classification may be attempted.   
As noted in \ref{pa:identifyYBMonFun}
we denote such categories by $YB(\X)$\footnote{Note that $YB(\X)$ is not typically a full subcategory of $YB(\Mat)$, since we require morphisms to be in $\X$.  We will relax this below.}.
  In particular we may study $YB(\PermMat)$ and $YB(\MonMat)$ and for each $N$ we may study the categories $YB(\Mat^N),YB(\Match^N)$ and $YB(\Matcha^N)$.

  The problem of classifying objects in $YB(\PermMat)$ goes back at least to \cite{drinfeld92}, and was substantially studied in \cite{ESS}. This continues to be a rich area of research, see e.g. \cite{Vendraminetal22} and references therein.

The classification of objects in  $YB(\MonMat)$ is related to those in $YB(\PermMat)$ due to the following result (well known to experts, see also \cite{Nemec}):
\begin{proposition}
    Suppose $(N,R)\in YB(\MonMat)$, and let $R^{\times}$ be the matrix obtained from $R$ by replacing all non-zero entries in $R$ with $1$.
    Then $(N,R^\times)\in YB(\PermMat)$. 
\end{proposition}
\begin{proof}
For a monomial matrix $R$ with non-zero entries $q_{ij}$, both sides of the YBE are also monomial so that each non-zero constraint is of the form $q_{ij}q_{mn}q_{xy}=q_{ab}q_{cd}q_{ef}$.    In particular if every non-zero entry of $R$ is $1$, the YBE is trivially satisfied.
\end{proof}
It follows that any monomial Yang-Baxter operator is of the form $DP$ where $P$ is a permutation Yang-Baxter operator. and $D$ is a diagonal matrix.

Objects in $\MonFun(\Bcat,\UMat)$ are of  special interest for several reasons including their relevance in {quantum} physics and quantum computation.  There is a conjectural characterization (see \cite[Conjecture 2.7]{RowellWang12} and \cite[Conjecture 1.1]{GalindoRowell}): if $(N,R)\in YB(\UMat)$  then the braid group images $\rho_n^R(B_n)$ are \emph{virtually abelian}, that is, have a normal abelian subgroup of finite index.  This has been verified for many classes, see e.g. \cite{GalindoRowell}.

The \ppmm{solution to the} problem of classifying objects in $YB(\Match^N)$ of the form $(1,R)$ was a main inspiration for our \ppmm{the present} paper.  They are ubiquitous: the $R$-matrices associated with the `standard' $N$-dimensional representation of $U_q\mathfrak{sl}_N$ and $U_q\mathfrak{gl}(k,N-k)$ are of this type.  

The category $YB(\Matcha^N)$ is one natural generalisation of  $YB(\Match^N)$--they coincide for $N=2$, and the objects of the form $(1,R)\in YB(\Matcha^3)$ were classified in \cite{HMR1}, see below. For an example with $N=4$ see \cite{garoufalidis2025}. {For another example, we note that the $R$-matrix associated with the standard representation of $U_q\mathfrak{so}_3$ is in $YB(\Matcha^3)$:} 
{
\[
\begin{aligned}
R={}&
qE_{11,11}
+(q-q^{-1})E_{21,21}
+E_{21,12} \\
&+(q-q^{-1})(1-q^{-1})E_{31,31}
-(q-q^{-1})q^{-1/2}E_{31,22}
+q^{-1}E_{31,13} \\
&+E_{12,21}
-(q-q^{-1})q^{-1/2}E_{22,31}
+E_{22,22} \\
&+(q-q^{-1})E_{32,32}
+E_{32,23} \\
&+q^{-1}E_{13,31}
+E_{23,32}
+qE_{33,33}.
\end{aligned}
\]
}

\ignore{{\[\left[ \begin {array}{ccccccccc} q&0&0&0&0&0&0&0&0
\\ \noalign{\medskip}0&q-{q}^{-1}&0&1&0&0&0&0&0\\ \noalign{\medskip}0&0
& \left( q-{q}^{-1} \right)  \left( 1-{q}^{-1} \right) &0&-{ \left( q-
{q}^{-1} \right) {\frac {1}{\sqrt {q}}}}&0&{q}^{-1}&0&0
\\ \noalign{\medskip}0&1&0&0&0&0&0&0&0\\ \noalign{\medskip}0&0&-{
 \left( q-{q}^{-1} \right) {\frac {1}{\sqrt {q}}}}&0&1&0&0&0&0
\\ \noalign{\medskip}0&0&0&0&0&q-{q}^{-1}&0&1&0\\ \noalign{\medskip}0&0
&{q}^{-1}&0&0&0&0&0&0\\ \noalign{\medskip}0&0&0&0&0&1&0&0&0
\\ \noalign{\medskip}0&0&0&0&0&0&0&0&q\end {array} \right] \] }}
{which has the same pattern of non-zero entries as the matrix $S$ in example \ref{ex:projection}.}

\subsection{Monoidal Subcategories of $YB(\Mat)$}

Given a collection of objects $\cO\subset \ob(YB(\Mat))$ we may define the monoidal subcategory $\langle \cO\rangle\subset YB(\Mat)$ as the full \emph{monoidal} subcategory generated by $\cO$.  For example, consider a subset $\cO_\catT$ {$\subset \ob(YB(\Mat))$ consisting of the} $(N,R)$ such that $R\in\catT$ is some set of matrices, and look at the subcategory of $YB(\Mat)$ they generate.   If $\cO_\catT$ is closed under the lashing product we will denote this full monoidal subcategory of $YB(\Mat)$ by $YB(\catT)$.  We describe a few examples of these monoidal subcategories.

\subsubsection{Group-type Matrices}
Denote by $\GTMat\subset \coprod_{N\geq 1}\Mat^N(2,2)$ the collection of group-type matrices
 as defined in section \ref{ss:subcats}. The following are particular examples of group-type Yang-Baxter objects.
\begin{example}
    Define $R\ket{ij}=g\ket{j}\otimes \ket{i}$ for a fixed invertible matrix $g\in GL(N)$. A calculation shows: $R_1R_2R_1\ket{ijk}=g^2\ket{k}\otimes g\ket{j}\otimes\ket{i}=R_2R_1R_2\ket{ijk}$. This provides an infinite family of group-type $(N,R)\in YB(\Mat)$ in all dimensions.
\end{example}

\begin{example}
In dimension 3 we may take $g_2=g_3=\Id_3$ and \[g_1= \left[ \begin {array}{ccc} 1&0&0\\ \noalign{\medskip}0&{\frac{3}{5}}&
{\frac{4}{5}}\\ \noalign{\medskip}0&-{\frac{4}{5}}&{\frac{3}{5}}
\end {array} \right].\]  One easily verifies that the corresponding group-type matrix yields a Yang-Baxter object.
\end{example}

If $R,S$ are group-type matrices then their lashing product is too, whether they correspond to Yang-Baxter objects or not, as the following shows.
\begin{lemma}
Suppose $R\in\Mat^N(2,2),S\in\Mat^M(2,2)$ are group-type matrices.  Then $R\boxtimes S$ is also of group type.
\end{lemma}
\begin{proof}
 Suppose $R(x_i\otimes x_j)=g_i(x_j)\otimes x_i$ and $S(y_k\otimes y_\ell)=h_k(y_\ell)\otimes y_k$ where  $g_i\in GL(V)$ and $h_k\in GL(W)$ and $V,W$ have ordered bases $\{x_i\}$ and $\{y_k\}$, respectively. Then, with respect to the product basis $\{x_i\otimes y_k\}$ 

$R\boxtimes S(x_i\otimes y_k\otimes x_j\otimes y_\ell)=(g_i\otimes h_k)(x_j\otimes y_\ell)\otimes (x_i\otimes y_k)$ which is clearly of group type.  
\end{proof}

  The objects in $YB(\GTMat)$ correspond to  finite dimensional Yetter-Drinfeld modules \cite{AS} over the group generated by the $g_i$, and thus are in some sense classified.  Here the lashing product is precisely the Yang-Baxter operator corresponding to the direct product of the groups.  It is worth pointing out that for any group-type Yang-Baxter operator $R$, the image $\rho_R(B_n)$ is virtually abelian \cite{GalindoRowell}.

\subsubsection{Involutive Matrices}\label{sss:invol}

Involutive Yang-Baxter operators are also clearly closed under the lashing product: if $R,S$ are involutive then so is \[(\id_N\otimes P_{N,M}\otimes \id_M)(R\otimes S)(\id_N\otimes P_{M,N}\otimes \id_M).\] Thus we may define $YB(\Invol)$ to be the full subcategory of $YB(\Mat)$ generated by involutive matrices, and we see that these are of the form $(N,R)\in YB(\Mat)$ where $R$ is involutive.  Notice that $YB(\Invol)\cong \MonFun(\Bcat,\Invol)$ is equivalent to $\MonFun(\Sym,\Mat)$ where $\Sym$ is the symmetric group category obtained by imposing $\sigma^2=\id$ in $\Bcat$.  In particular the images of the braid groups under the corresponding representations are finite groups.  {In this setting it is particularly natural to consider $\MoonFun(\Bcat,\Invol)\cong\MoonFun(\Sym,\Mat)$, as is done in \cite{LechnerPennigWood}, since the representation theory is particularly well-behaved. }

\medskip

\ignore{{Here are some potential new directions suggested by the above discussion.
\begin{question}
\begin{enumerate}
    
    \item What is the subcategory generated by the unitary object $R$ of example \ref{ex:gauss2}?  
\end{enumerate}
Notice that the lashing product $R\boxtimes S$ is similar (i.e., conjugate) to the Kronecker product $R\otimes S$, so the spectrum of $R\boxtimes S$ is obtained as products of those of $R$ and $S$. In particular if $R$ has eigenvalues that are $n$th roots of unity, any subobject of a lashing power of $R$ also has eigenvalues in the set of $n$th roots of unity.
\end{question}}}

\mdef
As noted in the Introduction, the motivating examples in this paper are objects $F$ in $\MonFun(\Bcat,\Match^N)$,
which have all been completely classified, subject only to the condition $F(1)=1$.
It is instructive for a number of reasons  to explain what happens when 
this condition is lifted.
The most obvious new case here is $F(1)=2$, and this example is indeed illuminating, as we
show next.  

\newcommand{\x}{\ast}
\setcounter{MaxMatrixCols}{20}

Here is the `shape' of an $R$-matrix targeting $\Match^2$ with $F(1)=2$:
\beq  \label{eq:F12shape}
\matt{cccc|cccc|cccc|cccc} 
\x &   &   &   &   &&&&&&&&   \\
   &\x &\x &   & \x &   &&& \x &&& \\
   &\x &\x &   & \x &   &&& \x &&& \\
   &   &   &\x &    &      \x & \x & & & \x&\x &  & \x&&& \\ \hline 
   &\x &\x &   & \x &   &&& \x &&&  \\
   &   &   &\x &    &      \x & \x & & & \x&\x &  & \x&&& \\ 
   &   &   &\x &    &      \x & \x & & & \x&\x &  & \x&&& \\ 
   &&&&&&& \x  &&&& \x && \x & \x\\ \hline
   &\x &\x &    &\x  & &&& \x &&& \\ 
   &   &   &\x &    &      \x & \x & & & \x&\x &  & \x&&& \\ 
   &   &   &\x &    &      \x & \x & & & \x&\x &  & \x&&& \\ 
   &&&&&&& \x  &&&& \x && \x & \x\\ \hline
   &   &   &\x &    &      \x & \x & & & \x&\x &  & \x&&& \\    
   &&&&&&& \x  &&&& \x && \x & \x\\ 
   &&&&&&& \x  &&&& \x && \x & \x\\ 
   &&&&&&&&&&&&&&&\x 
\ttam 
\eq 
The natural comparison here is with $F(1)=1$ at rank $N=4$, which is completely understood.
In this case, for comparison, the shape is

\newcommand{\xxx}{} 

\[
\matt{cccc|cccc|cccc|cccc} 
\x &   &   &   &   &&&&&&&&   \\
   &\x &\xxx &   & \x &   &&& \xxx &&& \\
   &\xxx &\x &   & \xxx &   &&& \x &&& \\
   &   &   &\x &    &      \xxx & \xxx & & & \xxx&\xxx &  & \x&&& \\ 
\hline 
   &\x &\xxx &   & \x &   &&& \xxx &&&  \\
   &   &   &\xxx &    &      \x & \xxx & & & \xxx&\xxx &  & \xxx&&& \\ 
   &   &   &\xxx &    &      \xxx & \x & & & \x&\xxx &  & \xxx&&& \\ 
   &&&&&&& \x  &&&& \xxx && \x & \xxx\\ 
\hline
   &\xxx &\x &    &\xxx  & &&& \x &&& \\ 
   &   &   &\xxx &    &      \xxx & \x & & & \x&\xxx &  & \xxx&&& \\ 
   &   &   &\xxx &    &      \xxx & \xxx & & & \xxx&\x &  & \xxx&&& \\ 
   &&&&&&& \xxx  &&&& \x && \xxx & \x\\ 
\hline
   &   &   &\x &    &      \xxx & \xxx & & & \xxx&\xxx &  & \x&&& \\    
   &&&&&&& \x  &&&& \xxx && \x & \xxx \\ 
   &&&&&&& \xxx  &&&& \x && \xxx & \x \\ 
   &&&&&&&&&&&&&&&\x 
\ttam 
\]

Note that in this case $F(1)=1$ lies `inside' $F(1)=2$ but is substantially smaller.

\ignore{{
We have described a number of subcategories of $YB(\Mat^N)$ above, such as $YB(\Match^N)$.  However, classifying all such objects remains too wild a problem.  Instead, we are most interested in the objects of the form $(1,R)\in YB(\Match^N)$ or $YB(\Matcha^N)$.  }}

\mdef \label{def: YB^a} For $\X$ a subcategory of $\Mat^N$ we denote by $YB^a(\X)$ the full subcategory of $YB(\X)$ whose objects are of the form $(a,R)\in YB(\X)$.  Similarly we define $\MoonFun^a(\Bcat,\X)$ to be the full subcategory of $\MoonFun(\Bcat,\X)$ with objects $F$ such that $F(1)=a$ and $F(\sigma)\in\X(a,a)\subset\Mat^N(a,a)$.  Crucially, both $YB^a(\X)$ and $\MoonFun^a(\Bcat,\X)$ are groupoids: any morphism between $(a,R)$ and $(a,S)$ must be an isomorphism.

Notice that $YB^a(\Mat^N)$ and $\YBone(\Mat^{N^a})$, have the same objects since $(a,R)\in YB(\Mat^N)$ is the same as $(1,R)\in\YBone(\Mat^{N^a})$.  A morphism in $YB^a(\Mat^N)$ from $(a,R)$ to $(a,S)$ is a $Q\in\Mat^N(a,a)$ such that $(Q\otimes Q) R=S(Q\otimes Q)$, so $Q$ is an $N^a\times N^a$ matrix.  Similarly, a morphism in $\YBone(\Mat^{N^a})$ from $(1,R)$ to $(1,S)$ is a $Q\in\Mat^{N^a}(1,1)$ such that $(Q\otimes Q) R=S(Q\otimes Q)$, so $Q$ is again an $N^a\times N^a$ matrix.  Thus 
$\YBone(\Mat^{N^a})\cong YB^a(\Mat^N)$.

The example that inspired this categorical approach is the case of  $(1,R)\in \YBone(\Match^N)$, i.e., charge conserving $N^2\times N^2$ Yang-Baxter matrices.
Here $\YBone(\Match^N)$ is the full subcategory of $YB(\Match^N)$ with objects of the form $(1,R)$. As we are taking the full subcategory of $YB(\Match^N)$ (as opposed to the full subcategory of $YB(\Mat^N)$) the morphisms between $(1,R)$ and $(1,S)$ are of the form $Q\in\Match^N(1,1)$ such that $R(Q\otimes Q)=(Q\otimes Q)S$.  Notice that $\YBone(\Match^N)$ is not \emph{monoidal} even though $YB(\Match^N)$ is monoidal: indeed, $\Match^N$ is closed under the lashing product but $(1,R)\boxtimes(1,S)\in YB^2(\Match^N)$. 

\section{Equivalences}\label{s:equivandauto}

Let $\X$ be either a subcategory of $\Mat$ or a subset of matrices such as $\Invol$ as described in section \ref{ss:subcats}.
As we mentioned in the introduction, useful classifications of 
braid representations with target $\X$, or equivalently objects in $YB(\X)$
as in \ref{de:YBC}, 
will be up to some form of equivalence.   
Some, but not all, of these can be described in terms of autoequivalences of $YB(\X)$, i.e., endofunctors $F\in\End(YB(\X))$ such that there exists a $G\in\End(YB(\X))$ with $FG$ and $GF$ both naturally isomorphic to the identity functor.  
In this section we explore several notions of equivalence, illustrated throughout by examples.  We will see that some forms of equivalence are too fine 
{to yield a manageable or even feasible classification}
while others are too coarse {to meaningfully distinguish the cases}--the goal is to find  
{a}
`Goldilocks level' of equivalence leading to a feasible and meaningful classification. 

\medskip 

Recall {from the Introduction} that at heart 
we are studying 
(aiming to classify) 
the objects in $\MoonFun(\Bcat,\catC)$ for various subcategories $\catC$ of $\Mat$. 
Thus one natural 
part of any 
notion of equivalence arrives through the natural isomorphisms - the isomorphisms in this category. 
In terms of the  
diagram here:
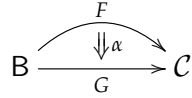
\[
\xymatrix@+24pt{ 
\Bcat   
\ruppertwocell^{F}_{}{\alpha}  \ar[r]_{G}  
         &   \catC 
}
\]
which is a picture of objects and morphisms in $\MoonFun(\Bcat,\catC)$,
we are saying $F,G$ are equivalent if 
(not necessarily only if)
there is a natural transformation $\alpha$ that  is a natural isomorphism. 
In the next picture, if we take it that every natural transformation shown is a natural isomorphism, then the `connected components' yield equivalence classes:
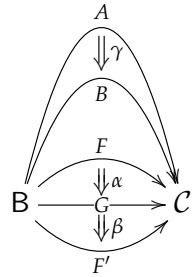
\[
%
\xymatrix@+24pt{ 
\Bcat   
\ar@/^4pc/[r]_{B}^{}="b"  
\ar@/^5.6pc/[r]^{A}_{}="a"  
\ruppertwocell^{F}_{}{\alpha}  
\ar[r]|-{G}  
\rlowertwocell^{}_{F'}{\beta} 
         &   \catC 
\ar@{=>}^{\gamma} "a";"b"
}
\ignore{
\hspace{3.1cm}  
\xymatrix@R+4pt{ 
\Bcat   
\ruppertwocell^{F}_{}{}  \ar[r]_{}  \rlowertwocell^{}_{F'}{\beta} 
\ar @/^3.5pc/ [rr]
         &   \catC 
\rtwocell^{G}_{Id}{\alpha}  &    \catC   \\
}}
\]

The natural isomorphisms themselves are amenable to partial classification,
according to how they can be constructed. 
For one thing we can separate into those that lie in $\MonFun(\Bcat,\catC)$ and those that do not
(for example recall that there are generally more isomorphisms in  $\MoonFun(\Bcat,\catC)$). 
But we can go further. 
The following schematics use some of the functors introduced in \S\ref{ss:Bcat}: 
\beq
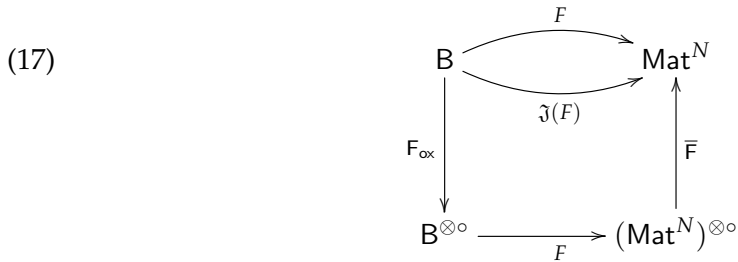
\xymatrix@C+24pt@R+24pt{ 
\Bcat   \ar[d]_{\Fox} 
\ar @/^1pc/ [r]^{F}_{}      
\ar @/_1pc/ [r]^{}_{\Jarmo(F)}      
         &   \Mat^N 
\\
       \Bcatmop  \ar[r]_{F}  &   (\Mat^N)^{\mop} \ar[u]_{\Farr} 
}
\eq
- observe here that a monoidal functor $F:C \rightarrow D$ yields identically a functor 
$F:C^{\mop} \rightarrow D^{\mop}$;
thus $F'$ is derived from $F$ such that the bottom square commutes.

Next we have:
\medskip \medskip 
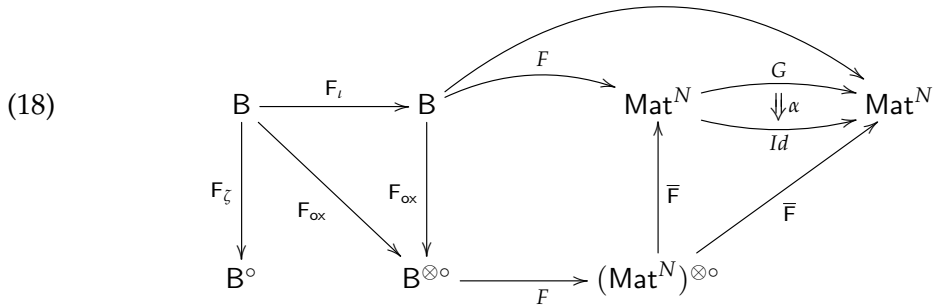
\beq 
\xymatrix@C+24pt@R+24pt{ 
\Bcat \ar[r]^{\Fio}   \ar[dr]_{\Fox}   \ar[d]_{\Feta} & 
\Bcat   \ar[d]_{\Fox} 
\ar @/^1pc/ [r]^{F}_{}   
\ar @/^3.15pc/ [rr]
         &   \Mat^N 
\rtwocell^{G}_{Id}{\alpha}  &    \Mat^N 
\\
\Bcat^\circ  
       &  \Bcatmop  \ar[r]_{F}  &   (\Mat^N)^{\mop} \ar[u]_{\Farr} \ar[ur]_{\Farr} 
}
\eq

Broadly 
we have the organisational scheme 
for constructing a key subclass of such isomorphisms 
indicated by the following diagrams, depicting composition of $F$ with isomorphism {\em functors} (i.e. not yet, and not necessarily leading to, 
natural transformations): 
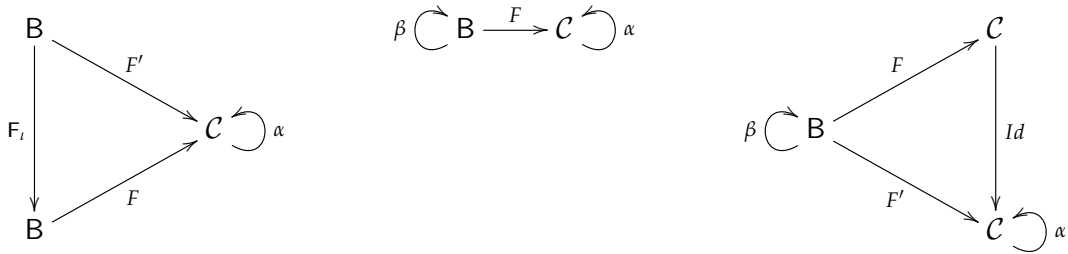
\[
\xymatrix{
\Bcat   
   \ar[dd]_{\Fio} \ar[drr]^{F'} &  \\ && \catC   \ar@(dr,ur)[]_{\alpha}
\\
\Bcat  \ar[urr]_{F}
}
\;\;\hspace{1cm}      
\xymatrix{
\Bcat   \ar@(dl,ul)[]^{\beta} \ar[r]^{F}   & \catC   \ar@(dr,ur)[]_{\alpha}
}
\;\;\hspace{1cm}        
\xymatrix{
             &           &  \catC \ar[dd]^{Id}
\\
\Bcat \ar@(dl,ul)[]^{\beta}   \ar[rru]^F  \ar[rrd]_{F'} & &  
\\
             &            &   \catC   \ar@(dr,ur)[]_{\alpha}
}
\]
In the central schematic we suppose that $\alpha, \; \beta$ are functors that are isomorphism functors - 
objects in the indicated functor categories. 
Then the corresponding equivalences 
are of the form $F \sim F\circ \beta$ and $F \sim \alpha\circ F$. 
We call a symmetry `global' 
if the transformation can be applied to any $F$ to obtain another
(once again it is instructive to compare and contrast with ordinary representation theory here - if we fix a vector space $V$ then every automorphism of this space can be applied to any $A$-module structure on $V$ to make an isomorphic $A$-module; but if we do not fix target $V$ then the algorithm for prescribing isomorphisms could in principle become more meta - although in this case it does not generalise further than to vector space isomorphisms). 
In this sense the constructions here indicate global symmetries. 
Unpacking an example, 
the diagram on the left says categorically that if we have a braid representation then we have another one by taking $F'(\sigma) = F(\sigma)^{-1}$ (we don't need any $\alpha$ here). 
Isomorphism of $\Fio$ follows from the 
symmetry of (\ref{eq:YBE0}) (or (\ref{eq:YBE}), or B1 and B2 in \S\ref{S:intro}) under reversal. 
(Observe that there is certainly not a natural isomorphism $\eta: F \Rightarrow F'$ in general.)
For another (meta) example $\beta$ could be the upside-down functor  
$\Feta$ for $\Bcat$: $\beta_0(n)=n$ on objects, and $\beta(g)$ for $g$ a morphism, i.e. a braid, is the upside-down  braid.
This takes $\sigma$ identically to $\sigma$, but passes to $\Bcat^\circ$; 
but then we can compose with transpose on the $\catC=\Mat^N$ side.
This tells us that any $F$ as above 
yields an $F'$ given by transpose:
\beq \label{eq:transpose-equiv} 
F'(\sigma) = F(\sigma)^{tr}
\eq 


Examples of all of these equivalences in action are given in \S\ref{s:classifications}. 

\medskip

 In Section~\ref{de:local-equiv} we discuss {specific} ways in which pairs of objects can be regarded as equivalent.
In Section~\ref{ss: symmetries} we consider {specific} global symmetries of our categories. 
In Section~\ref{ss:universal} we cast the problem in terms of `universality'.
In general fixing a source category $A$ (in our case this is $\Bcat$ - a given collection of groups;
but it could be some other collection of groups or algebras - indeed replacing $\Bcat$ with the linear version does not change our problem greatly, since group representations immediately give group algebra representations); and target $\X$, we are studying the objects of $\MonFun(A,\X)$.
A target $\X\subset\Mat$ is universal if it contains (in a suitable sense) every irreducible representation of every group/algebra in the source $A$ that appears as a factor in functors to $\Mat$. 
In \S\ref{ss:universal} subsection we will illustrate with some key examples.

\subsection{Equivalences between pairs of solutions}\label{de:local-equiv} 
Once one has a parametrisation of some partial classification it is natural to consider the solutions up to categorical isomorphisms.  There are choices to be made here as well.  If we are only interested in the representation-theoretic aspects 
(such as operator spectra)
as in many applications to physics, natural isomorphisms are a standard choice, i.e. equivalence in $\MoonFun(\Bcat,\Mat^N)$.  Another common choice is to consider monoidal natural isomorphisms, which preserve the local tensor product decomposition, as might be of interest in quantum information. 

\subsubsection{Local isomorphisms}  
For $\X\subset\Mat^N$ a subcategory, an isomorphism $\varphi_Q:(a,S)\cong (a,R)$ in $YB(\X)$ consists of an invertible $Q\in\X(a,a)$ such that $(Q\otimes Q)R=S(Q\otimes Q)$.   This is often called a \emph{local equivalence} between $R$ and $S$.  

Many classification schemes 
{in the literature}
incorporate local equivalence, for example in the set-theoretic setting {- see e.g.} \cite{ESS}, and for low ranks \cite{Hietarinta1992,HMR1}.

\subsubsection{$p$-equivalence}   \label{ss:p}
From a representation theory point of view, 
restricting  the operative notion of equivalence to 
local equivalence may be regarded as  
unnecessarily strong.
In this section we explain this, and introduce a framework for alternatives.

Let $(N,R)\in YB(\Mat)$ be a Yang-Baxter object. 
Recall that this just means that we have a braid representation given by $F(\sigma)=R$. 
Recall in particular from section \ref{S:intro}
that this gives a matrix representation 
of $B_n$ for each $n$, which we can write as 
$\rho^R_n : B_n \rightarrow \Mat^N(n ,n)$,
given by 
\beq  \label{eq:matrep}
\rho^R_n (\sigma_i) = 
\Id_N  \otimes  \cdots\otimes\Id_N \otimes R \otimes \Id_N \otimes \cdots \otimes \Id_N
\eq 
with the $R$ in the $i$-th position.

Now
let $(N,R),(M,S)\in YB(\Mat)$.
We say that these are 
{\em $p$-equivalent} 
if there is 
an invertible matrix $T_n$ for each $n\leq p$ such that 
\beq \label{eq:Tn}
T_{n}\rho_n^S
= \rho_n^R  \; T_{n}  .
\eq 

Observe that 1-equivalence of $(N,R)$ and $(M,S)$ only requires that $N=M$, while $2$-equivalence is the condition that $R$ and $S$ be similar matrices.
Of course $p$-equivalence implies $p-1$-equivalence. 
Finally, 
$\infty$-equivalence is precisely the same as a (not necessarily monoidal) natural isomorphism.
That is, we have the following.

\mdef  \label{pr:infMoon} 
Two braid representations (given by $R$ and $S$ say 
- so $R\in\Mat(N^2,N^2)$, i.e. $(N,R) \in \YB(\Mat)$, for some $N$), 
are $\infty$-equivalent 
iff they are isomorphic in $\MoonFun(\Bcat,\Mat)$
iff they are isomorphic in $\MoonFun(\Bcat,\Mat^N)$.  

\mdef Note that 2-equivalence does not imply $\infty$-equivalence (except, trivially, at rank-1).   
For example the `Ising' or Gaussian solution $R$ of Example \ref{ex:gauss2} (see \ref{ss:jarmo}) is similar to case $a$ of the $\Match^2$ classification \cite{MR1X} with parameters $\alpha=\zeta_8$ and $\beta=\zeta_8^{-1}$ where $\zeta_8:=\frac{1}{\sqrt{2}}+i\frac{1}{\sqrt{2}}$ explicitly:
\[R':=\left[ \begin {array}{cccc} \zeta_{{8}}&0&0&0\\ \noalign{\medskip}0&
\zeta_{{8}}+{\zeta_{{8}}}^{-1}&-{\zeta_{{8}}}^{-1}&0
\\ \noalign{\medskip}0&\zeta_{{8}}&0&0\\ \noalign{\medskip}0&0&0&{
\zeta_{{8}}}^{-1}\end {array} \right]. \]
However, these are not $3$-equivalent: for example, the traces of the images of $\sigma_1\sigma_2^{-1}$ under these two representations do not match, giving $6$ and $4$ respectively.  In fact $\rho_4^{R'}(B_4)$ is infinite, whilst $\rho_n^{R}(B_n)$ is finite for all $n$ \cite{FRW}.

\mdef {The fact that $2$-equivalence does not imply $\infty$-equivalence} can also be deduced for example from \cite{LechnerPennigWood}.  They show that there are exactly $20$ $\infty$-equivalence classes of unitary involutive $16\times 16$ Yang-Baxter matrices, corresponding to ordered pairs of Young diagrams with a total of 4 boxes between them.  Since  involutive $16\times 16$ matrices are diagonalisable, there are only $17$ distinct Jordan forms (given, say, by the dimension of the $-1$ eigenspace $0\leq E(-1)\leq 16$, and distinguished by the trace)
  there can be at most $17$ $2$-equivalence classes of involutive Yang-Baxter matrices.  
  {On the other hand,
  we do not know, for example, if $3$-equivalence implies $\infty$-equivalence.}
{It is known that $3$-equivalence between $R$ and $S$ is equivalent to the existence of a Drinfeld twist $(F,\Phi,\Psi)$ on $R$ such that $S=FRF^{-1}$, see e.g., \cite[Prop. 2.9]{Ferri2025}.  The question of when a Drinfeld twist can be extended to yield an $\infty$-equivalence is explored {for example} in \cite{KulishMudrov}, which suggests that one should not expect $3$-equivalence to imply $\infty$-equivalence.}
  
\mdef\label{2eigs} {In the case that $R$ and $S$ have exactly 2 eigenvalues $a$ and $b$ with $-a/b$ not a non-trivial root of unity, a criterion for $\infty$-equivalence could be derived from \cite{Davydov2000} and \cite[Chapter 5 Section 6]{Davydov1998}, in terms of the multiplicities of the $1$-dimensional $B_n$-subrepresentations associated with the eigenvalue $a$, for all $n$.  This would be an infinite number of conditions to check, but would be a valuable reduction nonetheless. }

\medskip

While local equivalence implies $\infty$-equivalence, the converse is false.  The following is an example: 
\begin{example}  \label{ex:9x9unitaries}
Set $\aaa =  \frac{\sqrt{3}+i}{2}$, a primitive 12-th root of unity.
Let 
$x=-\frac{1}{3} \left( 1+ \aaa^2 \right)$, $y=x+1$ and  
$ z =  -(x+y)$.
The following are in $YB(\UMat)$ (constructed as in \cite{Peking}, cf. \cite{RowellWang12})):
\[
R=\left[ \begin {array}{ccc|ccc|ccc} x&0&0&0&y&0&0&0&y
\\ \noalign{\medskip}0&x&0&0&0&x&z&0&0\\ \noalign{\medskip}0&0&x&z&0&0
&0&x&0\\ 
\hline 
\noalign{\medskip}0&0&x&x&0&0&0&z&0\\ \noalign{\medskip}y&0&0
&0&x&0&0&0&y\\ \noalign{\medskip}0&z&0&0&0&x&x&0&0
\\ 
\hline 
\noalign{\medskip}0&x&0&0&0&z&x&0&0\\ \noalign{\medskip}0&0&z&x&0&0
&0&x&0\\ \noalign{\medskip}y&0&0&0&y&0&0&0&x\end {array} \right],
\quad S=\left[ \begin {array}{ccc|ccc|ccc} x&0&0&y&0&0&y&0&0
\\ \noalign{\medskip}0&x&0&0&x&0&0&z&0\\ \noalign{\medskip}0&0&x&0&0&z
&0&0&x\\ 
\hline 
\noalign{\medskip}y&0&0&x&0&0&y&0&0\\ \noalign{\medskip}0&z&0
&0&x&0&0&x&0\\ \noalign{\medskip}0&0&x&0&0&x&0&0&z
\\ 
\hline 
\noalign{\medskip}y&0&0&y&0&0&x&0&0\\ \noalign{\medskip}0&x&0&0&z&0
&0&x&0\\ \noalign{\medskip}0&0&z&0&0&x&0&0&x\end {array} \right]
\]

One can easily verify computationally that these are not locally equivalent.  On the other hand,  they are $\infty$-equivalent, i.e., yield equivalent representations of $B_n$ for all $n$.  This is accomplished by a computation akin to that of \cite[Theorem 5.1]{RowellWang12}, passing through representations of the Temperley-Lieb algebra at $6th$ roots of unity.

\end{example}

\mdef The following is \cite[Lemma 3.1]{DS} (statement and proof adapted to our notation) which, as we will see, yields $\infty$-equivalence.

\begin{lemma}\label{le:DSequiv}
    Let $(N,R')\in YB(\Mat)$ and $Q\in \Aut(N,R')$ (as defined in (\ref{de:Aut}).
    Define \[R:=(Q\otimes I_N) R'(Q^{-1}\otimes I_N).
   \]
    Then  $(N,R)\in YB(\Mat)$.
\end{lemma}

\begin{proof} 
     We only need to check that \ppmm{the YBE} \eqref{eq:YBE} is satisfied.  The key observation is that \begin{equation}\label{dseq}
         (Q\otimes I_N) R'(Q^{-1}\otimes I_N)
    =(I_N\otimes Q^{-1})R'(I_N\otimes Q),
     \end{equation} since we may freely conjugate 
 $R'$ by $(Q\otimes Q)$ as they commute.
 
Let $R_1:=R\otimes \Id_N$, $R_2:=\Id_N\otimes R$,
$R'_1:=R'\otimes \Id_N$ and $R'_2:=\Id_N\otimes R'$ as usual.
Making liberal use of \eqref{dseq}, we have
\begin{eqnarray*}
   && R_1R_2R_1=\\&&(Q\otimes \Id_N\otimes \Id_N )R'_1 (Q^{-1}\otimes \Id_N \otimes Q^{-1})R'_2(Q\otimes \Id_N \otimes Q)R'_1(Q^{-1}\otimes \Id_N \otimes \Id_N)\\
&=&(Q\otimes \Id_N\otimes Q^{-1} )R'_1 R'_2R'_1(Q^{-1}\otimes \Id_N \otimes Q),\end{eqnarray*} and similarly \[R_2R_1R_2=(Q\otimes \Id_N\otimes Q^{-1} )R'_2 R'_1R'_2(Q^{-1}\otimes \Id_N \otimes Q).\]
    Thus, the result follows, using that $(N,R')\in YB(\Mat)$.
\end{proof}

We will say that such an $R'$ and $R$ are \emph{DS-equivalent,} where $DS$ is an abbreviation of Doikou-Smoktunowicz.

\begin{remark} 
    Suppose $Q,Q'\in\Aut(N,R')$ as in   Lemma~\ref{le:DSequiv}. 
    Then conjugating by $(Q\otimes Q'^{-1})$ is the same as conjugating by $(QQ'^{-1}\otimes \Id_N)$. 
\end{remark}

{
In fact this DS-equivalence implies $\infty$-equivalence:}
\begin{theorem}\label{th:ds is infinity}
 Let $(N,R)\in YB(\Mat)$.
Suppose that $Q\in \AutYB(N,R)$, and define $S=(\Id_N\otimes Q)R(\Id_N\otimes Q^{-1})$.  Then $R$ and $S$ are $\infty$-equivalent.   
\end{theorem}\begin{proof} Define $A_n:=\Id_N\otimes Q\otimes\cdots\otimes Q^n$.  We claim that $A_n\rho_n^RA_n^{-1}=\rho_n^S$.  It is clearly true for $n=1,2$.  Assuming it holds for some $n$, to verify it for $n+1$ only requires checking that $A_{n+1}\rho_{n+1}^R(\sigma_n)=\rho_{n+1}^S(\sigma_n)A_{n+1}$. This is as follows:
\[ A_{n+1}\rho_n^R(\sigma_n)A_{n+1}^{-1}=\Id_N^{\otimes n-1}\otimes[(Q^{n}\otimes Q^{n+1})(R)(Q^{n}\otimes Q^{n+1})^{-1}]=\Id_N^{\otimes n-1}\otimes S=\rho^S_n(\sigma_n)
\]
since $(Q\otimes Q)^{n}$ commutes with $R$. 
\end{proof}


\begin{example}  Local equivalence does not imply DS-equivalence, and hence in particular $\infty$-equivalence does not imply $DS$-equivalence.  For an example, 
    consider the monomial solution $ \left[ \begin {array}{cccc} k&0&0&0\\ \noalign{\medskip}0&0&p&0
\\ \noalign{\medskip}0&q&0&0\\ \noalign{\medskip}0&0&0&s\end {array}
 \right].$  For distinct $k,q,p,s$ the only invertible $2\times 2$ matrices that commute with $R$ are the diagonal matrices.  Therefore, its DS equivalence class consists of monomial matrices as well.  Thus it is easy to find a locally equivalent YBO that is not DS equivalent, simply conjugate by $A\otimes A$ for some non-monomial matrix $A$.
\end{example}

\subsection{Symmetries of Categories}\label{ss: symmetries}

{In a 
classification program one may first use 'global' symmetries to reduce the computational complexity of the problem.  This can often facilitate finding a normal form, or reducing the number of parameters.  Categorically these 
symmetries 
sometimes take the form of autoequivalences.   Among the autoequivalences are those that are essentially internal to the category.  For example, in $YB(\Mat^N)$ a local basis change by means of an invertible $N\times N$ matrix $Q$ are of this type.  Conjugation by the flip matrix is not of this type, but is an autoequivalence of $YB^1(\Mat^N)$ \ref{lem:flip is auto for ybmatn}. Even more simple is the autoequivalence that rescales all the $(N,R)$ in $YB(\Mat^N)$. Beyond autoequivalences one could consider the transpose operation, which is not an endofunctor from $YB(\Mat^N)$, but could be used to reduce parameters.  We will explore various types of global symmetries.
\subsubsection{Inner Autoequivalences}  \label{ss:local} 

\begin{defin}
    For any category $\catC$, an autoequivalence $\F\colon \catC \to \catC$ is \emph{inner} if it is naturally isomorphic to the identity functor.
\end{defin}
 The following inner autoequivalences take a particularly simple form, and will be useful in the sequel. \begin{lemma}\label{def: PhiQ}
    Suppose $\X\subset\Mat^N$  and $Q\in\X(a,a)$. Define, on $YB^a(\X)$, the functor $\Phi_Q$ by \[\Phi_Q((a,R))=(a,Q^{\otimes 2} R(Q^{-1})^{\otimes 2})\] on objects and similarly $\Phi_Q(f)=QfQ^{-1}$ on morphisms $f:(a,R)\rightarrow (a,S)$.  Then $\Phi_Q$ is an inner autoequivalence. 
\end{lemma}
\begin{proof}  It is immediate that $\Phi_Q$ defines a functor with inverse $\Phi_{Q^{-1}}$.
    A natural isomorphism $\eta$ between $\Phi_Q$ and $\Id$ is  given in components by 
\[\eta_{(a,R)}=Q^{-1}\in \HomYB(\Phi_Q((a,R)),(a,R)),\] 
as one can check that $(Q^{-1})^{\otimes 2}$ intertwines $Q^{\otimes 2} R(Q^{-1})^{\otimes 2})$ and $R$.
\end{proof}

\mdef Although we will not have a particular use for them here, we can characterise inner autoequivalences more generally. Consider $\Psi$ such an inner autoequivalence of $YB^a(\X)$ for some $\X\subset \Mat^N$. What does the natural isomorphism $\eta: \Psi\Rightarrow\Id$ look like? Notice that the only morphisms are isomorphisms. The components are isomorphisms \[\eta_{(a,R)}:\Psi(a,R)=(a,\psi(R))\cong (a,R),\] represented by an invertible $Q_R\in\X(a,a)$ such that $Q_R\otimes Q_R\psi(R)=RQ_R\otimes Q_R$.  In particular $\psi(R)=(Q_R\otimes Q_R)^{-1}R(Q_R\otimes Q_R).$  If $\varphi:(a,R)\rightarrow (a,S)$ is an (iso)morphism given by $A\in\X^a(1,1)$ with $(A\otimes A)R=S(A\otimes A)$ then $\Psi(\varphi):(a,\psi(R))\rightarrow (a,\psi(S))$ is given by the composition $Q_S^{-1}AQ_R$.  In particular \emph{any} inner autoequivalence in $YB^a(\X)$ may be constructed by choosing an invertible $Q_R\in\X(a,a)$ for each $(a,R)$, defining $\Psi(a,R)=(a,Q_R^{-1}RQ_R)$ on objects and as above on morphisms.

\mdef The general inner autoequivalences of $\MoonFun^a(\Bcat,\X)$ for $\X\subset \Mat^N$ are similarly classified.  For each object $F$ we fix a natural isomorphism $\xi_F:F\rightarrow \tau(F)$ (see \ref{why not moon} on how to assemble such isomorphisms).  Then define $\Xi$ on objects by $\Xi(F)=\xi_F(F)$.  For morphisms $\varphi:F\rightarrow G$ we define $\Xi(\varphi)$ by $\xi_G^{-1}\varphi\xi_F$.  

\begin{example}
    Consider $\YBone(\Match^N)$. The morphisms $Q\in\Match^N(1,1)$ are all diagonal matrices.  Thus the inner autoequivalences of $\YBone(\Match^N)$ of the form $\Phi_Q$ are not particularly discerning from the perspective of classification.  
\end{example}

\begin{example}
    The inner autoequivalences of $\YBone(\X)$ of the form $\Phi_Q$ for $\X\subset \Mat^N$ a monoidal groupoid such as $\PermMat^N,\MonMat^N,\UMat^N$ and $\OMat^N$ correspond to the $Q\in\X(1,1)$, which in this case are exactly the $N\times N$ matrices of the given form.  
\end{example}

\subsubsection{More General Symmetries}  \label{ss:moregenauto}

There are natural non-inner autoequivalences of $\YBone(\Mat^N)$ that are useful for classification.  For example:
\begin{lemma}\label{lem:flip is auto for ybmatn}
 Let $P:=P_{N,N}\in\Mat^N(2,2)$ be the usual flip, {as in (\ref{ybmat is monoidal})}.  
Then $(1,R)\mapsto (1,PRP)$ is an non-inner autoequivalence of $\YBone(\Mat^N)$ that is the identity on morphisms.
\end{lemma} \begin{proof}
    It is well-known that $(1,PRP)\in \YBone(\Mat^N)$ iff $(1,R)\in \YBone(\Mat^N)$.  We must check that if $A\in\HomYB((1,R),(1,S))$ then $A\in\HomYB((1,PRP),(1,PSP))$.  But since $P$ commutes with $A\otimes A$ this is immediate.  On the other hand, conjugation of $R$ by $P$ is not equivalent to conjugation by some $Q\otimes Q$ with $Q$ commuting with every $A\in\End(N)$, so this is not an inner autoequivalence.
\end{proof}    

 \begin{theorem}\label{thm:flipliftstoequiv}
    The symmetry $R\mapsto PRP$ lifts to an $\infty$-equivalence, and thus, in particular, is an autoequivalence of $\MoonFun^1(\Bcat,\Mat^N)$.
\end{theorem}

\begin{proof}  This is an essentially standard fact, but we provide the details for completeness.
    Let $(1,R)\in YB(\Mat^N)$, and let $\rho_n$ be the representation associated with $R$.  Set $R'=PRP$, and denote by $\rho'_n$ the corresponding representation.  Define $\Omega_n\ket{i_1\cdots i_k}=\ket{i_k\cdots i_1}$ to be the tensor-factor-reversing matrix and denote by $\Delta_n$ the Garside half-twist braid satisfying $\Delta_n\sigma_i\Delta_n^{-1}=\sigma_{n-i}$ for $1\leq i\leq n-1$ \cite{Garside}.  We claim that $\Gamma_n:=\Omega_n\rho_n(\Delta_n)$ intertwines $\rho_n$ and $\rho'_n$.  Indeed $\Omega_n\rho_n(\sigma_i)\Omega_n=\rho'_n(\sigma_{n-i})$ since reversing the tensor factors moves $R$ acting on the $(i,i+1)$ factors to $R'$ acting on the $(n-i,n-i+1)$ factors. So if we conjugate by $\rho_n(\Delta_i)$ followed by $\Omega_n$, we obtain the result.
\end{proof}

\ignore{{
The transpose operation $(N,R)\mapsto (N,R^T)$ yields an anti-autoequivalence of $YB(\Mat)$, acting on morphisms by $A\mapsto A^T$.  This is due to the contravariant nature of the transpose operation.  Nonetheless, this symmetry can be useful in classifications \cite{HMR1}.
\ft{[It feels like there is a sentence missing here, at the moment I can't understand what the 'this` is referring to, or why it makes sense for the last sentence to start `Nonetheless'. Doesn't $R\mapsto R^T$ yield a contravariant functor, so isn't the point here that this is not an auto equivalence, but an equivalence between  $YB(\Mat)$ and $YB(\Mat)^{op}$?]}\ecr{[this is addressed elsewhere, so will presumably be deleted.  There are bookkeeping symmetries and real symmetries, is perhaps the point: scalar multiples for example.  Transpose doesn't even preserve natural isomorphism, since it swaps socles and cosocle?]}}}

As we have seen, inner autoequivalences for $\YBone(\X)$ with $\X\subset\Mat^N$ may not provide useful symmetries.  A very natural idea is to look at the inner autoequivalences of $\YBone(\Mat^N)$ that stabilise $\YBone(\X)$.  This can mean two things: stabilise the objects, or restrict to an autoequivalence.
For $\X=\Match^N$ we have:
\begin{theorem}\label{thm: outer auto match}
       The inner autoequivalences of $\YBone(\Mat^N)$ of the form $\Phi_Q$ that restrict to an autoequivalence of $\YBone(\Match^N)$ are those with $Q$ an invertible $N\times N$ monomial matrix.
\end{theorem}

\begin{proof} 
    It is shown in \cite{MR1X} that if $Q$ is monomial then conjugation by $Q\otimes Q$ preserves charge conservation. Since the morphisms in $\YBone(\Match^N)$ are diagonal, and conjugation by a monomial matrix preserves this property, we see that $\Phi_Q$ with $Q$ monomial restricts to an autoequivalence of $\YBone(\Match^N)$.  Thus it is enough to show that this property is \emph{only} enjoyed by monomial $Q$.  

    Let $S$ be a charge conserving YBO such that $S\ket{jj}=\beta_j\ket{jj}$ has each $\beta_j$ distinct, {note that such an $S$ does exist as explained in Section~\ref{s:classifications}}. Now suppose that $(Q\otimes Q)^{-1}S(Q\otimes Q)=R$ is charge conserving, and define $\alpha_i$ by $R\ket{ii}=\alpha_i\ket{ii}$.  Now fix $i$, and let $Q\ket{i}=\sum_j q_{ij}\ket{j}$ so that $Q\otimes Q\ket{ii}=\sum_j(q_{ij})^2\ket{jj}+\sum_{j\neq k}q_{ij}q_{ik}\ket{jk}$. Since $S(Q\otimes Q)=(Q\otimes Q)R$ and both $R$ and $S$ leave invariant the subspace spanned by those $\ket{jk}$ with $j\neq k$, we find that \[\sum_j\alpha_i(q_{ij})^2\ket{jj}=\sum_j(q_{ij})^2\beta_j\ket{jj},\] for any fixed $j$.  In particular, if $q_{ij}\neq 0$ we have $\alpha_i=\beta_j$.  But since the $\beta_j$ are distinct, for each fixed $i$ there is at most, therefore, exactly one $j$ for which $q_{ij}\neq 0$.  Thus $Q$ is a monomial matrix.
    
   \end{proof}

 Since  $\Match$ is contained in the diagonal of $\Matcha$, the corresponding $\YBone(\Matcha^N)$-stabilising autoequivalences of $\YB^1(\Mat^N)$ of the form $\Psi_Q$ also have $Q$ monomial matrices.  In fact, we expect it to be much smaller:
\begin{conjecture}\label{conj: Matcha outer}
       The inner autoequivalences of $\YBone(\Mat^N)$ of the form $\Phi_Q$ that stabilise $\YBone(\Matcha^N)$ are those with $Q=DS$ where $D$ is an invertible diagonal matrix and $S$ is the permutation matrix associated with the transposition $(1\; N)(2\;N-1)\cdots$.
\end{conjecture}
\mdef To see that these autoequivalences do stabilise $\YBone(\Matcha^N)$ observe that  diagonal matrices clearly stabilise $\YBone(\Match^N)$, so it is enough to show that $(S\otimes S) R(S\otimes S)$ is additive charge conserving if $R$ is.   But notice that $S\otimes S$ interchanges the spaces $V^{k}=\C\{\ket{ij}:i+j=k\}$ and $V^{2N+2-k}$, which are preserved by $R$, so each $V^{k}$ is invariant under $(S\otimes S) R (S\otimes S)$. 

 On the other hand, it is straightforward to prove that if an inner autoequivalence of $\YBone(\Mat^N)$ preserves $\Matcha^N(2,2)$ then it must be of the above form. Indeed, there exists an $M\in\Matcha^N(2,2)$ such that for {every} quadruple $(i,j,a,b)$ with $i+j=a+b$ we have $M_{i,j}^{a,b}\neq 0$.  Then $M$ preserves $V_k:=\C\{\ket{ij}: i+j=k\}$ for all $1\leq k\leq N$.  Since any $Q$ such that $Q\otimes QM(Q\otimes Q)^{-1}\in\Match^N(2,2)$ is monomial, then the dimension $Q\otimes QV_k$ must be the same as $\dim(V_k)$.  Thus $Q\ket{1}$ must be either $\ket{1}$ or $\ket{N}$, since the only two 1-dimensional $V_k$ are for $k=2$ and $k=2N$.  An easy induction on these two cases shows that $Q$ must be either a diagonal matrix times the identity matrix or a diagonal matrix times the order two matrix above.  It is conceivable (although unlikely) that the the smaller set of Yang-Baxter matrices that lie in $\Matcha^N$ have a larger stabiliser, but lacking a classification we cannot rule this out.

\mdef   \label{pa:X}
For $\X\subset\Mat^N$, $\YBone(\X)$ may have other natural symmetries, beyond those that come from stabilising inner autoequivalences of $\YBone(\Mat^N)$ of the form $\Phi_Q$.  One example is the following: in \cite{MR1X}, it is shown that if $(1,R)\in \YBone(\Match^N)$ then $(1,XRX^{-1})\in \YBone(\Match^N)$ for any invertible \emph{diagonal} $X\in\Mat^N(2,2)$.  However, this does not generally lift to a inner autoequivalence of $\YBone(\Mat^N)$.  Indeed, for $f\in\Hom((1,R),(1,S))$ we would need to define $F(f)$ an isomorphism from $XRX^{-1}$ to $XSX^{-1}$.  A condition that would suffice is if $F(f)\otimes F(f)=X^{-1}f\otimes fX$, but if $X$ is not a Kronecker product of some fixed matrix $Y$, this is not possible to accomplish uniformly.  In \cite{MR1X} this was called $X$-symmetry.  We do not know, but 
{based on low-rank testing}
strongly suspect, that $(1,R)$ and $(1,XRX^{-1})$ are $\infty$-equivalent, i.e., correspond to isomorphic objects in $\MoonFun^1(\Bcat,\Mat^N)$, {and, in fact, may be implemented as an inner autoequivalence of $\MoonFun^1(\Bcat,\Mat^N)$.}
Concretely, we have:
\begin{conjecture}\label{conj: X is natural}
Let $N \in \N$. 
    Let $(1,R)\in \YBone(\Match^N)$, $X$ a diagonal matrix in $\Mat^N(2,2)$ and define $S=XRX^{-1}$
{(recall from (\ref{pa:X}) or \cite{MR1X}  that this means that $S$ is a braid representation)}.  
    Then there exists a sequence of diagonal matrices $A_n\in \Mat^N(n,n)$ {independent of $R$} 
{(depending only on $X$)}
    such that $A_n\rho_n^R=\rho_n^S A_n$ {for all $n$}.  
   Thus, the assignment $R\mapsto XRX^{-1}$ lifts to an inner autoequivalence of $\MoonFun^1(\Bcat,\Match^N)$.
{ And in particular to an $\infty$-equivalence.} 
\end{conjecture}
To assemble this to an inner autoequivalence of $\MoonFun^1(\Bcat,\Mat)$ we define $\Psi(f)_n=A_nfA_n^{-1}$ for any natural isomorphism $f:F_R\rightarrow F_S$ with $F_R(\sigma)=R$ and $F_S(\sigma)=S$.

In fact, this appears to be a property of $\Match^N$, 
as it doesn't seem to rely on the Yang-Baxter equation. In the case $N=2$ we have the following:
\begin{lemma} Let $R:=\left[ \begin {array}{cccc} \alpha&0&0&0\\ \noalign{\medskip}0&a&b&0
\\ \noalign{\medskip}0&c&d&0\\ \noalign{\medskip}0&0&0&\beta
\end {array} \right]
\in\Match^2(2,2)$  and $X=diag(w,x,y,z)$. 
Fix $n \in \N$. Define $S=XRX^{-1}$ and 
$A(i)   
:=diag(\left(\frac{y}{x}\right)^{n-i},1) \in \Match^2(1,1)$.  
Then setting $R_i=\Id_2^{\otimes i-1}\otimes  R\otimes \Id_2^{\otimes n-i-1}$ and $S_i=\Id_2^{\otimes i-1}\otimes  S\otimes \Id_2^{\otimes n-i-1}$ and 
$A 
= A(1)\otimes A(2) \otimes \cdots \otimes A(n)$ 
(all in $\Match^2(n,n)$) 
we have 
$A R_i A^{-1}=S_i$ 
for all $1\leq i\leq n-1.$

\end{lemma}
\begin{proof}
    First observe that conjugating $R_i$ by a diagonal matrix only changes off-diagonal entries. We denote by $\ket{1}$ and $\ket{2}$ the standard basis for $\C^2$.  In this notation we have $\bra{12}S\ket{21}=\frac{by}{x}$ and $\bra{21}S\ket{12}=\frac{cx}{y}$.  Thus, for example, 
    \[\bra{j_1\cdots j_{i-1}12j_{i+2}\cdots j_n}S_i\ket{j_1\cdots j_{i-1}21j_{i+2}\cdots j_n}=\frac{by}{x}.\]
Now since 
$A(i) \ket{1}=\left(\frac{y}{x}\right)^{n-i}\ket{1}$ 
and 
$A(i) \ket{2}=\ket{2}$ we compute:
\begin{eqnarray*}
   && \bra{j_1\cdots j_{i-1}12j_{i+2}\cdots j_n}
 A R_i A^{-1} 
   \ket{j_1\cdots j_{i-1}21j_{i+2}\cdots j_n}
   =\\&&\bra{12}
    A(i)\otimes A(i+1) R (A(i)   \otimes    A(i+1))^{-1}
  \ket{21}
  =\left(\frac{y}{x}\right)^{n-i}b\left(\frac{y}{x}\right)^{i+1-n}
  =\frac{by}{x}.
\end{eqnarray*}  
   A similar computation shows that $S_i$ and 
   $A R_i A^{-1} $  
   coincide on vectors of the form \[\ket{j_1\cdots j_{i-1}12j_{i+2}\cdots j_n}.\] Since the remaining entries of $R_i$ are diagonal, this completes the proof.
\end{proof}
This immediately implies:
\begin{corollary} Conjecture \ref{conj: X is natural} is true for $N=2$.
\end{corollary}

\begin{remark}
    A computer calculation verifies the validity of Conjecture \ref{conj: X is natural} for $N=3$ up to $n=5$.  However, the corresponding diagonal matrices $A_n$ do not factorise into local factors.  Note also that there is no requirement that the $A_n$ be diagonal for the $X$ symmetry to correspond to a natural isomorphism, \emph{per se} but the computational experiments suggest this simplification.
\end{remark}

\mdef 
{Question: Is there a remnant of this in higher levels? I.e. when $F(1)>1$.} Does X symmetry generalise? If so, is it an inner autoequivalence?
\medskip

{The following application illustrates the hierarchy of equivalences in a well-studied setting: set-theoretic solutions in rank $3$.}
\begin{example}\label{ex:rank 3 ST}
    We may apply our methods to $\YB^1(\PermMat^3)$ (recall Def. \ref{def: YB^a}).  As is well-known, a complete enumeration of solutions yields 73 objects $(3,R)$. First we partition these into 8 classes by cycle type.  Then within each class we stratify by whether they are equivalent under: 1) inner automorphisms of the form $\Phi_Q$, i.e., diagonal $S_3$-action 2) outer automorphisms, of the form $\Phi_Q$ i.e., diagonal $GL_3$-action and 3) DS equivalence. We carried out this computational sorting job and obtained the following: 
\begin{center}
\small
\begin{tabular}{l|rrrr}
\text{cycle type}
& \text{raw}
& \(S_3\)
& \(GL_3\)
& DS\\
\hline
\(1+1+1+1+1+1+1+1+1\) & 1 & 1 & 1 & 1\\
\(1+1+1+1+1+2+2\) & 6 & 2 & 2 & 1\\
\(1+1+1+2+2+2\) & 12 & 5 & 4 & 2\\
\(1+1+1+2+4\) & 12 & 4 & 3 & 1\\
\(1+1+1+3+3\) & 12 & 6 & 4 & 1\\
\(1+2+2+4\) & 12 & 4 & 4 & 2\\
\(1+4+4\) & 12 & 4 & 3 & 2\\
\(3+6\) & 6 & 3 & 3 & 1\\
\hline
\text{total} & 73 & 29 & 24 & 11
\end{tabular}
\normalsize
\end{center}
In particular there are at most 11, and at least 8, $\infty$-equivalence classes.  In fact, results of \cite{LechnerPennigWood} show that the two DS inequivalent classes in the third (involutive) cycle type are $\infty$-equivalent.  The first two classes contain all of the solutions that are degenerate in the sense of \cite{Vendraminetal22}--and this fits with the known result that there are 26 non-degenerate solutions up to basis permutation.

\end{example}

\ignore{{
Since $\MonFun(\Bcat,\Mat^N)$ is a subcategory of $\Fun(\Bcat,\Mat^N)$ one may also ask which isomorphisms in $\Fun(\Bcat,\Mat^N)$, yield auto-equivalences of $YB(\X)\subset YB(\Mat^N)$.  That is, we may drop the requirement that a natural isomorphism be monoidal, but still require that the image be in $\MonFun(\Bcat,\X)$.  This could then be used to classify members of $YB(\X)$ by their corresponding representations of $B_n$, since such a natural isomorphism is given by a family $T_n\in\Aut^N(n,n)$ of matrices such that $(N,T_2RT_2^{-1})\in YB(\X)$.  Of course we do not know if the $B_n$ representations associated with $(N,R)$ and $(N,T_2RT_2^{-1})$ are equivalent for all $n$, since the $T_n$ are not required to be related for distinct $n$.  We may also stratify these in a manner similar to $p$-equivalence described above.  Namely we can define the \emph{$p$ supported automorphisms}  of $YB(\X)$ to be those that have $T_n=\Id$ for all $n>p$.  Of greatest interest are the $2$-supported automorphisms.  For example it is shown in \cite{MartinRowellTorzewska} that taking $T_2=D\in\Mat^N(2,2)$ a diagonal matrix and $T_n=\Id$ for all $n>2$ yields a $2$-supported automorphism of $YB(\Match^N)$.
}}

\subsection{Continuity and varieties of solutions}   $\;$ 

Ultimately, good notions of equivalence of braid representations (and indeed in general) are those that help to organise and understand. 
Such notions can of course extend beyond the artificial constraints of an extrinsically arbitrary choice of categorical context.   
As many of our examples show, 
braid representations  
often arrive in varieties - i.e. in sets 
indexed by free or almost free complex parameters (parameters selectable from Zariski-open sets). 
For a moment let us write $F_x \colon\Bcat \rightarrow \Mat$ for such a variety, 
where the $x$ represents a suitable suite of parameters. 
If the resultant representation theory depends continuously on the parameters then it is natural to classify the whole variety together, even though different points in the variety cannot in general be `intertwined' (indeed the spectrum of $F_x(\sigma)$ may vary with $x$). 
For example if the representation theory is semisimple then all the algebras arising will be isomorphic as algebras (if not isomorphic as braid representations), since isomorphism classes of semisimple algebras are discrete, and continuous variations of discrete variables are necessarily constant. 

This aspect of classification leads to profound differences in framework, for example if we restrict the source. An obvious example is restricting from $\Bcat$ to $\Sym$. In $\Sym$ the spectrum of $F(\sigma)$  is restricted to $\{1,-1\}$, which massively reduces the possible complexity of varieties. 
We will return to this point in \S\ref{s:classifications}, where we will have a helpful profusion of illustrative examples. 
For now we observe only that the symmetry of overall scaling (if $F(\sigma)$ is a solutions then $F'(\sigma) = 3F(\sigma)$, and so on, is also a solution) is a very simple example of this type.

\subsection{Universality}  \label{ss:universal}

Given a classification problem, a natural question that arises is what notion of equivalence is to be applied - and then the problem is replaced by the problem of classifying up to the given equivalence. 
Sometimes it will be obvious what the appropriate notion of equivalence is. Sometimes not.
If not, then we observe that there are two extremal cases: {\em discrete}, in which the relation is equality; and {\em trivial}, in which the whole set is one class. The latter makes classification of classes easy, but the classification is useless. The former is useful in principle, but the classification problem takes its hardest form. 
In our case 
(unlike `ordinary'  
representation theory of algebras)
there is not a canonical notion of equivalence other than discrete. However there are several possible notions of equivalence that may lie in the sweet-spot of utility while still making classification easier. 

Suppose we pick one of these notions. Now for each target restriction there is a natural question: does the target contain a transversal of classes. 
In this case we say that the target is {\em universal}. 

There is another related notion of universality specific to our 
(higher but) linear representation theoretic setting. 
Here in general we are studying representations $F:C \rightarrow D$ where $C$ is a quotient of 
the linearisation of $\Bcat$ - so each $C(n,n)$ is an algebra.
We say that a target is {\em (representation)-universal} if every algebra 
$C(n,n)$ is faithfully represented  by some $D$-rep.
(The essential point here is a computational-physics one: our target is universal if it 
captures every eigenvalue in the spectrum of every operator.
For this it is sufficient that every irreducible representation arises; and 
faithfulness is sufficient for this.)

It is convenient to illustrate with an example directly relevant to our case
- classifying braid representations $F:\Bcat\rightarrow \Mat$. 
Indeed the following example illustrates several aspects of our case. 
In particular we have the strategy of simplifying the problem by restricting the target,
and we have also discussed restricting the source. 

\newcommand{\Hcat}{{\mathsf{H}}}

An example of restriction of source is to pass to the Hecke category $\Hcat(q)$ - here $q \in \C$. It is the same problem, except that one imposes also a quadratic relation on $\sigma$, depending on $q$. Notice here that any classification of braid representations passes relatively straightforwardly to a classification of Hecke representations --- simply remove all solutions that fail to satisfy the new condition. Thus Hecke is a nice sub-problem of our main problem. 
(Even nicer, from the computational-physics perspective, are the sequence of further quotients - each quotienting by a fixed Young $q$-antisymmetriser - sometimes called TLM categories \cite{Brzezinski1995}. For example the rank 2 case is the diagonal Temperley--Lieb category itself.) 

This can also be viewed as a target restriction. Simply restrict to $R$-matrices with the appropriate two-eigenvalue spectrum. This is not a restriction to a target category, but we already have several examples of this kind of non-categorical restriction. 

Observe that $\Hcat(q)$ is again diagonal. It consists of all the Hecke algebras $H_n(q)$. 
As such we have $\Hcat(q)$ representation-universality for the CC/$\Match$ categories taken holistically.

\mdef {\bf{Theorem}}.  Every Hecke algebra $H_n(q)$ is faithfully represented by 
some CC representations, i.e. by some $F:\Bcat\rightarrow\Match^N$ for some $N$. 

\medskip 

\noindent {\em Proof}. 
For $n \in \N$ write $e_n \in H_n(q)$ for the preidempotent (defined up to scalars) that induces the $q$-alternating representation (N.B. for $q\in\C$ a root of unity the preidempotent is not generally normalisable as an idempotent). 
Observe that $e_n \in H_m(q)$ for all $m \geq n$ by the inclusion $H_n(q) \hookrightarrow H_m(q)$ which takes generator $g_i$ to $g_i$. 
For $N \in \N$ define the quotient $H^N_n(q) = H_n(q)/e_{N+1}$ (this is the quotient by the ideal generated by $e_{N+1}$ if defined in $H_n(q)$, or by 0 otherwise). 
Then by \cite{Martin91} the $U_q sl_N$ spin-chain representation is CC and is faithful on the quotient $H^N_n(q) = H_n(q)/e_{N+1}$. 
Fix $n \in \N$ and consider $N=n$. 
Of course with $n=N$ then $e_{N+1} = e_{n+1} \not\in H_n(q)$ 
so $H_n(q)/e_{n+1} = H_n(q)$.  \qed

\section{Classifications}\label{s:classifications}

\subsection{All Rank Classifications}

There are only a few situations where a classification is available that covers all ranks, i.e., sizes of matrices.  We describe 2 such \ppmm{that have proven particularly useful}.

\subsubsection{Charge Conserving braid representations/Yang-Baxter Objects}

A classification of objects $F$ in $\MonFun(\Bcat,\Match^N)$ with $F(1)=1$ (or equivalently objects in $\YBone(\Match^N)$ - see \ref{de:YBC}) \ppmm{is given} for all $N$
in \cite{MR1X}.  
In fact there are two classifications, one for equality as the notion of equivalence, and one for the notion of equivalence described briefly below.   
The outline 
for given $N$ is as follows:
\begin{enumerate}
    \item 
    {One considers}
    equivalence classes up to the following symmetries: $R\mapsto (A\otimes A)R(A\otimes A)^{-1}$, where $A$ is a monomial matrix (see Theorem \ref{thm: outer auto match}) and $R\mapsto XRX^{-1}$ where $X$ is a diagonal $N^2\times N^2$ matrix.  
    In practice we need only consider $A$ a permutation matrix, since supplemented with the second symmetry one obtains all monomial $A$.
    \item 
    Next, we construct `enough' {varieties of} YBOs $R\in\Match^N(2,2)$, parametrised by multisets of bi-coloured composition tableaux {with the  standard left-to-right/top-to-bottom filling}. 
    These are lists of composition tableaux 
    with a total of $N$ boxes such that 1) the filling respects the natural order on $\{1,\ldots,N\}$, and 2) the rows of each tableau are either coloured or left uncoloured with the first row always uncoloured. See (\ref{pa:CCrank2}) for examples.  
    To each such bi-coloured composition tableau 
    is associated an $R\in\Match^N$ where the coloured/uncoloured rows dictate the scalars $R\mid_{\C\{\ket{ii}\}}$ and the relative position of $i\neq j$ within the list/Young diagram dictates the form of $R_{\mid_{\C\{\ket{ij},\ket{ji}\}}}$, \emph{i.e.}  there are three cases corresponding to whether  $i,j$ are in the same/different Young diagram, and same/different row.
    \item 
    Finally, it is shown that any YBO $R\in\Match^N$ is equivalent to one as constructed above, by means of the symmetries described in the second step.
\end{enumerate}

\newcommand{\kkk}{\chi} 
\newcommand{\gammad}{\gamma} 

\mdef \label{pa:CCrank2}
For example, and for use later, the complete list of  
varieties of solutions in rank $N=2$ is
derived from the labels:
\beq  \label{eq:ccN2}
\two, \qquad \oneone, \qquad 
\oneonex,  \qquad 
\square\; \square 
\hspace{.2in} 
\eq 
The usual shorthand labels for these are $F_0$, $F_f$, $F_a$ and $F_/$ respectively. 
Case $F_0$ is the (up to scalar) trivial representation.
Case $F_f$ is the representation corresponding to $U_q sl_2$ symmetry;
and $F_a$ to $U_q sl_{1|1}$. And $F_/$ is the generalised flip. 
Explicitly:
\beq   \label{eq:F/}
F_/ (\sigma) \;  =  \;  
\begin{bmatrix} 
\alpha &&& \\ && \gammad \kkk \\ & 
\frac{\gammad^{}}{\kkk}\\ &&&\beta 
\end{bmatrix}
\; =\; \mat \alpha \\ & \gamma \kkk \\ && \frac{\gammad^{}}{\kkk}\\ &&&\beta \tam
\mat 1 \\ &&1 \\ &1 \\ &&&1 \tam
\eq 

\mdef
With $\MoonFun(\Bcat,\Match^N)$ and hence $\infty$-equivalence in mind, it is worth noting that the 
all-$n$ 
representation theory is completely understood in all the above rank-2 cases - albeit that some of them are highly non-trivial. 
See e.g. \cite{JamesMurphy79,DeguchiMartin92,MartinRittenberg92} for $F_a$; \cite{Martin92} for $F_f$;
and \cite{GalindoRowell} for $F_/$.

\subsubsection{Involutive Yang-Baxter Objects}\label{sss: invol ybo}
In \cite[Theorem 4.8]{LechnerPennigWood} the objects in $\MoonFun(\Bcat,\Invol)$ are classified up to isomorphism, i.e. $\infty$-equivalence.  They show that for each $N$, the isomorphism classes of $N^2\times N^2$ involutive Yang-Baxter objects are in 1-1 correspondence with pairs of Young diagrams with a total of $N$ boxes.  The key point is that any such object corresponds to 
{a $\Sym$ representation as in \ref{de:Sym} (i.e. a functor $F:\Sym \rightarrow \Mat$)}. They also show that $YB(\Invol)$ is closed under both the $\boxplus$ and $\boxtimes$ operations.  See also \cite{Davydov1998,Davydov2000} and paragraph \eqref{2eigs} for related results.

\subsection{Bounded Rank Classifications}  \label{ss:partialclass}
Most classifications in the literature bound the rank.  A few of these are as follows.
\subsubsection{$\YBone(\Mat^2)$}\label{ss:jarmo}
Hietarinta \cite{Hietarinta1992} classified all rank-2 ($4\times 4$) solutions to the constant Yang-Baxter equation up to 
sec.\ref{s:equivandauto}]}
\\ $\bullet$ 
local isomorphism (as in \S\ref{ss:local})
\\ $\bullet$ 
scalar multiples, 
\\ $\bullet$ 
transpose
(as in (\ref{eq:transpose-equiv})), 
\\ $\bullet$ 
conjugation by the $N=2$ swap matrix $\PP=\Fff_2$  (equivalent to passing to the `other' Kronecker convention, as in \ref{de:Farr}, 
\\ $\bullet$ 
the simultaneous reflection across the diagonal and skew diagonal 
(as in (\ref{eq:transpose-equiv})). 

Hietarinta's conventions in \cite{Hietarinta1992} are distinct from ours as he describes $RP$ rather than $R$.  We reproduce the invertible, non-scalar, solutions in terms of our conventions.  
A convenient nomenclature arises if we use the term `glue' for a non-zero matrix entry where the transpose-position entry is zero (a mild abuse of notation from non-semisimple linear representation theory). 
Almost all solutions are then charge-conserving (i.e. in $\Match^2$) together possibly with some glue, and we derive our nomenclature from this property. 
We use the name coming from (\ref{eq:ccN2}), and then append a sub-name indicating glue content where appropriate. 
Thus the first solution variety below is {\em slash-type}; 
the second is DS-equivalent (via $Q=\left[ \begin {array}{cc} 0&p\\ \noalign{\medskip}q&0\end {array}
 \right] 
$) to the $k=s$ case of the first 
(Hietarinta did not use DS-equivalence);
and the third to fifth are {\em slash-glue-types}: 
\ignore{{ 
\[ \left[ \begin {array}{cccc} k&q&p&s\\ \noalign{\medskip}0&0&k&q
\\ \noalign{\medskip}0&k&0&p\\ \noalign{\medskip}0&0&0&k\end {array}
 \right] ,
 \left[ \begin {array}{cccc} k&0&0&0\\ \noalign{\medskip}0&0&p&0
\\ \noalign{\medskip}0&q&0&0\\ \noalign{\medskip}0&0&0&s\end {array}
 \right] ,
 \left[ \begin {array}{cccc} p&0&0&k\\ \noalign{\medskip}0&0&p&0
\\ \noalign{\medskip}0&q&p-q&0\\ \noalign{\medskip}0&0&0&-q\end {array}
 \right],\]\[
 \left[ \begin {array}{cccc} 0&0&0&p\\ \noalign{\medskip}0&k&0&0
\\ \noalign{\medskip}0&0&k&0\\ \noalign{\medskip}q&0&0&0\end {array}
 \right] 
,
 \left[ \begin {array}{cccc} {k}^{2}&-kp&kp&pq\\ \noalign{\medskip}0&0
&{k}^{2}&kq\\ \noalign{\medskip}0&{k}^{2}&0&-kq\\ \noalign{\medskip}0&0
&0&{k}^{2}\end {array} \right], \left[ \begin {array}{cccc} 1&0&0&1\\ \noalign{\medskip}0&0&-1&0
\\ \noalign{\medskip}0&-1&0&0\\ \noalign{\medskip}0&0&0&1\end {array}
 \right]  ,\]\[
\left[ \begin {array}{cccc} 1&0&0&1\\ \noalign{\medskip}0&1&1&0
\\ \noalign{\medskip}0&-1&1&0\\ \noalign{\medskip}-1&0&0&1\end {array}
 \right], \left[ \begin {array}{cccc} {p}^{2}+2\,pq-{q}^{2}&0&0&{p}^{2}-{q}^{2}
\\ \noalign{\medskip}0&{p}^{2}-{q}^{2}&{p}^{2}+{q}^{2}&0
\\ \noalign{\medskip}0&{p}^{2}+{q}^{2}&{p}^{2}-{q}^{2}&0
\\ \noalign{\medskip}{p}^{2}-{q}^{2}&0&0&{p}^{2}-2\,pq-{q}^{2}
\end {array} \right],\]\[
 \left[ \begin {array}{cccc} {k}^{2}&0&0&0\\ \noalign{\medskip}0&0&kq
&0\\ \noalign{\medskip}0&kp&{k}^{2}-pq&0\\ \noalign{\medskip}0&0&0&{k}
^{2}\end {array} \right] , \left[ \begin {array}{cccc} {k}^{2}&0&0&0
\\ \noalign{\medskip}0&0&kq&0\\ \noalign{\medskip}0&kp&{k}^{2}-pq&0
\\ \noalign{\medskip}0&0&0&-pq\end {array} \right] .
\]
}}

\newcommand{\Hslashglue}{
\left[ \begin {array}{cccc} k&q&p&s\\ \noalign{\medskip}0&0&k&q
\\ \noalign{\medskip}0&k&0&p\\ \noalign{\medskip}0&0&0&k\end {array}
 \right] 
 }
\newcommand{\Hslash}{
 \left[ \begin {array}{cccc} k&0&0&0\\ \noalign{\medskip}0&0&p&0
\\ \noalign{\medskip}0&q&0&0\\ \noalign{\medskip}0&0&0&s\end {array}
 \right] 
 }

\newcommand{\Haglue}{
 \left[ \begin {array}{cccc} p&0&0&k\\ \noalign{\medskip}0&0&p&0
\\ \noalign{\medskip}0&q&p-q&0\\ \noalign{\medskip}0&0&0&-q\end {array}
 \right]
 }
 
\newcommand{\HslashDS}{
 \left[ \begin {array}{cccc} 0&0&0&p\\ \noalign{\medskip}0&k&0&0
\\ \noalign{\medskip}0&0&k&0\\ \noalign{\medskip}q&0&0&0\end {array}
 \right] 
}

\newcommand{\Hslashgluex}{
 \left[ \begin {array}{cccc} {k}^{2}&-kp&kp&pq\\ \noalign{\medskip}0&0
&{k}^{2}&kq\\ \noalign{\medskip}0&{k}^{2}&0&-kq\\ \noalign{\medskip}0&0
&0&{k}^{2}\end {array} \right]
}
\newcommand{\Hslashglu}{
\left[ \begin {array}{cccc} 1&0&0&1\\ \noalign{\medskip}0&0&-1&0
\\ \noalign{\medskip}0&-1&0&0\\ \noalign{\medskip}0&0&0&1\end {array}
 \right] 
}

\newcommand{\Hising}{
\left[ \begin {array}{cccc} 1&0&0&1\\ \noalign{\medskip}0&1&1&0
\\ \noalign{\medskip}0&-1&1&0\\ \noalign{\medskip}-1&0&0&1\end {array}
 \right]
}
\newcommand{\Height}{
 \left[ \begin {array}{cccc} {p}^{2}+2\,pq-{q}^{2}&0&0&{p}^{2}-{q}^{2}
\\ \noalign{\medskip}0&{p}^{2}-{q}^{2}&{p}^{2}+{q}^{2}&0
\\ \noalign{\medskip}0&{p}^{2}+{q}^{2}&{p}^{2}-{q}^{2}&0
\\ \noalign{\medskip}{p}^{2}-{q}^{2}&0&0&{p}^{2}-2\,pq-{q}^{2}
\end {array} \right]
}

\newcommand{\Hf}{
 \left[ \begin {array}{cccc} {k}^{2}&0&0&0\\ \noalign{\medskip}0&0&kq
&0\\ \noalign{\medskip}0&kp&{k}^{2}-pq&0\\ \noalign{\medskip}0&0&0&{k}
^{2}\end {array} \right]
}
\newcommand{\Ha}{
\left[ \begin {array}{cccc} {k}^{2}&0&0&0
\\ \noalign{\medskip}0&0&kq&0\\ \noalign{\medskip}0&kp&{k}^{2}-pq&0
\\ \noalign{\medskip}0&0&0&-pq\end {array} \right]
}

\[ \hspace{-.2in}
\Hslash , \HslashDS, \Hslashglu, \Hslashglue, \Hslashgluex, 
\]
\[
\Ha, \Haglue, \hspace{.2in} \Hf, 
\]
\[
\Hising, \;\; \Height 
\]

The middle row of solutions above are the {\em a-type}, {\em a-glue-type} and {\em f-type} respectively. The final row contains the only `anomalies' from CC-with-glue - let us call them {\em Ising-type} and {\em eight-vertex-type} respectively.  

\mdef 
The all-$n$ representation theory of the charge-conserving cases has already been discussed above. We will return to the glue types and eight-vertex in a separate work.
{The braid group representations of the Ising-type solution was analysed in \cite{FRW}. }

\mdef 
Here are the Jordan forms of the \ppmm{generic cases of the} solutions above, in the order above:
\[ \hspace{-.2in} 
\mat k \\ & \sqrt{pq} \\ && \!\! -\!\sqrt{pq} \\ &&&s \tam,
\mat \sqrt{pq} \\ & k \\ && k \\ &&& \!\!-\!\sqrt{pq} \tam,
\mat 1  & 1 \\ &1 \\ &&1 \\ &&& \!\! -1   \tam ,
\mat  k&1& \\ &k&1 \\ &&k \\ &&& \!\! -k \tam ,
\mat k^2 & \\ & k^2 \\ &&k^2 \\ &&& \!\! -k^2 \tam
\]
\[
\mat k^2 \\ &k^2 \\ &&-pq \\ &&&-pq \tam , \mat  p\\&p\\&&-q\\&&&-q \tam, 
\mat k^2 \\ &k^2 \\ &&k^2 \\ &&&-pq \tam 
\]
\[
\mat 1+i \\ & 1+i \\ && 1-i \\ &&& 1-i \tam, 
\mat 2p^2  
\\ &2p^2 
\\ && -2q^2 
\\ &&& -2q^2  
\tam
\]

\begin{remark}
\begin{itemize}
    \item The first two  
        solutions are DS-equivalent when $k=s$, see Lemma \ref{le:DSequiv}.  However, the        
        second solution is indecomposable, while the 
        first is not.  
        \item The  
        slash (first), type-a (sixth) and type-f (eighth) solutions are in $YB(\Match^2)$.  
        \item the  
        seventh solution with $k=0$ has the same form as the 
        sixth solution--both are in $YB(\Match^N)$.  In fact, one may be obtained from the other by conjugating by a diagonal matrix.
        \item The 
        Ising (ninth) solution is unitary after a suitable normalization.  
        \item The  
        third and fourth solutions (two of the three slash-glue-types) are the only non-diagonalisable solutions.
    \end{itemize}
\end{remark}

\subsubsection{$YB(\PermMat^N)$, $N\leq 10$}
Objects in $YB(\PermMat^N)$ are the same as bijective set-theoretic solutions to the Yang-Baxter equation \cite{ESS}, i.e., bijective maps $r\in\Aut(X\times X)$ for some finite set $X$ so that $r_1:=r\times \id_X$ and $r_2:=\id_X\times r$ satisfy $r_1r_2r_1=r_2r_1r_2$.   It is standard to define, for each $x,y\in X$, functions $\lambda_x$ and $\rho_y$ on $X$ by $r(x,y)=(\lambda_x(y),\rho_y(x))$, and consider only those $r$ for which $\lambda_x$ and $\rho_y$ are bijections--these $r$ are called \emph{non-degenerate}.  A further reduction is to require that $r^2=\id_{X\times X}$, i.e. involutive set-theoretic solutions.  Etingof, Schedler and Soloviev \cite{ESS} give a classification of all non-degenerate involutive solutions for $N=|X|\leq 8$ up to isomorphism (local equivalences) in $YB(\PermMat^N)$ (although the count is off by 2 for $N=8$). A complete classification for of non-degenerate involutive solutions for $N=9,10$ is found in \cite{Vendraminetal22}, where they also classify the non-involutive non-degenerate solutions for $N=8$, up to isomorphism: there are 422,449,480 such classes, while the involutive non-degenerate solutions form  34,530 isomorphism classes.  
While these do not include any degenerate solutions, such as the identity, note that the involutive solutions must all fit into at most 185 $\infty$-equivalence classes, as in the \cite{LechnerPennigWood} classification.  It would be interesting to understand which of these 185 classes contain a permutation solution.

\subsubsection{$YB(\Matcha^3)$}
A classification of additive charge conserving Yang-Baxter operators of rank $3$ is described in \cite{HMR1}.  There are 6 general classes stratified by the positions of the non-zero entries of the restriction to the span of the $\ket{ij}$ with $i+j=4$, i.e. $\C\{\ket{22},\ket{31},\ket{13}\}$, up to symmetries.  
Two of these classes yield no solutions at all, 
while
the remaining 4 yield a total of 13 distinct varieties of solutions.  In this case the symmetries {used} are: 
1) conjugation by the flip $\PP$, an autoequivalence; 
2) transpose symmetry, an anti-autoequivalence 
and 
3) the order 2 autoequivalence described in Conjecture \ref{conj: Matcha outer}, i.e., corresponding to the transposition $(1\; 3)$.

\section{Discussion}

The categorical perspective 
taken here 
can be brought to bear on many other problems by relaxing some of the conditions we have imposed above. 
Here is a sampling of the situations we have in mind:
\begin{itemize}
   
    \item Our monoidal functors are always strict, and in particular strong.  There are many situations where one has \emph{lax} or \emph{oplax} monoidal functors, where $F(a)\otimes F(b)$ and $F(a+b)$ are not isomorphic.  
    For example, if $X\in\ob(\catC)$ is an object in a (say, strict) braided fusion category we can define a functor from $\Bcat$ to $\Vect$ by $F(n):=\End(X^{\otimes n})$ and 
    $F(\sigma)=c_{X,X}\in\End(X^{\otimes 2})$.  Here one has a lax monoidal structure  $J_{n,m}:F(n)\otimes F(m)\hookrightarrow F(n+m)$ where $f\otimes g\in \End(X^{\otimes n})\otimes \End(X^{\otimes m})$ is mapped to $(f\otimes \id_X^{\otimes m})(\id_X^{\otimes n}\otimes g)\in\End(X^{\otimes n+m})$.   One could imagine skeletalising $Vec$ to get a monoidal functor to $\Mat$, but the monoidal product in $\Mat$ may not be the usual one--choosing a consistent basis for each $F(n)$ is one issue.
    \item In \cite{RowellWang12} the concept of a localisation of a sequence of braid 
    {\em group} 
    representations $(\rho_n,V_n)$ is introduced, in which this sequence is related to the representations coming from some $(N,R)\in YB(\Mat)$.  The motivation comes from the problem of simulating the quantum circuits coming from braiding in topological quantum computation directly on a circuit-based model with gate set consisting of a single Yang-Baxter operator. A categorical interpretation of localisation is missing, but would likely require passing to directed system for a linearisation of $\Bcat$. I.e., we first take morphisms to be group algebras $\C[B_n]$ and fix an embedding of $\C[B_n]\hookrightarrow \C[B_{n+1}]$. Then decompose $\Bcat$ as the directed system of categories $\Bcat_n$ each with a single object $\ast_n$ and morphisms $B_n$, connected via the usual embeddings, for $j\leq k$,  $f_{j,k}:B_j\rightarrow B_{k}$ defined by $f_{j,k}(\sigma_i)=\sigma_i$.  Then we may constrain a functor $F:\Bcat\rightarrow\Mat$ to respect these embeddings, after equipping $\Mat$ with a similar decomposition as a directed system of categories.  For example, the usual block matrix embeddings of $\Mat_n$ into $\Mat_{n+1}$ by placing a $1\times 1$ identity block in the lower right would then allow families of representations such as the Burau representation.

\end{itemize}

\medskip  

The abovementioned  categorical perspective on classification has informed our exposition at every stage in this work. Of course there are alternative perspectives that suggest different approaches, and even take the classification programme in different directions. For example the statistical mechanics perspective offers numerous insights, and radical constructions, even while 
continuing to regard  
braid representations as monoidal functors (in statistical mechanical language this includes, for example, working with suitably `categorified' Andrews--Baxter--Forrester-type models rather than quantum spin chains). We will address this in a separate work.

{As a different form of restriction, one could attempt the classification of R-matrices 
(as opposed to $R^{}$-check matrices - in our terms $R\PP$ rather than  $R$) that are upper-triangular. 
Note that the target space here is smaller than the glue target space, 
as mentioned in \S\ref{ss:subcats}. }
\ppmm{Here we 
recall the  rank-3 classification due to 
Hietarinta in \cite{HietarintaUpperTri}.  All higher ranks remain open.}


\ignore{{This is a slight generalisation of Theorem \ref{thm:endoissub/quo}, which allows us to prove a version of Schur's Lemma for $YB(\Mat)$.}}

\bibliographystyle{abbrv}
\bibliography{bib/local.bib,bib/Loopy}

\end{document}